\documentclass[10pt,english,reqno]{amsart}
\usepackage{amsmath,amssymb,amsthm,amsfonts,graphicx,color}
\usepackage{amssymb,geometry}

\usepackage{color, graphics}

\usepackage{graphicx}
\usepackage[T1]{fontenc}
\usepackage[latin1]{inputenc}
\usepackage[english]{babel}
\usepackage{geometry}
\usepackage{indentfirst} % RIENTRO PRIMO PARAGRAFO

\usepackage{cases}

\usepackage[colorlinks=true]{hyperref}

\makeatletter
\def\@currentlabel{2.1}\label{e:dispaa}
\def\@currentlabel{2.21}\label{e:dispau}
\def\@currentlabel{2.22}\label{e:dispav}
\def\@currentlabel{2.23}\label{e:dispaw}
\def\@currentlabel{2.24}\label{e:dispax}
\def\theequation{\thesection.\@arabic\c@equation}
\makeatother

\usepackage[colorlinks=true]{hyperref}
\hypersetup{linkcolor=black,citecolor=black,filecolor=black,urlcolor=black} % black links, for printed output

\newcommand{\red}[1]{{\color{red} #1}}

\newcommand{\p}{{p}}

\newcommand{\h}{{\tt h}}
\newcommand{\q}{{\tt q}}
\newcommand{\pe}{{\tt p}}
\newcommand{\bi}{{\tt b}}
\newcommand{\ta}{{\tt a}}
\newcommand{\x}{{\tt x}}
\newcommand{\y}{{\tt y}}
\newcommand{\s}{{\tt s}}
\newcommand{\z}{{\tt z}}

\newcommand{\Q}{\mathbb{Q}}
\newcommand{\R}{\mathbb{R}}

\newcommand{\diver}{\text{div}}

\newcommand{\eps}{\varepsilon}

\DeclareMathOperator{\sgn}{sign}

\renewcommand{\theequation}{\thesection.\arabic{equation}}
 
 \newtheorem{lemma}{Lemma}[section]
\newtheorem{definition}{Definition}[section]
\newtheorem{theorem}{Theorem}[section]
\newtheorem{proposition}{Proposition}[section]

\newtheorem{remark}{Remark}[section]
\newcommand{\bremark}{\begin{remark} \em}
\newcommand{\eremark}{\end{remark} }

\numberwithin{equation}{section}

\title[\tiny{doubling construction for $O(m)\times O(n)$ invariant solutions to the Allen-Cahn equation}]{Doubling construction for $O(m)\times O(n)$ invariant solutions to the Allen-Cahn equation}
\author{OScar Agudelo}
\address{University of West Bohemia in Pilsen-NTIS, Univerzitn\'{i} 22, Czech Republic.}
\email {oiagudel@ntis.zcu.cz}

\author{Micha{\l } Kowalczyk}
\address{Departamento de Ingenier\'{\i}a Matem\'atica and Centro
de Modelamiento Matem\'atico (UMI 2807 CNRS), Universidad de Chile, Casilla
170 Correo 3, Santiago, Chile.}
\email {kowalczy@dim.uchile.cl}
\author{Matteo Rizzi}
\address{Centro
de Modelamiento Matem\'atico (UMI 2807 CNRS), Universidad de Chile, Casilla
170 Correo 3, Santiago, Chile.}
\email{mrizzi1988@gmail.com}
\thanks{O. Agudelo was supported by the Grant 18-032523S
of the Grant Agency of the Czech Republic and also by the Project LO1506 of the Ministry
of Education, Youth and Sports of the Czech Republic.  
M. Kowalczyk was partially supported by Chilean research grants Fondecyt 1130126 and 1170164 and CMM Conicyt PIA AFB170001 CMM-Chile.
M. Rizzi was partially supported by Fondecyt postdoctoral research grant 3170111 and CMM Conicyt PIA AFB170001 CMM-Chile.}

\begin{document}

\maketitle
\begin{abstract}
We construct new families of two-ended $O(m)\times O(n)$-invariant solutions to the Allen-Cahn equation $\Delta u+u-u^3$=0 in $\R^{N+1}$, with $N\ge 7$, whose zero level sets diverge logarithmically from the Lawson cone at infinity. The construction is based on a careful study of the Jacobi-Toda system on a given $O(m)\times O(n)$-invariant manifold, which is asymptotic to the Lawson cone at infinity.
\end{abstract}

\tableofcontents

\section{Introduction}\label{introduction}
In this work we study existence and asymptotic behaviour of bounded, sign-changing solutions to the {\it Allen-Cahn} equation
\begin{equation}
\label{Allen-Cahn-eq}
\Delta u+u-u^3=0 \quad \hbox{in} \quad \R^{N+1}.
\end{equation}

\medskip
Equation \eqref{Allen-Cahn-eq} was introduced in \cite{ALLENCAHN1979} to model the allocation of binary mixtures and it is the prototype equation for the continuous modelling of {\it phase transition phenomena}.

%\medskip
%In the gradient theory of phase transitions, the function $u$ represents a continuous realization of the phase, with values making a transition between the pure states
%$-1$ and $+1$, along an interface. The most interesting solutions in this context are therefore those in which this transition wall takes an identifiable shape.

\medskip
In the one dimensional case, the Allen-Cahn equation becomes the ordinary diffe\-rential equation (ODE)
\begin{equation}\label{Allen-Cahn-eqn 1D}
v''(t) + v(t)-v^3(t)=0 \quad \hbox{in} \quad \R.
\end{equation}

\medskip

With the boundary conditions $v(\pm \infty)=\pm 1$, the equation \eqref{Allen-Cahn-eqn 1D} has an explicit solution given by
\begin{equation}\label{Allen-Cahn-eqn 1D Sln}
v_{\star}(t) = \tanh\bigg(\frac{t}
{\sqrt{2}}\bigg) \quad \hbox{for} \quad t\in \R
\end{equation}
and up to translations, this solution is unique. Besides, $v_{\star}$ is strictly monotone increasing, i.e. $v_{\star}'(t)>0$ for every $t\in \R$. In particular, $\{v_{\star}=0\}=\{0\}$.

\medskip
Assume $N \geq 1$ and fix $\vec{a}\in \R^{N+1}$, a unit vector and $\xi_0 \in \R^{N+1}$. The function 
\begin{equation}\label{DeGiorgiconjectsln}
u(\xi):= v_{\star}(t), \quad t= \vec{a}\cdot (\xi -\xi_0)\quad \hbox{for} \quad \xi\in \R^{N+1}
\end{equation}
is a bounded and sign-changing solution to \eqref{Allen-Cahn-eq} which is  monotone in the direction of  $\vec{a}$ and whose nodal set is the hyperplane with equation $\vec{a}\cdot (\xi-\xi_0)=0$. We remark that, up to a translation and rotations of the axis, these solutions depend only on one variable and in this sense they are trivial.

\medskip
In 1978 (see \cite{DEGIORGI1978}), E. De Giorgi conjectured  that if $1 \leq N \leq 7$, then for any solution to \eqref{Allen-Cahn-eq} which is monotone in one direction, the level sets $\{u=c\}$ must be parallel hyperplanes. This is equivalent to saying that for some unit vector $\vec{a}$ and some point $\xi_0$, the solution $u$ satisfies \eqref{DeGiorgiconjectsln}.

\medskip
De Giorgi's conjecture shows also evidence of the strong connection between the study of bounded solutions to the Allen-Cahn equation and the theory of minimal hypersurfaces, see for instance \cite{MODICA1979}. 

\medskip
In \cite{AMBROSIOCABRE2000,GHOUSSOUBGUI1998}, De Giorgi's conjecture was established in dimensions $N=2,3$. In \cite{SAVIN2009}, it was proven true in dimensions $4 \leq N \leq 8$, under the additional assumption
$$
\lim_{\xi_{N+1}\to \pm \infty} u(\xi',\xi_{N+1}) = \pm 1.
$$

Further evidence of the connection between solutions to \eqref{Allen-Cahn-eq} and the theory of minimal hypersurfaces is the {\it Bernstein Conjecture}, concerning rigidity of minimal hypersurfaces, see \cite{ALMGREN1966,BERNSTEIN1917,BDG_cones, FLEMING1962,Si}.% and being in analogy with the De Giorgi's conjecture.

\medskip
In \cite{DELPINOKOWALCKZYKWEI2011} a counterexample, disproving the De Giorgi's conjecture, was built for $N=8$ using the non-trivial minimal graph $\Gamma$ built in \cite{BDG_cones}, as a counterexample to Bernstein's conjecture. The solution found in \cite{DELPINOKOWALCKZYKWEI2011} is bounded, monotone in one direction and its zero level set  is close to the dilated surface $\Gamma_{\eps} =\eps^{-1} \Gamma$, where $\eps>0$ is a small positive number. The main strategy is based upon the fact that $\Gamma_{\epsilon}$  is nearly flat around each of this points and hence the quantity $v_{\star}(t)$ is a good approximation to a solution of \eqref{Allen-Cahn-eq}, where $t=t(\xi)$ is a choice of normal coordinate (signed distance) from $\xi$ to $\Gamma_{\epsilon}$.

\medskip
In \cite{DELPINOKOWALCZYKWEI2013I}, the same approach was used  for the case $N=2$ to construct a solution to \eqref{Allen-Cahn-eq} having zero level set close to a large dilation of an embedded minimal surface with finite total curvature, that satisfies certain non-degeneracy assumptions. One important example of such surfaces is the {\it catenoid} leading to an axially symmetric solution of \eqref{Allen-Cahn-eq}.

\medskip
The aforementioned construction was generalized to the case $N \geq 3$ in \cite{AGUDELODELPINOWEI2016}, where the authors built an axially symmetric solution to \eqref{Allen-Cahn-eq} having nodal set close to a large dilation of a logarithmic correction of the higher dimensional catenoid. This logarithmic correction, which is governed by the Liouville equation, is needed due to the fact that outside a large ball, the higher dimensional catenoid is asymptotic to two parallel planes.  

\medskip
In the works mentioned above, the solutions are constructed using the infinite dimensional Lyapunov-Schmidt reduction and based on the previous knowledge of what the nodal set should be, profiting also from the properties of those minimal hypersurfaces in each particular case.

\medskip
To mention some further relevant works, let us introduce some notation. Throughout this work, we assume that $N+1:=m+n\ge 8$ and that $m,n \ge 2$. For points in $\R^{N+1}$ we write $\xi=(x,y)\in\R^m\times\R^n$. 

\medskip

For {$m,n\ge 2$}, we introduce the minimal hypersurface
$$
C_{m,n}:=
\left\{(x,y)\in\R^m\times\R^n:\,|x|^2=\frac{m-1}{n-1}|y|^2
\right\}
$$
known as the Lawson cone. In the case $m=n$, $C_{m,m}$ is known as the Simon's cone. Also, observe that the Lawson cone is invariant under the action of the group of rotations $O(m)\times O(n)$.

\medskip
In \cite{CABRETERRA2009,CABRETERRA2010} existence and qualitative properties of saddle-shaped solutions to \eqref{Allen-Cahn-eq} are studied. The nodal set of this solutions is exactly the Simons cone $C_{m,m}$.

\medskip
We remark that for $m,n \geq 2$, the opent set $\R^{N+1}\backslash C_{m,n}$ has two connected components, namely
$$
\begin{aligned}
E_{m,n}^+ &:=\bigg\{(x,y)\in\R^m\times\R^n:\,|x|^2<\frac{m-1}{n-1}|y|^2\bigg\}, \\ 
E_{m,n}^-& :=\bigg\{(x,y)\in\R^m\times\R^n:\,|x|^2>\frac{m-1}{n-1}|y|^2\bigg\},
\end{aligned}
$$
corresponding to the interior and the exterior of $C_{m,n}$, respectively. The sets $E_{m,n}^\pm$ help to describe an important feature of $C_{m,n}$ that has been already studied in \cite{HS,M,SS} and that is described in the next result.

\medskip 
\begin{theorem}\cite{M,SS}
Let $m,\,n\ge 3$, $n+m=N+1\ge 8$. Then there exist two unique minimal hypersurfaces $\Sigma^\pm_{m,n}\subset E^\pm_{m,n}$ which are asymptotic to $C_{m,n}$ at infinity and $d(\Sigma^\pm_{m,n},\{0\})=1$. Moreover, $\Sigma^\pm_{m,n}$ are $O(m)\times O(n)$-invariant.% $\{(x,y)\in\R^m\times\R^n:\,x_i=0, \,1\le i\le m\}$ and $\{(x,y)\in\R^m\times\R^n:\,y_j=0, \,1\le j\le n\}$.
\label{th_min-surf_cone}
\end{theorem}

\medskip
We are interested in solutions $u$ to \eqref{Allen-Cahn-eq} which are invariant under the action of the group of rotations $O(m)\times O(n)$. In this regard, the saddle-shaped solutions in \cite{CABRETERRA2009, CABRETERRA2010} enjoy this symmetry. We stress that functions with such symmetry are even in each of the variables.

\medskip
Let $\Sigma$ be one of the minimal hypersurfaces $\Sigma_{m,n}^{\pm}$. In \cite{PW} the authors construct stable, $O(m)\times O(n)-$invariant solutions to \eqref{Allen-Cahn-eq}, changing sign once and having nodal set close to a large dilation of $\Sigma$. Their construction follows the approach from \cite{DELPINOKOWALCKZYKWEI2011} using extensively the area-minimising character of the underlying cone. Also, from the results in \cite{HS}, this contruction can be generalized to more general minimal hypersurfaces asymptotic to an area-minimising cone.

\medskip 
In this work, we generalise  this construction to built solutions to \eqref{Allen-Cahn-eq}, but changing sign twice near a large dilation of $\Sigma$. 

\medskip
The first step in this generalisation is the following theorem, which is our first main result in this work.

%%\begin{theorem}
%Let $m,\,n\ge 3$, $n+m=N+1\ge 8$. Let $\Sigma$ be one of the two minimal hypersurfaces constructed in Theorem \ref{th_min-surf_cone}. Then $\Sigma$ has exactly two linearly independent Jacobi fields in the class of functions which are $O(m)\times O(n)$-invariant and decaying at infinity.
%\label{th_Jacobi}
%%\end{theorem}

\begin{theorem}
Let {$m,\,n\ge 3$}, $n+m=N+1\ge 8$  and let $a_{\star}>0$ be a constant. Let $\Sigma$ be one of the two minimal hypersurfaces constructed in Theorem \ref{th_min-surf_cone}, then there exists $\delta_*>0$ small such that if $0<\delta \leq \delta_*$, then the equation
\begin{equation}
\label{eq_Liouville}
\delta \big(\Delta_\Sigma w+|A_\Sigma|^2 w\big)=  2a_{\star}e^{-\sqrt{2}w}
\end{equation}
has a smooth solution which is $O(m)\times O(n)$-invariant.
\label{th_Liouville}
\end{theorem}

Above $\Delta_\Sigma$ is the Laplace-Beltrami operator on $\Sigma$,  $|A_\Sigma|$ is the norm of the second fundamental form and $\Delta_\Sigma+|A_\Sigma|^2$ is the Jacobi operator. The nonlinear, exponential term on the right hand side of (\ref{eq_Liouville}) describes the Toda interaction between the two sheets of the zero level set of $u$,   hence the name the Jacobi-Toda equation given to (\ref{eq_Liouville}). Before  we will describe the role it plays in the problem at hand in more details  we will first explain why it is an important problem on its own. 

Geometric analogs of the existence result for  \eqref{Allen-Cahn-eq} proven in this paper are doubling construction for minimal surfaces (e.g.  \cite{kapouleas2010doubling}, \cite{kapouleas2014minimal}, \cite{KAPOULEASYANG2010}) and connected sum construction for CMC (constant mean curvature) surfaces (e.g. \cite{maz_pac_pol}).   Both are based on a similar idea of taking two copies of a given minimal or CMC surface and building a new, connected surface of the same type by inserting a catenoidal bridge between them. In general this requires also a deformation of  the original surfaces. Likewise,  in our case we want to  "double" the zero level set $\Sigma$ of the solution $u$ of  \eqref{Allen-Cahn-eq} and the Jacobi-Toda equation provides  the "connection" between  the two components. In this context a general principle would be: if for a given minimal surface one can solve the Jacobi-Toda equation (\ref{eq_Liouville}) then a connected sum  construction for \eqref{Allen-Cahn-eq} based on such surface should be possible (e.g. \cite{DELPINOKOWALCZYKWEIYANG2010,DELPINOKOWALCZYKWEI2013II}). As recent results in \cite{2018arXiv180302716C} and \cite{doi:10.1002/cpa.21812} show that  the Jacobi-Toda equation also  plays a crucial role in the problem of classification of finite Morse index solutions of the Allen-Cahn equation.

We  state now the second  result in this work, concerned directly with the existence of  solutions to \eqref{Allen-Cahn-eq}.

\begin{theorem}\label{2ndMainTheo}
Let {$m,\,n\ge 3$}, $n+m=N+1\ge 8$ and let $\Sigma$ be one of the two minimal hypersurfaces
described in Theorem \ref{th_min-surf_cone}. Then there exists $\eps_0>0$ such that, for any $\eps\in(0,\eps_0)$, there exists a solution $u_\eps$ to \eqref{Allen-Cahn-eq} in $\R^{N+1}$ such that 
\begin{enumerate}
\item[(i)] $u_\eps$ is smooth and $O(m)\times O(n)$-invariant;
%\item The zero level set of $u_\eps$ has $k>1$ connected components, each of which is asymptotic, as $\eps\to 0$, to a normal graph over $\Sigma_\eps=\eps^{-1}\Sigma$.

\medskip
\item[(ii)] the zero level set of $u_\eps$ is the disjoint union of $2$ connected components, which are normal graphs over $\Sigma_\eps:=\eps^{-1}\Sigma$ of $O(m)\times O(n)$-invariant functions;

\medskip
\item[(iii)] \label{est_energy_grad} there exists a constant $c>0$ such that for any $\eps \in (0,\eps_0)$ and any $R> 2\eps^{-1}$,
\begin{equation}\notag
\int_{B_R}\frac{1}{2} |\nabla u_\eps|^2+\frac{1}{4}(1-u_\eps^2)^2\le c R^N.
\end{equation}
\end{enumerate}
\end{theorem}

\medskip
A few comments are now in order. First, from Theorem \ref{2ndMainTheo} and taking $m,n \geq 3$, there are two associated minimal hypersurfaces $\Sigma^{\pm}_{m,n}$, each of which giving rise to a family of solutions. Since the nonlinearity is odd, if $u$ is a solution then also $-u$ is also a solution. Thus, we actually have $4$ families of solutions.

\medskip
Let $\Sigma$ be as in Theorem \ref{2ndMainTheo}. We remark also that the nodal set of the solution is governed by a Jacobi-Toda system associated to $\Sigma$, namely:
\begin{equation}\label{JacTodaSystIntro}
\begin{aligned}
\eps^2 \big(\Delta_{\Sigma} \h_1 + |A_{\Sigma}|^2 \h_1\big) - a_\star e^{-\sqrt{2}(\h_2 -\h_1)}&=0\\
\eps^2 \big(\Delta_{\Sigma} \h_2 + |A_{\Sigma}|^2 \h_2\big) + a_\star e^{-\sqrt{2}(\h_2 -\h_1)}&=0
\end{aligned}
\quad \hbox{in} \quad \Sigma,
\end{equation}
where $\Delta_{\Sigma}$ and $|A_{\Sigma}|$ are the {\it Laplace-Beltrami} operator and the norm of the second fundamental form on $\Sigma$, respectively,   and $a_{\star}>0$ is a constant depending only on the function $v_{\star}$ described in \eqref{Allen-Cahn-eqn 1D Sln}.

\medskip 
As we will see, the system \eqref{JacTodaSystIntro} is decoupled into the system  
\begin{equation}\label{JacTodaSystIntro2}
\begin{aligned}
\Delta_{\Sigma} {\tt v}_{0,1} + |A_{\Sigma}|^2 {\tt v}_{0,1}&=0\\
\eps^2 \big(\Delta_{\Sigma} {\tt v}_{0,2}+ |A_{\Sigma}|^2 {\tt v}_{0,2} \big) - 2a_\star e^{-\sqrt{2}{\tt v}_{0,2}}&=0
\end{aligned}
\quad \hbox{in} \quad \Sigma,
\end{equation}
that can be solved using  nondegeneracy  of $\Sigma$ (see Proposition \ref{prop_Jacobi_lin} below) and Theorem \ref{th_Liouville} with $\delta=\eps^{2}$.

\medskip

The proof of Theorem \ref{2ndMainTheo} is based on the {\it infinite dimensional Lyapunov-Schmidth reduction} technique and as we will see, we can be more precise regarding the asymptotic behavior of the solution $u_{\eps}$. In particular the energy growth estimate follows from this.

\medskip
To explain the previous paragraph, let $\eps\in (0,\eps_0)$ and let $\nu_{\Sigma}$ be the choice of the unit normal vector to $\Sigma$ pointing towards the hyperplane $\{0\}\times \R^n$. Let also $\Sigma_{\eps}:=\eps^{-1}\Sigma$ be a large dilation of $\Sigma$. Consider a tubular neighbourhood of $\Sigma_{\eps}$ of the form
$$
\mathcal{N}_\eps:=\{\pe + \z \nu_{\Sigma}(\eps \pe)\,:\, \pe \in \Sigma_{\eps}, \quad  |\z|<\eps^{-1}\eta+ c|\pe|\}
$$
for some $\eta>0$ small.

\medskip
Observe that for any  $\xi=\pe + \z \nu_{\Sigma}(\eps \pe)\in \mathcal{N}_\eps$, $|\z|={\rm dist}(\xi, \Sigma_{\eps})$ and the solution $u_{\eps}$ satisfies that
$$
u_{\eps}(\xi) \sim v_{\star}(\z - \h_1(\eps \pe)) - v_{\star}(\z - h_2(\eps \pe)) -1 
$$
while for $\xi$ far from the set $\mathcal{N}$, $u_{\eps}(\xi)\sim -1$ at an exponential rate in $|\z|$.

\medskip
Similar constructions have already been carried out for the equation \eqref{Allen-Cahn-eq} under different geometric settings. We mention for instance  \cite{DELPINOKOWALCZYKWEIPACARD2010}, where the authors build solutions to \eqref{Allen-Cahn-eq} in $\R^2$, having multiple ends governed at main order by a one-dimensional {\it Toda System}.

\medskip
In \cite{AGUDELODELPINOWEI2015} similar techniques as in \cite{DELPINOKOWALCZYKWEIPACARD2010} were used to construct solutions to \eqref{Allen-Cahn-eq} in $\R^3$ whose nodal set, outside a large ball, has multiple catenoidal like components. These components are governed by the Jacobi-Toda system associated to the catenoid.

\medskip
Our developments are in the spirit of the construction of sign-changing solutions for the travelling wave problem for the Allen-Chan equation, carried out in \cite{DELPINOKOWALCZYKWEI2013II}, within the context of hypersurfaces with constant mean curvature, where the nodal set of the solutions is also governed by a Jacobi-Toda type system.

\medskip
Finally, part {\it (iii)} in Theorem \ref{2ndMainTheo} implies that for any $\eps \in (0,\eps_0)$, 
$$
\limsup \limits_{R \to \infty} \frac{1}{R^N}\int_{B_R}\frac{1}{2}|\nabla u_\eps|^2 +\frac{1}{4} (1-u^2_\eps)^2< \infty
$$
and this suggests that these solutions should have {\it finite Morse index}, see \cite{{AMBROSIOCABRE2000},{GHOUSSOUBGUI1998}}.

\medskip
The rest of the paper is organized as follows: Section \ref{sec_invariant_minsurf} discusses the topic of minimal hypersurfaces, minimal cones and area-minimising hypersurfaces. It also contains a detailed discussion on nondegeneracy properties of the hypersurfaces $\Sigma_{m,n}^{\pm}$. Section \ref{JacToda Section} discusses the proof of Theorem \ref{th_Liouville} and the solution of system \eqref{JacTodaSystIntro}. Section \ref{approxslnAC} presents the approximate solution to \eqref{Allen-Cahn-eq} and gives the preliminary sketch of the proof of Theorem \ref{2ndMainTheo}. In section \ref{LSreduction} we present the Lyapunov-Schmidth reduction and we finish the proof of part {\it (i)} and {\it (ii)} in Theorem \ref{2ndMainTheo}. Finally, in Section \ref{sec_energy_ball}, we prove the energy estimate of the gradient in part {\it (iii)} of Theorem \ref{2ndMainTheo}.

\section{$O(m)\times O(n)$-invariant minimal hypersurfaces}\label{sec_invariant_minsurf}

%The study of \textit{Minimal hypersurfaces} has been of great interest in several fields of mathematics. In the context of the {\it Calculus of Variations}, these hypersurfaces arise as critical points of the area functional. In differential geometry, they are zeroes of the {\it mean curvature operator}. In turn, the mean curvature operator characterises the first varia\-tion of the area functional, exposing one instance of the deep connection between these two fields. 

\subsection{The Jacobi operator and stability}\label{subsec_stability}

A hypersurface having zero mean curvature is not necessarily a minimiser of the area functional. One way to study the stability properties of a minimal hypersurface is through the study of the second variation of the area functional around the hypersurface, whenever the area functional is smooth enough.

\medskip
To be more precise, let $\Sigma\subset \R^N$ be a hypersurface with singular set ${\rm sing}(\Sigma)$ and normal vector $\nu:\Sigma - {\rm sing}(\Sigma)\to \R^{N}$. For any ${\tt v}\in C^{\infty}_c\big(\Sigma \backslash {\rm sing}(\Sigma)\big)$, consider the normal graph
$$
\Sigma_{{\tt v}}:=\{{\pe}+{\tt v}(\pe)\nu_\Sigma(\pe):\, \pe\in\Sigma\}.
$$

It is well known that $\Sigma$ is a minimal hypersurface if it is a critical point (in some  appropriate topology) of the functional
\begin{equation}\label{areafunctional}
C^{\infty}_c\big(\Sigma  \backslash{\rm sing}(\Sigma)\big) \ni {\tt v}\mapsto \int_{\Sigma_{{\tt v}}}\, 1 \,d\sigma,
\end{equation}which is equivalent to saying that ${\tt v}=0$ is a zero for the mean curvature operator,
\begin{equation}\label{meancurvatureop}
C^{\infty}_c\big(\Sigma \backslash {\rm sing}(\Sigma)\big) \ni {\tt v} \mapsto H_{\Sigma_{{\tt v}}}.
\end{equation}

The second variation of \eqref{areafunctional} and the first variation of \eqref{meancurvatureop} at ${\tt v}=0$ give rise to the quadratic form 
$$
{\tt v}\in C^{\infty}_c(\Sigma \backslash {\rm sing}(\Sigma))\mapsto \int_{\Sigma}\bigg(|\nabla_{\Sigma} {\tt v}|^2 - |A_{\Sigma}|^2 {\tt v}^2 \bigg) d\sigma,
$$
which is characterised by the {\it Jacobi operator} of $\Sigma$, 
\begin{equation}\label{JacobiOp}
J_\Sigma:=\Delta_\Sigma+|A_\Sigma|^2,
\end{equation}
where $\Delta_{\Sigma}$ is the {\it Laplace-Beltrami} operator and $|A_{\Sigma}|$ is the norm of the second fundamental form of $\Sigma$.

\medskip
Stability properties of $\Sigma$ can be studied with the help of the operator $J_{\Sigma}$. We say that $\Sigma$ is {\it stable} if for every ${\tt v}\in C^\infty_c(\Sigma\backslash Sing(\Sigma))$,
\begin{equation}
\label{def_stability_ineq}
\int_\Sigma (|\nabla_\Sigma {\tt v}|^2-|A_\Sigma|^2 {\tt v}^2)d\sigma\ge 0.
\end{equation}

We also say that $\Sigma$ is {\it strictly stable} if the inequality in (\ref{def_stability_ineq}) is strict.

\medskip
Observe that the minimal hypersurface $\Sigma$ is stable if and only if the eigenvalues of $-J_\Sigma$ are non negative, while it is strictly stable if and only if all these eigenvalues are positive.

\medskip
Next, we consider the case when $\Sigma$ is a minimal cone, which we denote by $C$ (saving $\Sigma$ for later porpuses). We will also assume that $C$ is regular, i.e. ${\rm sing}(C)$ consists only on one point which is assumed to be the origin.

\medskip
Next, we analyse the stability properties of $C$. Let $B_\rho$ denote the ball in $\R^N$ and $S^{N-1}_{\rho}$ be the sphere in $\R^N$ both of them of radius $\rho>0$ and centered at the origin. 

\medskip
Set $\Lambda:=C\cap \partial B_1$ so that $\Lambda$ is a minimal submanifold of $\partial B_1$ and  
 \begin{equation}
\label{reg_cone}
C=\{r\pe:\, \pe\in\Lambda,\, r>0\}.
\end{equation}

In the $(r,\pe)$-coordinates, the Jacobi operator $J_C= \Delta_{C} + |A_C|^2$ takes the form
\begin{equation}\label{Jacobi_cone}
J_C=\partial^2_r+\frac{N-1}{r}\partial_r+\frac{1}{r^2} J_\Lambda,
\end{equation}
where $J_\Lambda$ corresponds to the Jacobi operator of $\Lambda$. 

\medskip
Let 
$$
\lambda_1 < \lambda_2 \leq \lambda_3 \leq \cdots \leq \lambda_k \leq \cdots
$$ 
be the sequences of eigenvalues of $-J_{\Lambda}$, counting multiplicities, and let $\{\varphi_j\}_{j\ge 1}$ be an orthonormal basis for $L^2(\Lambda)$ with $\varphi_j$ being an eigenfunction associated to $\lambda_j$.

\medskip
For any $\phi\in L^2(C)$, we perform the {\it Fourier decomposition},
$$
\phi(r\pe)=\sum_{j=1}^\infty \phi_j(r)\varphi_j(\pe)\quad \hbox{with}\quad \phi_j(r)=\int_\Lambda \phi(r\pe)\varphi_j(\pe)d\sigma(\pe)
$$
for all $r>0$.

\medskip
Therefore, $\phi$ is an eigenfunction of $-J_C$ with eigenvalue $\mu\in \R$ if and only if for every $j\in \mathbb{N}$,
\begin{equation}\label{EqnFourierdecomp}
-\Delta\phi_j+\frac{\lambda_j}{|\mathbf{y}|^2}\phi_j=\mu\phi_j \quad \hbox{in} \quad \R^{N}.
\end{equation}

Multiply \eqref{EqnFourierdecomp} by $\phi_j$ and integrate by parts, using that for $j>1$, $\lambda_j>\lambda_1$ and the Hardy inequality (see \cite{CP}, Proposition $1.20$) to find that
$$
\mu\int_{\R^N}\phi_j^2 d\mathbf{y}\ge \left(\left(\frac{N-2}{2}\right)^2+\lambda_1\right)\int_{\R^N}\frac{\phi_j^2}{|\mathbf{y}|^2}d\mathbf{y},
$$
which yields that the cone $C$ is stable if 
\begin{equation}
\label{def_stable_cone}
\lambda_1\ge -\left(\frac{N-2}{2}\right)^2
\end{equation}
and strictly stable if 
\begin{equation}
\label{def_strictly_stable_cone}
\lambda_1> -\left(\frac{N-2}{2}\right)^2.
\end{equation}

We now focus in the particular instance when  $C$ is a {\it Lawson cone}. To be more precise, let $m,n\geq 2$ and such that $N=m+n -1$. Recall that the Lawson cone is defined as 
\begin{equation}
\label{Lawson_cones}
C_{m,n}:=\left\{(x,y)\in\R^m\times\R^n:\,|x|^2=\frac{m-1}{n-1}|y|^2\right\}.
\end{equation}

In this case, 
$$
\Lambda=\Lambda_{m,n}:=S^{m-1}_{\rho_m}\times S^{n-1}_{\rho_n},
$$
where $\rho_m:=\sqrt{\frac{m-1}{N-1}}$ and $\rho_n:=\sqrt{\frac{n-1}{N-1}}$.

\medskip
Since $|A_{\Lambda_{m,n}}|^2= N-1$, the Jacobi operator of $\Lambda_{{m,n}}$ takes the form 
$$
J_{\Lambda_{m,n}}=\Delta_{\Lambda_{m,n}}+(N-1)$$ 
and the first eigenvalue of $-J_{\Lambda_{m,n}}$ is $\lambda_1=-(N-1)<0$. Thus, \eqref{def_strictly_stable_cone} translates into $$
\frac{(N-2)^2}{4}-(N-1)>0
$$
or equivalently $N\ge 7$.

\medskip
Therefore, the Lawson cone $C_{m,n}$ is stable whenever $n+m\ge 8$, $n,m\ge 2$ (see \cite{PW}) and its Jacobi operator $J_{C_{m,n}}$ reads as
\begin{equation}
\label{Jacobi_Cmn}
J_{C_{m,n}}=\partial^2_r+\frac{N-1}{r}\partial_r+\frac{1}{r^2}\left(\Delta_{\Lambda_{m,n}}+(N-1)\right).
\end{equation}

\medskip

\subsection{Area-minimising hypersurfaces} \label{subsec_area-minimising}

Another property of the Lawson cone $C_{m,n}$ that concerns with its stability is that $C_{m,n}$ is not only a minimal hypersurface, but actually it is an area minimising hypersurface. Next, we discuss this topic in more detail. 

\medskip
First, define the perimeter of a subset $E\subset\R^{N+1}$ in an open set $\Omega\subset\R^{N+1}$ as
\begin{equation}\label{Perimeter}
{\rm Per}(E,\Omega):=\sup\left\{\int_E \diver X \,d\xi:\, X\in C^\infty_c(\Omega,\R^{N+1})\right\}.
\end{equation}

If $E$ has smooth boundary $\partial E$, it follows from the Divergence Theorem that ${\rm Per}(E,\Omega)$ coincides with the $N$-dimensional {\it Hausdorff measure} of $\partial E\cap\Omega$. However, the definition in \eqref{Perimeter} allows us to treat the case of sets $E$ with non-smooth boundary $\partial E$, like the case when $\partial E =C_{m,n}$.

\medskip 
 
Up to a translation, there is no loss of generality in assuming that $0\in\partial E$.  Next, we define the concept of area minimising hypersurface. 

\begin{definition}
\label{def_area-minimising}
We say that $\Sigma:=\partial E$ is an area-minimising hypersurface if for any $\rho>0$ and for any smooth set $F\subset\R^{N+1}$ such that $F\backslash B_\rho=E\backslash B_\rho$,
$$
{\rm Per}(E,B_{2\rho})\le {\rm Per}(F,B_{2\rho}).
$$
\end{definition}

This definition is equivalent to say that $\Sigma$ is a global minimiser of the area functional. Therefore, any area-minimising hypersurface is a minimal hypersurface.

\medskip
We refer interested readers to \cite{CP,CF} and references therein, for a deeper understanding of the Definition \ref{def_area-minimising} and related topics. %However, it is important to notice that the definitions of area-minimising hypersurface therein are slightly different from ours, since Cozzi and Figalli \cite{CF} are basically interested in the Plateau problem in a bounded set, thus they just fix $\rho=1$, while Cabr\'{e} and Poggesi give the definition just for regular sets $E$, thus they do not need to compare the perimeters in a larger ball.

\medskip
A relevant question concerns with the regularity of area-minimising hypersurfaces and, in general, of minimal hypersurfaces. In \cite{Si}, Simons proved that if $N\le 6$, $N$-dimensional area-minimising hypersurfaces are smooth. In higher dimension, area-minimising cones are known to exist (see \cite{D}). This is the essence of the following theorem.

\begin{theorem}\label{th_min_cones}{\cite{D}}
\begin{enumerate}
\item[(i)] If $N=m+n-1> 7$ with $n,m\ge 2$, then the Lawson cone $C_{m,n}$ is area-minimising.
 
 \medskip
\item [(ii)]If $N=m+n-1=7$, then the Lawson cone $C_{m,n}$ has zero mean curvature everywhere except in the origin, which is singular. Moreover, it is area-minimising if and only if $|m-n|\le 2$. 
\end{enumerate}
\end{theorem}

Theorem \ref{th_min_cones} provides an example of non smooth $N$-dimensional area-minimising hypersurfaces, at least of dimension $N\ge 7$. The first result in this direction was obtained in \cite{BDG_cones} for the case $n=m\ge 4$. The case $m+n>8$ was treated in \cite{La}, while in \cite{Sim} it is proven that $C_{3,5}$ is area minimising and that $C_{2,6}$ has zero mean curvature, but it is not a global minimiser of the area.

\medskip
We finish this discussion with the notion of \textit{strictly} area-minimising cones.

\begin{definition}
Let $C\subset \R^{N+1}$ be a cone and set $\Lambda:=C\cap  B_1$. Then $C$ is said to be strictly area-minimising if there exist constants $\theta>0$ and $\eps_0>0$ such that
for any $\eps\in (0,\eps_0)$ and for any smooth hypersurface $\Gamma\subset\R^{N+1}\backslash B_\eps$ such that $\Gamma\backslash B_1=C\backslash B_1$
$$
\mathcal{H}_N(\Lambda)\le \mathcal{H}_N(\Gamma\cap B_1)-\theta\eps^N,
$$ 
where $\mathcal{H}_N$ stands for the $N-$dimensional Hausdorff measure in $\R^{N+1}$.
\end{definition}
For instance, if $n,m\ge 2$ and $N=m+n-1> 7$, then $C_{m,n}$ is strictly area minimising. The same is true if $n,m\ge 3$ and $N=m+n-1=7$. This agrees with the fact that in this case $C_{m,n}$ is strictly stable.

\subsection{Minimal hypersurfaces asymptotic to a cone}\label{subsec_as-min-surf}

Next, we discuss the existence and asymptotic behaviour of smooth minimal hypersurfaces that are asymptotic to the cone $C_{m,n}$, $m,n\ge 3$, $m+n\ge 8$. In this work, these hypersurfaces are the core  of the construction of sign changing solutions to \eqref{Allen-Cahn-eq}.

\medskip
In what follows we make extensive use of some of the symmetries of the cone $C_{m,n}$. To be more precise, consider the group $O(m)\times O(n)$. From \eqref{Lawson_cones} it is clear that $C_{m,n}$ is invariant under the action of this group.  

\medskip
A function ${\tt v}:C_{m,n}\to\R$ is invariant under the action of $O(m)\times O(n)$ if and only if there exists $v:(0,\infty)\to\R$ such that 
$$
{\tt v}(r\pe)=v(r) \quad \hbox{for all}\quad  r>0, \quad \pe\in\Lambda_{m,n}.
$$

We are next interested in the solutions of the homogeneous equation 
$$
J_{C_{m,n}}{\tt v}=0 \quad \hbox{in} \quad C_{m,n},
$$
also known as {\it Jacobi fields} of $C_{m,n}$. In particular, we are interested in the $O(m)\times O(n)-$invariant Jacobi fields.

\medskip

It is straightforward to verify that the only two $O(m)\times O(n)-$invariant Jacobi fields of $C_{m,n}$ are  
$$
{\tt u}_{\pm}(r\pe)=u_{\pm}(r)=r^{\gamma^\pm}, 
$$
where $\gamma^\pm$ are the roots (usually referred to as \textit{indicial roots}) of
\begin{equation}\notag
\gamma^2+\gamma(N-2)+(N-1)=0,
\end{equation}
that is
$$
\gamma^\pm=-\frac{N-2}{2}\pm\sqrt{\left(\frac{N-2}{2}\right)^2-(N-1)}.
$$

Observe that $\gamma_{\pm}\in (-\infty,0)$ if and only if $N\ge 7$.

\medskip
Let $E^\pm$ denote the two connected components of $\R^{N+1}\backslash C_{m,n}$, where $E^-$ is the component containing the hyperplane $\{(x,y)\in\R^m\times\R^n:\, y=0\}$. 

\medskip

The following result summarises the discussion in this section. 

\begin{theorem}\label{th_Al}{\cite{ABPRS,HS,M,SS}}

Let $m,n\ge 3$, $m+n\ge 8$. Then there exist two unique minimal hypersurfaces $\Sigma^\pm_{m,n}\subset E^\pm$ satisfying that
\begin{enumerate} 
\item[(i)] $\Sigma^\pm_{m,n}$ are smooth;

\medskip
\item[(ii)] ${\rm dist}(\Sigma^{\pm}_{m,n},\{0\})=1$;

\medskip
\item[(iii)] for any $\xi\in E^\pm$, the ray $\{\lambda\xi:\,\lambda>0\}$ intersects $\Sigma^\pm_{m,n}$ exactly once;

\medskip
\item[(iv)]  $\Sigma^\pm_{m,n}$ are $O(m)\times O(n)$-invariant;

\medskip
\item[(v)] there exist constants $R_{\pm}=R_{\pm}(C_{m,n})>0$, $c>0$ and $O(m)\times O(n)-$invariant functions ${\tt w^\pm}: C_{m,n}\backslash B_{R_{\pm}} \to \R$ such that
$$
{\tt w}^+>0,\quad{\tt w}^-<0 \quad \hbox{in}\quad C_{m,n}\backslash B_{R_{\pm}}\quad \hbox{respectively};
$$
for any $\pe \in C_{m,n}\backslash B_{R_{\pm}}$
\begin{equation}\label{rate-dec-Sigma}
{\tt w}^\pm(\pe)=c|\pe|^{\gamma^+}(1+o(1)) \quad \hbox{as}\quad |\pe|\to\infty
\end{equation}
and 
$$
\Sigma^\pm_{m,n}=\{\pe + {\tt w}^\pm(\pe) \nu_{C_{m,n}}(\pe)\,:\, \pe\in C_{m,n}\backslash B_{R_{\pm}}\},
$$ 
where $\nu_{C_{m,n}} :C_{m,n}\backslash \{0\} \to \R^{N+1}$ is the choice of the normal vector to $C_{m,n}$ pointing towards $E^+$.
\end{enumerate}

%$\Sigma^+_{m,n}$ satisfies a similar statement as in ${\rm (v)}$, but with $\nu_{C_{m,n}}$ pointing towards $E^+$.
\end{theorem}

\medskip
\begin{proof}
%Existence and uniqueness of two minimal surfaces $\Sigma^{\pm}_{m,n}$ asymptotic to the cone $C_{m,n}$ fulfilling (ii) follow from Theorem 1 in \cite{M} or equivalently, from Theorem 0.3 in \cite{SS}. 
In Theorem 2.1 from \cite{HS}, given an area minimising cone $C$, the authors prove the existence of two unique smooth \textit{area minimising} hypersurfaces $\Sigma^\pm_{m,n}\subset E^\pm$ with $d(\Sigma^\pm_{m,n},\{0\})=1$ and asymptotic to $C_{m,n}$ in the sense that, outside a ball, they are normal graphs over $C_{m,n}$ of functions ${\tt w}^+>0$ and ${\tt w}^-<0$ respectively. It is also proven that the scaling $\lambda \Sigma^\pm_{m,n}$, $\lambda >0$, foliates $E^\pm$ respectively. The decay rate of these graphs is given by Theorem 3.2 in \cite{HS}, provided $C$ is strictly area minimising, which is the case if $C=C_{n,m}$ for such $m+n=8$ and $m,n\ge 3$ or $n+m\ge 9$ and $m,n\ge 2$.

%\medskip
%From the uniqueness of $\Sigma^{\pm}_{m,n}$ we can conclude that $\Sigma^{\pm}_{m,n}=S^{\pm}$ which together with the above discussion proves (i), (ii), (iii) and (v).

\medskip
As for (iv), in \cite{ABPRS} the authors prove the existence of two $O(m)\times O(n)-$invariant stable minimal hypersurfaces $\Gamma^\pm_{m,n}\subset E^\pm$  which are asymptotic to $C_{m,n}$ at infinity and that satisfy (i),(ii) and (iii), for $n+m\ge 8$, $m,n\ge 3$. By the uniqueness result in \cite{M}, we find that $\Sigma^{\pm}_{m,n}=\Gamma^\pm_{m,n}$, thus for such $m,n$, (iv) is satisfied too. 
\end{proof}

\begin{remark}
{\rm The restrictions about $m$ and $n$ are crucial in the proof of Theorem \ref{th_Al}, since Theorems 2.1 and 3.2 from \cite{HS} relies on the strict minimality of the cones. In dimension $N=m+n-1\le 6$ there exist no area-minimising cones.}
\end{remark}

\subsection{The Jacobi operator on an $O(m)\times O(n)$-invariant minimal hypersurface}\label{subsec_Jacobi-invariant}

%and make the important remark that all the developments in this work, from this point and on, can be performed in an analogous fashion for the case $\Sigma=\Sigma^{+}_{m,n}$, with the corresponding obvious changes. 

\medskip

%For points $(x,y)\in\R^m\times\R^n$ we introduce introduce the variables
%\begin{equation}
%\label{O(m)O(n)inv_change_variables}
%a=|x|,\qquad
%b=|y|
%\end{equation}

%\medskip
In what follows we set $\Sigma:=\Sigma^-_{m,n}$. This represents no significant restriction in our developments since $\Sigma^+_{m,n}=\sigma^{-1}(\Sigma^-_{n,m})$, where $\sigma(x,y)=(y,x)$. In particular, if a family of solutions $u_\eps$ to the Allen-Cahn equation satisfying the properties of Theorem \ref{2ndMainTheo} with $\Sigma=\Sigma^-_{n,m}$ exists, then the family $v_\eps:=u_\eps\circ\sigma$ will enjoy the same properties with $\Sigma=\Sigma^+_{m,n}$. 

%In what follows we let $\Sigma$ be one of the minimal hypersurface $\Sigma_{m,n}^{-}$, described in Theorem \ref{th_Al}. 

\medskip
Similar to $C_{m,n}$, the set $\R^{N+1}\backslash \Sigma$ has two connected components one, which we denote, abusing the notation, by $E^{\pm}$. We make the convention that $E^{+}$ is the connected component containing the hyperplane $\{0\}\times \R^{n}$.

\medskip
The $O(m)\times O(n)-$invariance of $\Sigma$ implies that $\Sigma$ is generated by a smooth, regular curve $\Upsilon:\R \to \R^2$, $\Upsilon(s):=(a(s),b(s))$ in the half-plane
$$
Q:= \{(a,b)\in \R^2\,:\, a>0\} %\quad \hbox{and} %\quad 
%Q^\llcorner:=\{(a,b)\in\R^2:\, a>0, \, b>0\}
$$
and such that $\Upsilon(0)= (1,0)$ and $\Upsilon'(0)=(0,1)$.

\medskip
To be more precise, let $\nu_{\Sigma}:\Sigma \to \R^{N+1}$ be the choice of the unit normal vector to $\Sigma$, pointing towards $E^+$. For any $\pe\in\Sigma$, there exists a unique $(s,\x,\y)=(s(\pe),\x(\pe),\y(\pe))\in\R\times S^{m-1}\times S^{n-1}$ such that
\begin{equation}
\label{param_O(m)O(n)_invariant}
\pe:=\left(a(s){\tt x},b(s){\tt y}\right).
\end{equation}

We stress that the function 
$$
\pe\in\Sigma\mapsto s(\pe)\in\R
$$
is surjective, but not injective.
%Next, fix $\Phi_m$ and $\Phi_n$ two parametrisations of $S^{m-1}$ and $S^{n-1}$, respectively. We parametrise $\Sigma$ using the mapping $\Psi: \R \times (0,2\pi)^{m-1}\times (0,2\pi)^{n-1} \to \Sigma$ defined by

\medskip
From \eqref{param_O(m)O(n)_invariant}, we compute for $\pe:=\left(a(s){\tt x},b(s){\tt y}\right)$,
\begin{equation}\label{param_O(m)O(n)_invariantnormalvec}
\nu_\Sigma(\pe)=\left(-b'(s){\tt x},a'(s){\tt y}\right)
\end{equation}
and the principal curvatures of $\Sigma$ are computed as
\begin{equation}
\label{princ_curv}
\begin{aligned}
&\lambda_0=\frac{-a''b'+a'b''}{((a')^2+(b')^2)^{3/2}},\\
&\lambda_i=\frac{b'}{a\sqrt{(a')^2+(b')^2}}, \qquad 1\le i\le m-1,\\
&\lambda_j=\frac{-a'}{b\sqrt{(a')^2+(b')^2}}, \qquad m\le j \le m+n-2.\\
\end{aligned}
\end{equation}

The principal curvatures of $\Sigma$ allow us to compute $H_{\Sigma}$ and $|A_{\Sigma}|^2$ as follows: 
\begin{equation}\label{princcurvmeancursecondfund}
\begin{aligned}
H_{\Sigma} & =  \lambda_0 + \lambda_1 + \lambda_2 + \cdots + \lambda_{m+n-2}\\
|A_{\Sigma}|^2 & =  \lambda_0^2 + \lambda_1^2 + \lambda_2^2 + \cdots + \lambda_{m+n-2}^2.
\end{aligned}
\end{equation}
\medskip

Without any loss of generality, assume that $\Upsilon$ is parametrised by arch-lenght, i.e.
\begin{equation}
\label{arc_length}
(a')^2+(b')^2=1,
\end{equation}
so the fact that $\Sigma$ is a smooth $O(m)\times O(n)-$ invariant minimal surface together with \eqref{princ_curv},\eqref{princcurvmeancursecondfund} and \eqref{arc_length}, yield
\begin{equation}
\label{0_mc_O(m)O(n)_invariant}
H_{\Sigma}=-a'' b'+b''a'+(m-1)\frac{b'}{a}-(n-1)\frac{a'}{b}=0
\end{equation}
and
\begin{equation}
\label{1_mc_O(m)O(n)_invariant}
\begin{aligned}
|A_{\Sigma}|^2 &= \left(-a''b'+ a'b''\right)^2 + (m-1)\left(\frac{b'}{a} \right)^2 + (n-1)\left(\frac{a'}{b} \right)^2.
\end{aligned}
\end{equation}

%From \eqref{0_mc_O(m)O(n)_invariant}, we observe that $\Sigma^+_{m,n}=\Sigma^-_{n,m}$. 
%\medskip
On the other hand, from parts (ii), (iii) and (v) in Theorem \ref{th_Al}, we find that for some $R>0$, $C_{m,n} \backslash B_R$ is diffeomorphic to $\Sigma$ via the $O(m)\times O(n)-$invariant mapping 
$$
C_{m,n}\backslash B_R \ni\pe\mapsto \pe+ {\tt w}^-(\pe)\nu_{C_{m,n}}(\pe)\in \Sigma.
$$

Therefore, the definition of $O(m)\times O(n)-$invariant function extends naturally to functions defined over $\Sigma$. 

\medskip
To be more precise, a function ${\tt v}:\Sigma\to\R$ is invariant under the action of $O(m)\times O(n)$ if and only if there exists $v:\R\to\R$ such that for every $s\in \R$, ${\tt x}\in (0,2\pi)^{m-1}$ and ${\tt y}\in (0,2\pi)^{n-1}$,
$$
{\tt v}(\pe)=v(s(\pe)).
$$

%\medskip

%\begin{definition}
%We say that a function ${\tt v}:\Sigma\to\R$ is $O(m)\times O(n)$-invariant if there exists an even function $v:\R\to\R$ such that ${\tt v}(\zeta)=v(s(\zeta))$.
%\begin{equation}\notag
%\psi(y)=\tilde{\psi}(s), \qquad\forall\, y=\Psi(s,{\tt x},{\tt y}):\, ({\tt x},{\tt y})\in(0,2\pi)^{m-1}\times(0,2\pi)^{n-1}.
%\end{equation}
%\label{def_O(m)O(n)_invariant_Sigma}
%\end{definition}
%Moreover, for any integer $l\ge 0$, and $\alpha\in(0,1)$, we define 
%\begin{equation}
%\label{def_space_O(m)O(n)_invariant_Sigma}
%\mathcal{C}^{l,\alpha}_{m,n}(\Sigma):=\{\psi\in C^{l,\alpha}(\Sigma):\,\psi\text{ is $O(m)\times O(n)$-invariant}\},
%\end{equation}
%where $C^{l,\alpha}(\Sigma)$ is the space of $l$-times differentiable functions $\psi:\Sigma\to\R$ whose $l$-th derivatives are H\"{o}lder continuous of exponent $\alpha$. In the sequel we will write $\psi(y)=\tilde{\psi}(s)$. 

Roughly speaking, a function defined on $\Sigma$ is $O(m)\times O(n)-$invariant if it depends only on the arch-length parameter $s$ of the profile curve $\Upsilon$ or equivalently it depends only $|\pe|$.

\medskip
Observe that \eqref{1_mc_O(m)O(n)_invariant} implies that $|A_{\Sigma}|^2$ depends only of the arch-length variable $s$ and hence it is $O(m)\times O(n)-$invariant.

%\medskip
%We next study the $O(m)\times O(n)-$invariant Jacobi fields of $\Sigma$.  We actually prove that the hypersurface $\Sigma$ constructed in Theorem \ref{th_Al} enjoys similar properties.

%\medskip
%In subsection \ref{subsec_stability}, we saw that the cone $C_{m,n}$ has exactly two positive linearly independent $O(m)\times O(n)$-invariant Jacobi fields. This is consistent with stability, which is known to be equivalent to the existence of at least one positive Jacobi field. 

\medskip
Next, study the invertibility theory for the linear equation
\begin{equation}
\label{eq_Jacobi_lin}
\Delta_\Sigma{\tt q}+|A_\Sigma|^2{\tt q}={\tt f} \quad \hbox{in} \quad \Sigma
\end{equation}
in the class of $O(m)\times O(n)$-invariant functions.

\medskip

\subsection{The Emden-Fowler change of variables in the Jacobi operator}\label{subsec_EF_Jacobi}

In terms of the coordinates $\pe=(a(s)\x,b(s)\y)$ in (\ref{param_O(m)O(n)_invariant}), the Laplace-Beltrami operator acting on $O(m)\times O(n)$-invariant functions $\q(\pe)=q(s(\pe))$ reads as
\begin{equation}
\label{Laplace_Sigma_O(m)O(n)_invariant}
\Delta_\Sigma{\tt q}=\partial^2_s q+\alpha(s)\partial_s q\quad \hbox{with}\quad \alpha(s):=(m-1)\frac{a'}{a}+(n-1)\frac{b'}{b}.
\end{equation}

Setting $\beta(s(\pe)):=|A_\Sigma(\pe)|^2$, the equation \eqref{eq_Jacobi_lin} becomes
\begin{equation}\label{JTincoordPsi}
\partial^2_s q+\alpha(s)\partial_s q+\beta(s) q=f \quad \hbox{in} \quad \R,
\end{equation}
where we have set ${\tt f}(\pe)=f(s(\pe))$.

\medskip

The $O(m)\times O(n)-$invariance allows us to restrict ourselves to the case when $q$ and $f$ are even and consequently, we study \eqref{JTincoordPsi} for $s>0$ with the boundary condition $\partial_s q(0) =0$. 

\medskip
Consider the {\it Emden-Fowler} change of variables $s=e^t$ and for $t\in \R$, set  
\begin{equation}\label{weightsJTEmdenFowler}
\tilde{\alpha}(t):=\alpha(e^t)e^t-1 \quad \hbox{and}\quad \tilde{\beta}(t):= \beta(e^t)e^{2t}.
\end{equation}

Consider also,
\begin{equation}\label{Cancellingfirstderterm1}
p(t):=\exp \bigg ({-\int_0^t \frac{\tilde{\alpha}(\tau)}{2}d\tau}\bigg) \quad \hbox{for} \quad t\in\R,
\end{equation}
so that $p(t)$ solves
\begin{equation}\label{Cancellingfirstderterm2}
2\frac{\partial_t p}{p}+\tilde{\alpha}(t)=0.
\end{equation}

\medskip
We look for a solution to \eqref{JTincoordPsi} having the form $q(s)=p(t)u(t)$. From \eqref{weightsJTEmdenFowler} and \eqref{Cancellingfirstderterm2}, we find that
$$
\partial^2_s q+\alpha(s)\partial_s q+\beta(s) q=
$$
\begin{equation}\notag
\begin{aligned}
&=\partial^2_t u+\bigg(2\frac{\partial_t \p}{\p}+\tilde{\alpha}(t)\bigg)\partial_t u+\left(\frac{\partial^2_t \p}{\p}+\tilde{\alpha}(t)\frac{\partial_t \p}{\p}+\tilde{\beta}(t)\right)u\\
&=  \partial^2_t u+\left(\frac{\partial^2_t \p}{\p}+\tilde{\alpha}(t)\frac{\partial_t \p}{\p}+\tilde{\beta}(t)\right)u. 
\end{aligned}
\end{equation}

Set 
\begin{equation}\label{righthandandpotential}
\begin{aligned}
\tilde{f}(t)&:=\frac{e^{2t}}{\p(t)}f(e^t)\quad \hbox{and} \quad 
V(t)&:=\frac{\partial^2_t \p}{\p}+\tilde{\alpha}(t)\frac{\partial_t \p}{\p}+\tilde{\beta}(t)
\end{aligned}
\end{equation}
for $t\in \R$.

\medskip
From \eqref{weightsJTEmdenFowler} and \eqref{Cancellingfirstderterm1}, 
$$
V(t)=-\frac{1}{4}(\alpha(e^t)e^t-1)^2+\frac{1}{2}(\alpha'(e^t)e^{2t}+\alpha(e^t)e^t)+\beta(e^t) e^{2t}
$$
and from the equation \eqref{JTincoordPsi} for ${q}$, we find that  $u$ must solve 
\begin{equation}
\label{ODE_Jacobi_EF}
\partial^2_t u+V(t)u=\tilde{f} \quad \hbox{in} \quad \R.
\end{equation}

In order to solve \eqref{ODE_Jacobi_EF}, we must analyze the asymptotic behavior of $V(t)$ as $t\to \pm \infty$. This analysis is done by studying the asymptotic behavior at $s=0$ and at infinity of the functions $a$ and $b$ related to \eqref{param_O(m)O(n)_invariant}, as well as its derivatives.

\medskip
First, notice that $a$ is an even function while and $b$ is odd and since $\Upsilon(0)=(1,0)$ and $\Upsilon'(0)=(0,1)$, then 
\begin{equation}
\begin{aligned}
a(0) &=1  \quad \hbox{and} \quad b(0)&=0,\\
a'(0)& =0 \quad \hbox{and} \quad b'(0)&=1.
\end{aligned}
\end{equation}

Let $T,N:\R \to \R^2$ denote the tangent and a choice of the unit normal vector to $\Upsilon$ respectively, so that $\{T,N\}$ is a Frenet frame with positive orientation. Thus, 
$$
T(s)=(a'(s),b'(s)) \quad \hbox{and} \quad N(s)=(-b'(s),a'(s)).
$$

\medskip
 Let $k:\R \to \R$ denote the curvature of $\Upsilon$. Thus, $k$ is an even function that does not change sign and from \eqref{0_mc_O(m)O(n)_invariant}, 
\begin{equation}\label{curvaturegamma}
k:=-T'\cdot N=-a''b+ a'b''= -(m-1)\frac{b'}{a}+(n-1)\frac{a'}{b}.
\end{equation}

Performing a Taylor expansion around $s=0$, we fix $s_0\in (0,1)$ such that for any $s\in (0,s_0)$,
\begin{equation}\label{asymptoticsabzero}
\begin{aligned}
a(s) &=1 + \frac{a''(0)}{2}s^2 + \mathcal{O}(s^4)\\
b(s) &=s + \frac{b^{(3)}(0)}{6}s^3 + \mathcal{O}(s^5).
\end{aligned}
\end{equation}

Also,
$$
0=\lim \limits_{s\to 0} H_{\Sigma}=  (m-1) -a''(0)n  
$$
so that 
$$
a''(0)=\frac{m-1}{n}.
$$

On the other hand, from Taylor expansion, \eqref{curvaturegamma} and \eqref{asymptoticsabzero},
\begin{equation}\label{curvatureupsilon}
k(s)= -\frac{m-1}{n} + \mathcal{O}(s^2) \quad \hbox{for} \quad s \in (0,s_0).
\end{equation}

Recall that $\beta(s)=|A_{\Sigma}(s(\pe))|^2$. From \eqref{0_mc_O(m)O(n)_invariant}, \eqref{1_mc_O(m)O(n)_invariant} and \eqref{curvatureupsilon}, performing again the Taylor expansion for $\beta(s)$ around zero we find that for any $s\in (0,s_0)$,
\begin{equation}\label{AAatzero}
\beta(s) = \underbrace{\frac{N(m-1)}{n}}_{=:c_0} + \mathcal{O}(s^2).
\end{equation}

\medskip

As for the asymptitoc behavior of $\beta(s)$ at inifinity, we proceed as follows.

\medskip
Since outside a ball $\Sigma$ is the normal graph over the cone $C_{m,n}$ of an $O(m)\times O(n)$-invariant function ${\tt w}^-(\pe)=w^-(r)$, $r=|\pe|$, (see part {v} in Theorem \ref{th_Al}), then
$$(a(s),b(s))%=\pe+{\tt w}^-(\pe)\nu_{C_{m,n}}(\pe)
=\frac{1}{\sqrt{N-1}}\left(\sqrt{m-1},\sqrt{n-1}\right)r + \frac{1}{\sqrt{N-1}}\left(-\sqrt{n-1},\sqrt{m-1}\right)w^-(r) 
$$ 
where $s=s(\pe)$ is the arch length parameter along $\Upsilon$.

\medskip
Thus, for some $r_0>0$ fixed, 
$$
s=\int_{r_0}^r \sqrt{1+(\partial_r w^-(r'))^2} dr'=r+O(r^{\gamma_+ -\alpha}) \quad \hbox{as} \quad r\to \infty
$$
with $\alpha>0$.

\medskip
On the other hand,
\begin{equation}\label{abasymptoticsinfinity1}
\begin{aligned}
\lim \limits_{s\to \infty}\frac{a(s)}{s}=
\lim \limits_{s\to \infty}a'(s)=\sqrt{\frac{m-1}{N-1}}\\
 \lim \limits_{s\to \infty}\frac{b(s)}{s}=
\lim \limits_{s\to \infty}b'(s)=\sqrt{\frac{n-1}{N-1}}
\end{aligned}
\end{equation}

%Abusing the notation, we notice from \eqref{abasymptoticsinfinity1} and part (v) in Theorem \ref{th_Al}  that for any $s\in \R$ large enough, 
%\begin{equation}\label{radialcrosssectSigma}
%\sqrt{1 + (\partial_s {\tt v}^-(s))^2}\, \Upsilon'(s) = \frac{1}{\sqrt{N-1}}\big(\sqrt{m-1}, \sqrt{n-1}\big) + \partial_s{\tt v}^-(s)\nu_{C_{m,n}}(s).
%\end{equation}
%From \eqref{rate-dec-Sigma} and \eqref{radialcrosssectSigma} (see also Theorem 2.1 in \cite{HS} for more details), 
and so we can fix $c_1$ and $s_1 \in (s_0,\infty)$ such that for any $s\in (s_1, \infty)$,
\begin{equation}\label{asymptoticsabatinfinity}
\begin{aligned}
a(s) &=\sqrt{\frac{m-1}{N-1}}s + \sqrt{\frac{n-1}{N-1}}c_1 s^{\gamma_+} + \mathcal{O}(s^{\gamma_+ - \alpha}),\\
b(s)& =\sqrt{\frac{n-1}{N-1}}s - \sqrt{\frac{m-1}{N-1}}c_1 s^{\gamma_+} + \mathcal{O}(s^{\gamma_+ - \alpha})
\end{aligned}
\end{equation}
and these expressions can be differentiated.

\medskip

Putting together \eqref{1_mc_O(m)O(n)_invariant} and \eqref{asymptoticsabatinfinity}, for any $s>s_1$,
\begin{equation}\label{AAatinfinity}
\beta(s)=\frac{N-1}{s^2} + \mathcal{O}(s^{-3}).
\end{equation}

We also remark that $\partial_{s}\beta<0$ in $\R$.

\medskip
Next, we study the asymptotic behavior of the function 
$$
\alpha= (m-1)\frac{a'}{a} + (n-1)\frac{b'}{b}\quad \hbox{in} \quad \R.
$$

By fixing $s_0>0$ smaller and $s_1>s_0$ larger if necessary, we find from \eqref{asymptoticsabzero} that for any $s\in (0,s_0)$,
\begin{equation}\label{coeffifirstderatzero}
\alpha(s)= \frac{n-1}{s} + \mathcal{O}(s) 
\end{equation}
and from \eqref{asymptoticsabatinfinity} that for any $s>s_1$,
\begin{equation}\label{coeffifirstderatinfty}
\alpha(s)= \frac{N-1}{s} + \mathcal{O}(s^{-2}) 
\end{equation}

Summarising, the function $\beta:\R \to (0,\infty)$ is a positive, smooth, even and strictly decreasing function such that for some $c_0>0$,
\begin{equation}\label{betaasymptotics}
\beta(s)=\left\{
\begin{aligned}
&c_0 + \mathcal{O}(s^2), & \quad  0< s <s_0\\
&\frac{N-1}{s^2}+ \mathcal{O}(s^{-3}), & \quad s>s_1,
\end{aligned}
\right.
\end{equation}
while the function $\alpha:(0,\infty)\to \R$ is a positive, smooth function such that
\begin{equation}\label{alphaasymptotics}
\alpha(s)=\left\{
\begin{aligned}
&\frac{n-1}{s} + \mathcal{O}(s), & \quad  0< s <s_0\\
&\frac{N-1}{s}+ \mathcal{O}(s^{-2}), & \quad s>s_1.
\end{aligned}
\right.
\end{equation}
  
%Since the curve $\Upsilon$ is parametrised by arch-length, we have 
%\begin{equation}
%\label{def_curv_gamma}
%\gamma^{''}(s)=k(s)\gamma'(s)^\bot,
%\end{equation} 
%where $k$ is the curvature of $\gamma$, thus
%\begin{equation}
%a^{''}(0)=k(0)=:k_0>0, \qquad b^{''}(0)=0,
%\end{equation}
%since $a$ achieves its global minimum at $s=0$. %Differentiating (\ref{def_curv_gamma}), we have
%\begin{equation}
%\gamma^{'''}(s)=-k^2(s)\gamma'(s)+k'(s)%\gamma'(s)^\bot,%
%\end{equation}
%which yields that
%\begin{equation}
%a^{'''}(0)=0, \qquad b^{'''}(0)=-k_0^2.
%\end{equation}
%Therefore
%\begin{equation}
%\begin{aligned}
%& a(s)=a_0+\frac{k_0}{2}s^2+\bar{a}(s), \qquad %|\bar{a}(s)|\le cs^4,\\
%& b(s)=s-\frac{k_0^2}{6}s^3+\bar{b}(s), \qquad |\bar{b}(s)|\le cs^5
%\end{aligned}
%\end{equation}
%for $s$ small enough. Using that, by %(\ref{0_mc_O(m)O(n)_invariant}), we have
%\begin{equation}\notag
%k(s)=\gamma''(s)\cdotp\gamma'(s)^\bot=-a''b'+b''a'=(n-1)\frac{a'}{b}-(m-1)\frac{b'}{a}.
%\end{equation}
%Evaluating at $s=0$, we get
%\begin{equation}
%a_0 k_0=\frac{m-1}{n-2}.
%\end{equation}

%Therefore, as $s\to 0^+$, we have
%\begin{equation}\notag
%\alpha(s)s\to n-1,\qquad \alpha'(s)s^2\to 1-n,%\qquad \beta(s)s^2\to k^2_0\bigg(n(n-1)+\frac{m(n-2)^2}{m-1}-2(n-1)(m-1)\bigg),
%\end{equation}

Next, denote
\begin{equation}\label{T0T1}
T_0:=\ln(s_0) \quad \hbox{and} \quad T_1:=\ln(s_1).
\end{equation}

From \eqref{weightsJTEmdenFowler}, \eqref{betaasymptotics} and \eqref{alphaasymptotics},
\begin{equation}\label{alphatildeasymptotics}
\tilde{\alpha}(t)=\left\{
\begin{aligned}
(n-2) + \mathcal{O}(e^{2t}) \quad \hbox{for} \quad t<T_0, \\
(N-2) + \mathcal{O}(e^{-t}) \quad \hbox{for} \quad t>T_1
\end{aligned}
\right.
\end{equation}
and
\begin{equation}\label{betatildeasymptotics}
\tilde{\beta}(t)=\left\{
\begin{aligned}
c_0 e^{2t} + \mathcal{O}(e^{4t}), & \quad \hbox{for}\quad t<T_0,\\
N-1 + \mathcal{O}(e^{-t}), & \quad \hbox{for}\quad t>T_1.
\end{aligned}
\right.
\end{equation}

On the other hand, from \eqref{Cancellingfirstderterm1} and \eqref{alphatildeasymptotics},
\begin{equation}\label{asymptoticsp}
\p(t)=\left\{
\begin{aligned}
e^{-\frac{n-2}{2}t}\left(1 + \mathcal{O}(e^{2t}) \right), & \quad t<T_0\\
e^{-\frac{N-2}{2}t}\left(1 + \mathcal{O}(e^{-t}) \right), & \quad t>T_1.
\end{aligned}
\right.
\end{equation}

We remark that the asymptotic behaviours described in \eqref{alphatildeasymptotics}, \eqref{betatildeasymptotics} and \eqref{asymptoticsp} can be differentiated in the variable $t$.

\medskip
From this discussion and a straight forward computation we find that
\begin{equation}\label{derivativesp}
\frac{\partial_{tt} \p(t)}{\p(t)} + \tilde{\alpha}(t)\frac{\partial_{t}\p(t)}{\p(t)}=\left\{
\begin{aligned}
-\frac{(n-2)^2}{4} + \mathcal{O}(e^{2t}), \quad \hbox{for}\quad t<T_0\\
-\frac{(N-2)^2}{4} + \mathcal{O}(e^{-t}), \quad \hbox{for}\quad t>T_1.
\end{aligned}
\right.
\end{equation}

Putting together  \eqref{righthandandpotential}, \eqref{betatildeasymptotics} and \eqref{derivativesp}, we find that 
\begin{equation}\label{Vasymptotics}
V(t)=\left\{
\begin{aligned}
- \frac{(n-2)^2}{4} +  \mathcal{O}( e^{2t}), & \quad \hbox{for}\quad t<T_0\\
- \frac{(N-2)^2}{4}+ (N-1)+\mathcal{O}(e^{-t}), & \quad \hbox{for}\quad t>T_1
\end{aligned}
\right.
\end{equation}
and this relations can be differentiated. We also  recall that $N \geq 7$ and so 
$$
\left(\frac{N-2}{2}\right)^2-(N-1)>0.
$$

%%\begin{equation}
%V(t)\to-\left(\frac{n-2}{2}\right)^2 \qquad%\text{as $t\to -\infty$.}
%\end{equation}
%
%
%The behaviour of the potential for large $s$ is given by the asymptotic behaviour of $a$ and $b$, that is
%\begin{equation}\notag
%\begin{aligned}
%& a(s)\sim\sqrt{\frac{m-1}{N-1}}s,\qquad b(s)\sim\sqrt{\frac{n-1}{N-1}}s,\\
%& a'(s)\to\sqrt{\frac{m-1}{N-1}},\qquad b'(s)\to\sqrt{\frac{n-1}{N-1}},\\
%& a''(s)\to 0 ,\qquad b''(s)\to 0
%\end{aligned}
%\end{equation} 
%as $s\to\infty$.  studying the asymptotic behaviour of $p$ as $t\to\pm\infty$. 
%
%
%\medskip
%Observe that
%\begin{equation}
%\label{p_pm_infty}
%\begin{aligned}
%&p(t)=O(e^{-\lambda t})\qquad \text{as $t\to-%\infty$,}\\
%&p(t)=O(e^{-\frac{N-2}{2} t})\qquad \text{as $t\to\infty$,}
%\end{aligned}
%\end{equation}
%where $\lambda:=\frac{n-2}{2}>0$. In order to compute the limits of the potential as $t\to\pm\infty$, This yields that
%\begin{equation}
%\label{V(infty)}
%V(t)\to-\left(\frac{N-2}{2}\right)^2+(N-1)%%\qquad\text{as $t\to\infty$.}
%\end{equation}

%The result for $\Sigma^+_{m,n}$ is obtained by observing that the functions $a^+_{m,n}$ and $b^+_{m,n}$ of $\Sigma^+_{m,n}$ satisfy
%$$a^+_{m,n}=b^-_{n,m}, \qquad b^+_{m,n}=a^-_{n,m}.$$
%In this case, we have
%$$\lim_{t\to-\infty}V(t)=\frac{(m-1)^2}{4},\qquad\lim_{t\to\infty}V(t)=\left(\frac{N-2}{2}\right)^2-(N-1).$$
%The expression of the squared second fundamental form follows from (\ref{princ_curv}), (\ref{0_mc_O(m)O(n)_invariant}) and (\ref{arc_length}). We stress that it is crucial to use that $\Sigma$ is a minimal surface.

\subsection{Jacobi fields of $\Sigma$}\label{subsec_Jacobi_fields}

The minimality of $\Sigma$ is invariant under dilation and this allows us to find an explicit  smooth  and $O(m)\times O(n)-$invariant Jacobi field for $\Sigma$, namely the function
$$
\Sigma \ni \pe \mapsto \pe \cdot \nu_{\Sigma}(\pe). 
$$

In addition, this Jacobi field does not change sign. This follows from the fact that the family $\{\lambda\Sigma\}_{\lambda>0}$ is a foliation of the connected component of $\R^{N+1}\backslash C_{m,n}$ containing $\Sigma$. 

\medskip
We will use this information to prove the following Proposition. Recall that 
$$
\gamma_{\pm}=-\frac{N-2}{2}\pm \sqrt{\frac{(N-2)^2}{4} - (N-1)}<0
$$
and consider $c_1>0$ and $\alpha>0$ the constants in \eqref{asymptoticsabatinfinity}.

\begin{proposition}\label{prop_Jacobi_Sigma}
There exist exactly two $O(m)\times O(n)$-invariant linearly independent Jacobi fields ${\tt v}_\pm(\pe)=v_\pm(s(\pe))>0$ of $\Sigma$ such that
\begin{itemize}
\item[(i)] $v_+(s)$ is smooth, even in the variable $s$, $v_+(0)=1$ and
\begin{equation}\label{JacField+asympt}
v_+(s)=c_1 s^{\gamma_+} \big( 1 + \mathcal{O}(s^{-\alpha})\big)\quad  \hbox{as} \quad s\to\infty
\end{equation}

\medskip
\item[(ii)] $v_-(s)$ is smooth except at $s=0$, where it is singular and for some $c_2>0$
\begin{equation}\label{JacField-asympt1}
v_-(s)= \left\{
\begin{aligned}
s^{-(n-2)}(1 + \mathcal{O}(s^2)) & \quad \hbox{as} \quad s \to 0\\
c_2s^{\gamma_-} \big( 1 + \mathcal{O}(s^{-\alpha})\big) &\quad  \hbox{as} \quad s \to \infty,
\end{aligned}
\right.
\end{equation}
where $\alpha\in (0,1)$. Moreover, relation \eqref{JacField+asympt} and \eqref{JacField-asympt1} can be differen\-tiated.
\end{itemize}
%
%\begin{equation}
%\begin{aligned}
%&v_\pm(s)=O(s^{\gamma^\pm}), \qquad\text{as $s\to\infty$,}\\
%&v_+(0)=1,\qquad v_-(s)=O(s^{-(n-2)}),\qquad\text{as $s\to 0^+$.}
%\end{aligned}
%\end{equation}
\end{proposition}

\begin{proof}
Set ${\tt v}_+(\pe)=-\pe\cdotp \nu_\Sigma (\pe)$ for $\pe\in \Sigma$. We know that ${\tt v}_+$ is a smooth Jacobi field which does not change sign. Directly from \eqref{param_O(m)O(n)_invariant} and \eqref{param_O(m)O(n)_invariantnormalvec}, we find for $\pe=\left(a(s){\tt x},b(s){\tt y}\right)$ that
\begin{equation}\label{Jacobi_field_positive}
{\tt v}_+(\pe)=-\pe\cdotp\nu_\Sigma(\pe)
\end{equation}
so that $v_+(s):={\tt v}(\pe)=a(s)b'(s)-a'(s)b(s)$ and hence it is even in the variable $s$. Since $v_+$ does not change sign and $v_+(0)=1$, then ${v}_+>0$ in $\Sigma$.

%\medskip
%From \eqref{radialcrosssectSigma}, the normal vector of $\Upsilon(s)$, namely $N(s)$, can be computed as
%\begin{equation}\label{radialnormalcrosssectSigma}
%\sqrt{1 + (\partial_s {\tt v}^-(s))^2}\, N(s) = \nu_{C_{m,n}}(s) + \frac{\partial_{s}{\tt v^-(s)}}{\sqrt{N-1}}\big(\sqrt{m-1}, \sqrt{n-1}\big).
%\end{equation}

%Thus, using \eqref{radialcrosssectSigma}, \eqref{radialnormalcrosssectSigma} and abusing the notation, we find that in the coordinates $\zeta =\Psi(s,\x,\y)$,
%$$
%\begin{aligned}
%v_+(s) &= \zeta\cdotp\nu_\Sigma(\zeta)\\
%&=(s,v(s))\cdotp\frac{(-v'(s),1)}{\sqrt{1+%(v')^2}}=\frac{-sv'(s)+v(s)}{\sqrt{1+(v')^2}}%
%=O(s^{\gamma^+}), \qquad s\to\infty.
%\end{aligned}
%$$

\medskip

From \eqref{asymptoticsabatinfinity} and a direct computation, 
\begin{equation*}
v_+(s)=c_1 s^{\gamma_+} \big( 1 + \mathcal{O}(s^{-\alpha})\big)\quad  \hbox{as} \quad s\to\infty,
\end{equation*}
where we assume that $\alpha \in (0,1)$. 

\medskip
This relation can be differentiated in the sense that 
\begin{equation}\label{JacField+derivasympt}
s v'_+(s)=  c_1 \gamma_+ \,s^{\gamma_+} \big( 1 + \mathcal{O}(s^{-\alpha})\big)\quad  \hbox{as} \quad s\to\infty
\end{equation}
and this proves \eqref{JacField+asympt} and completes the proof of (i).

\medskip
We next, prove (ii). With the Emden-Fowler change of variables $s=e^t$ and using the function defined in \eqref{Cancellingfirstderterm1}, we write
$$
v_+(s)=\p(t)u_+(t) \quad \hbox{for} \quad t\in \R.
$$

From \eqref{ODE_Jacobi_EF} with $\tilde{f}=0$, we see that $u_+$ must solve the ODE 
\begin{equation}\label{JacFieldinEmdFow+}
\partial^2_t u_+ +V(t)u_+=0 \quad \hbox{in} \quad \R.
\end{equation}

Setting, 
\begin{equation}\label{capitallowercaselambda}
\lambda:=\frac{n-2}{2} \quad \hbox{and}\quad  \Lambda:= \sqrt{\left(\frac{N-2}{2}\right)^2-(N-1)}
\end{equation}
we find from \eqref{asymptoticsp} and \eqref{JacField+asympt} that for some $c>0$,
\begin{equation}\label{asymptoticsu+}
u_+(t)=\left\{
\begin{aligned}
e^{\lambda t}\left(1 + \mathcal{O}(e^{2t}) \right), & \quad t<T_0\\
c_1e^{\Lambda t}\left(1 + \mathcal{O}(e^{- \alpha t}) \right), & \quad t>T_1.
\end{aligned}
\right.
\end{equation}

Proceeding in a similar fashion using \eqref{asymptoticsp}, \eqref{JacField+asympt}, \eqref{JacField+derivasympt}, \eqref{asymptoticsu+} and the fact that, 
$$
\partial_t u_+(t)=\frac{e^t}{\p(t)}\partial_s v_+(e^t) - \frac{\partial_t \p(t)}{\p(t)} u_+(t),
$$
we find that 
\begin{equation}\label{asymptoticsderivu+}
\partial_ t u_+(t)=\left\{
\begin{aligned}
\lambda \, e^{\lambda t}\left(1 + \mathcal{O}(e^{2t}) \right), & \quad t<T_0\\
\Lambda c_1\,e^{\Lambda t}\left(1 + \mathcal{O}(e^{- \alpha t}) \right), & \quad t>T_1.
\end{aligned}
\right.
\end{equation}

Next, we find the second predicted Jacobi field of $\Sigma$. Set,
\begin{equation}\label{VarparamJacFieldEmdFow}
u_-(t):=u_+(t)\int_t^{\infty} \frac{1}{u_+(\tau)^2}d\tau \quad \hbox{for}\quad t\in \R.
\end{equation}

Clearly, $u_-$ is smooth and positive in $\R$. From \eqref{JacFieldinEmdFow+} and the variation of para\-meters formula, we conclude that $u_-$ solves 
\begin{equation}\label{JacFieldinEmdFow-}
\partial^2_t u_- +V(t)u_-=0 \quad \hbox{in} \quad \R
\end{equation}
and $u_+,u_-$ are linearly independent.

\medskip
As for the asymptotic behavior of $u_-$, we estimate directly from \eqref{asymptoticsu+} and \eqref{VarparamJacFieldEmdFow}, to find that after normalisation of the solution, for some $c_2>0$,
\begin{equation}\label{asymptoticsu-}
u_-(t)=\left\{
\begin{aligned}
e^{-\lambda t}\left(1 + \mathcal{O}(e^{2t}) \right), & \quad t<T_0\\
c_2e^{-\Lambda t}\left(1 + \mathcal{O}(e^{- \alpha t}) \right), & \quad t>T_1,
\end{aligned}
\right.
\end{equation}
where we have taken $T_0<0$ smaller and $T_1> T_0$ larger if necessary, but still fix.

\medskip
Similarly,  
\begin{equation}\label{asymptoticsderivu-}
\partial_ t u_-(t)=\left\{
\begin{aligned}
-\lambda \, e^{-\lambda t}\left(1 + \mathcal{O}(e^{2t}) \right), & \quad t<T_0\\
-\Lambda \,c_2e^{-\Lambda t}\left(1 + \mathcal{O}(e^{- \alpha t}) \right), & \quad t>T_1.
\end{aligned}
\right.
\end{equation} 

 \medskip
Going back to the original coordinate $\pe=\left(a(s){\tt x},b(s){\tt y}\right)$, we define 
$$
{\tt v}_-(\pe) =v_-(s) :=\p\big(\log(s)\big)\,u_-(\log(s)).
$$

We conclude that $v_-$ is smooth and positive in  $(0,\infty)$. From \eqref{JacFieldinEmdFow-}, $v_-(s)$ is another Jacobi field and the classical ODE theory yields that $v_+(s), v_-(s)$ form a fundamental set for all the Jacobi fields of $\Sigma$ that are $O(m)\times O(n)$-invariant. 

\medskip
Finally, from \eqref{asymptoticsp}, \eqref{asymptoticsu-} and \eqref{asymptoticsderivu-}, 
\begin{equation*}\label{JacField-asympt}
v_-(s)= \left\{
\begin{aligned}
s^{-(n-2)}(1 + \mathcal{O}(s^2)) & \quad \hbox{as} \quad s \to 0\\
s^{\gamma_-} \big( 1 + \mathcal{O}(s^{-\alpha})\big) &\quad  \hbox{as} \quad s \to \infty
\end{aligned}
\right.
\end{equation*}
and this relation can be differentiated, i.e.
\begin{equation*}\label{JacField-derivasympt}
s v_-(s)= \left\{
\begin{aligned}
s^{-(n-2)}(1 + \mathcal{O}(s^2)) & \quad \hbox{as} \quad s \to 0\\
s^{\gamma_-} \big( 1 + \mathcal{O}(s^{-\alpha})\big) &\quad  \hbox{as} \quad s \to \infty.
\end{aligned}
\right.
\end{equation*}

This proves \eqref{JacField-asympt1} and completes the proof of the proposition.

%\medskip
%Jacobi field ${\tt v}_-$ with the required asymptotic behaviour as $|\zeta|\to\infty$ and such that $$v_-(s)=O(s^{-(n-2)})\qquad\text{as $s\to 0^+$}$$
%and $$\tilde{\zeta}_-(s)=O(s^{\gamma^-}), \qquad s\to\infty.$$ 
%Note that ${\tt v}_-$ is singular in $(\R^m\times\{0\})\cap\Sigma$. %In the case of $\Sigma^+_{m,n}$, we use oce again the fact that $$a^+_{m,n}=b^-_{n,m}, \qquad b^+_{m,n}=a^-_{n,m},$$
%which yields that
%$$\zeta_+(y)=y\cdotp\nu_\Sigma(y)=\tilde{\zeta}_+(s)=-a(s)b'(s)+a'(s)b(s)>0.$$
\end{proof}

At this point, a few important remarks are in order.

% we have $v_\pm(s)=p(t)u_\pm(t)$, where
%\begin{equation}
%\label{Jacobi_fields_EF}
%\begin{aligned}
%&u_\pm(t)=O(e^{\pm\Lambda t}),\qquad\text{as $t%\to\infty$,}\\
%&u_\pm(t)=O(e^{\pm\lambda t}),\qquad\text{as $t%\to-\infty$.}
%\end{aligned}
%\end{equation}
\begin{remark}{\rm 
\begin{enumerate}
\item[(i)] We stress out that in the upcoming developments, the asymptotics described in expressions \eqref{asymptoticsu+}, \eqref{asymptoticsderivu+}, \eqref{asymptoticsu-} and \eqref{asymptoticsderivu-} will play a crucial role. Particularly, in the proof of Proposition \ref{prop_Jacobi_lin} (see below). 

\medskip
%
%\item[(ii)] Theorem \ref{th_Jacobi} follows as a corollary of Proposition \ref{prop_Jacobi_Sigma}. Additionally, we have also proven that these Jacobi fields can be chosen to be positive.
%

\medskip
\item[(ii)] Since minimal hypersurfaces in $\R^{N+1}$ are invariant under translations and rotations also, one expects the existence of other linearly independent Jacobi fields. However, these Jacobi fields are not $O(m)\times O(n)-$invariant and hence we do not study them here. %they are not important for our construction, since  Moreover, those Jacobi fields do not decay at infinity.
\end{enumerate}}
\end{remark}

\medskip

\subsection{The Jacobi equation}\label{subsec_Jacobi_lin}

Next, we introduce suitable function spaces to study invertibility theory for \eqref{eq_Jacobi_lin}.

%\medskip
%Let $\Psi$ be the parametrisation from \eqref{param_O(m)O(n)_invariant}. For any $\zeta\in\Sigma$, let $s(\zeta)$ be the unique $s\in\R$ such that $\zeta=\Psi(s,\x,\y)$ for some $(\x,\y)\in(0,2\pi)^{m-1}\times(0,2\pi)^{n-1}$.

\medskip
For $\beta \in (0,1)$, $\mu>0$ and a function ${\tt f}:\Sigma \to \R$ we set 
\begin{equation}\label{normsJacLinTheory}
\begin{aligned}
\|{\tt f}\|_{\infty,\mu} & := \|(s(\pe)^2 +2)^{\frac{\mu}{2}}{\tt f}\|_{L^{\infty}(\Sigma)},\\ 
\|{\tt f}\|_{C^{0,\beta}_\mu(\Sigma)}& :=\sup_{\zeta\in\Sigma}(s(\pe)^2+2)^{\frac{\mu}{2}}\|{\tt f}\|_{C^{0,\beta}(B_1(\pe))}
\end{aligned}
\end{equation}
and we consider the Banach space $C^{0,\beta}_{\mu}(\Sigma)$ defined as the space of $O(m)\times O(n)-$invariant functions ${\tt f}\in C^{0,\beta}_{loc}(\Sigma)$ for which the norm
\begin{equation}\label{functionspacerighthandsideJT}
\|{\tt f}\|_{\mathcal{C}^{0,\beta}_{\mu}(\Sigma)}< \infty.
\end{equation}

% $\Sigma$, where ${\tt f}(\zeta)=f(s(\zeta))$ and ${\tt q}(\zeta)=q(s(\zeta))$ are   For $\beta>0$, $\gamma>0$ and functions ${\tt f}\in C^{0,\beta}_{loc}(\Sigma)$, we introduce the norm
%while, for functions ${\tt q}\in C^{2,\beta}_{loc}(\R)$, we introduce the norm
%\begin{equation}
%\|\phi\|_{C^2_\gamma(\R)}:=\sigma^{-1}\|\partial^2_s\phi\|_{2+\gamma,-\frac{1}{2}}
%+\sigma^{-\frac{1}{2}}\|\partial_s\phi\|_{1+\gamma,\frac{1}{2}}+\|\phi\|_{\gamma,\frac{1}{2}}.
%\end{equation}
%In view of this definition, we say that a function ${\tt f}\in C^{0,\beta}_{loc}(\Sigma)$ is in $\mathcal{C}^{0,\beta}_{2(1+\gamma)}(\Sigma)$ if it is $O(m)\times O(n)$-invariant and 

\bigskip
We also consider the Banach space $\mathcal{C}^{2,\beta}_{\mu}(\Sigma)$ defined as the space of $O(m)\times O(n)-$invariant functions ${\tt q}\in C^{2,\beta}_{loc}(\Sigma)$ for which the norm
\begin{equation}
\|{\tt q}\|_{\mathcal{C}^{2,\beta}_{\infty,\mu}(\Sigma)}:=\|D^2_\Sigma{\tt q}\|_{C^{0,\beta}_{2+\mu}(\Sigma)}+\|\nabla_{\Sigma}{\tt q}\|_{\infty,1+\mu}+\|{\tt q}\|_{\infty,\mu}<\infty.
\end{equation}

The following proposition shows that in this functional analytic setting, the linear operator in \eqref{eq_Jacobi_lin} has an inverse.

\medskip  
\begin{proposition}
\label{prop_Jacobi_lin}
Let $\beta\in(0,1)$ and $\mu>0$. Then, there exists a constant $c>0$ such that for any ${\tt f}\in\mathcal{C}^{0,\beta}_{2+\mu}(\Sigma)$, there exists a solution ${\tt q}\in \mathcal{C}^{2,\beta}_{\infty,\mu}(\Sigma)$ to (\ref{eq_Jacobi_lin}) such that
\begin{equation}\label{EstJacLinProp}
\|{\tt q}\|_{\mathcal{C}^{2,\beta}_{\infty,\mu}(\Sigma)}\le c\|{\tt f}\|_{\mathcal{C}^{0,\beta}_{2+ \mu}(\Sigma)}.
\end{equation}
\end{proposition}

%
%The remaining part of Section \ref{sec_invariant_minsurf} is devoted to the proof of Proposition \ref{prop_Jacobi_lin}. The plan is the following: in subsection \ref{subsec_EF_Jacobi} we introduce suitable coordinates to compute the relevant geometric quantities. In subsection \ref{subsec_Jacobi_fields} we consider the homogeneous version of  \eqref{eq_Jacobi_lin} and study the properties of $O(m)\times O(n)$-invariant Jacobi fields and prove Theorem \ref{th_Jacobi}. Finally, in subsection \ref{subsec_Jacobi_lin} we construct the inverse of \eqref{eq_Jacobi_lin} using the variation of parameters formula.
%

\begin{proof}
Given any ${\tt f}\in C^{0,\beta}_{2+\mu}(\Sigma)$, we write ${\tt f}(\pe)=f(s)$ and look for a solution of  
\eqref{eq_Jacobi_lin} having the form 
$$
q(s):= \p(t)u(t) \quad \hbox{for} \quad s=e^t >0. 
$$

Therefore, $u(t)$ must solve \eqref{ODE_Jacobi_EF} with 
$$
\tilde{f}(t):= \frac{e^{2t}}{\p(t)}f(e^t) \quad \hbox{for} \quad t \in \R
$$
and where the function $p(t)$ and the potential $V(t)$ are described in \eqref{asymptoticsp} and \eqref{Vasymptotics}.

\medskip
Observe that
\begin{equation}\notag
|\tilde{f}(t)|\leq \left\{
\begin{aligned}
e^{(\frac{n}{2}+1)t}\|{\tt f}\|_{C^{0,\beta}_{2+ \mu}(\Sigma)} &\quad \hbox{for} \quad  t\leq T_0,\\
e^{(\frac{N-2}{2} -\mu)t}\|{\tt f}\|_{C^{0,\beta}_{2+ \mu}(\Sigma)} &\quad \hbox{for} \quad  t \geq T_1.
\end{aligned}
\right.
\end{equation}

A solution of \eqref{ODE_Jacobi_EF} is given by the formula:
\begin{equation}\label{SlnODEJac}
u(t)=u_+(t)\int_{-\infty}^t u_-(\tau)W^{-1}(\tau)\tilde{f}(\tau)d\tau-u_-(t)\int_{-\infty}^t u_+(\tau)W^{-1}(\tau)\tilde{f}(\tau)d\tau
\end{equation}
for $t\in \R$.

\medskip
Observe that the Wronskian $W(t)$ of $u_+(t)$ and $u_-(t)$ is constant. Directly from \eqref{asymptoticsu+}, \eqref{asymptoticsderivu+}, \eqref{asymptoticsu-},  \eqref{asymptoticsderivu-} and \eqref{SlnODEJac} we find that for some $C>0$ depending only on $\mu$ and  $N$,
\begin{equation}\label{EstLinftyJaclin}
|u(t)|+|\partial_t u(t)|\le C\left\{
\begin{aligned}
e^{(\frac{n}{2}+1)t}\|{\tt f}\|_{C^{0,\beta}_{2+\mu}(\Sigma)} &\quad\hbox{for} \quad  t\le T_0,\\ 
e^{(\frac{N-2}{2}-\mu) t}\|{\tt f}\|_{C^{0,\beta}_{2+\mu}(\Sigma)}&\quad\hbox{for} \quad t\geq T_1.
 \end{aligned}
\right.
\end{equation}

Pulling back the change of variables, we find from \eqref{EstLinftyJaclin},that $q(s)\sim s^2$ as $s\to 0^+$. In particular, $q(0)=\partial_s q(0)=0$ so that $q$ can be extended to an even function over $\R$.

\medskip
We also find from \eqref{JTincoordPsi}, \eqref{betaasymptotics}, \eqref{alphaasymptotics} and \eqref{EstLinftyJaclin} that 
\begin{equation}
\label{est_EDO_Jacobi}
\|D^2_\Sigma{\tt q}\|_{\infty,2 + \mu}+\|\nabla_\Sigma{\tt q}\|_{\infty,1 + \mu}+\|{\tt q}\|_{\infty,\mu}\le c\|{\tt f}\|_{C^{0,\beta}_{2+ \mu}(\Sigma)}.
\end{equation}

We finish the proof of the estimate \eqref{EstJacLinProp} by applying standard H\"{o}lder regularity. Since the coefficient $\alpha(s)$ in \eqref{JTincoordPsi} is singular at the origin, we rather use regularity theory to the corresponding partial differential equation \eqref{eq_Jacobi_lin}. Thus, using \eqref{eq_Jacobi_lin} we notice that for $s(\pe)\in(0,\frac{1}{2})$ we have
$$
\|{\tt q}\|_{C^{2,\beta}(B_{\frac{3}{2}}(\pe))}\le c(\|{\tt q}\|_{L^\infty(B_2(\pe))}+\|{\tt f}\|_{C^{0,\beta}(B_2(\pe))})\le \|{\tt f}\|_{\mathcal{C}^{0,\beta}(B_2(\pe))}.
$$

If $s(\pe)\ge\frac{1}{2}$, the same estimate follows from \eqref{JTincoordPsi}, \eqref{est_EDO_Jacobi} and the fact that $\alpha$ is bounded and smooth outside a neighbourhood of the origin. This completes the proof of the proposition.
\end{proof}

\section{The Jacobi-Toda equation} \label{JacToda Section}

In this section we provide a detailed proof of the Theorem \ref{th_Liouville}.  We proceed by studying solvability theory for the equation
\begin{equation}\label{JacTodaSection3}
\delta J_\Sigma {\tt h}-2 a_\star e^{-\sqrt{2}{\tt h}}=0,
\end{equation}
where $\delta>0$ is a small parameter and $a_{\star}>0$ is a constant. We also recall that $J_{\Sigma}$ is the Jacobi operator of $\Sigma$ described in \eqref{eq_Jacobi_lin},  \eqref{Laplace_Sigma_O(m)O(n)_invariant} and \eqref{JTincoordPsi}.

%\medskip
%The forthcoming analysis is done within the class of $O(m)\times O(n)-$invariant functions and hence we write 
%$$
%{\tt f}(\zeta)=f(s) \quad \hbox{for} \quad s\in \R
%$$
%with $f:\R\to\R$ an even function.

\medskip
In what follows we will use the following notation. For a function ${\tt v}$ defined on $\Sigma$, set
\begin{equation}\label{errorJacTodaSection3}
E_\delta({\tt v}):=\delta J_\Sigma {\tt v}-  2a_\star e^{-\sqrt{2}{\tt v}}
\end{equation}
and also denote
\begin{equation}\label{JacTodaQuadracticerror}
Q(t):=e^{-\sqrt{2}t}-1+\sqrt{2}t \quad \hbox{for}\quad t\in \R.
\end{equation}

\medskip
To solve \eqref{JacTodaSection3}, we look for a $O(m)\times O(n)-$invariant solution ${\tt h}$ having the form
$$
{\tt h}={\tt v}+{\tt q} \quad \hbox{in} \quad \Sigma,
$$
where ${\tt v}$ is an approximate solution and ${\tt q}$ is a small correction to get a  genuine solution of \eqref{JacTodaSection3}.

\medskip
A direct calculation shows that \eqref{JacTodaSection3} becomes
\begin{equation}
\label{eq_JT}
\delta J_\Sigma {\tt q}+2\sqrt{2}a_\star e^{-\sqrt{2}{\tt v}}{\tt q}=-E_\delta({\tt v})+2a_\star e^{-\sqrt{2}{\tt v}}Q({\tt q}) \quad \hbox{in} \quad \Sigma.
\end{equation}

The strategy consists in selecting ${\tt v}$ as accurately as possible so that the term $E_{\delta}({\tt v})$ is small for $\delta>0$ small, in some suitable topology that allows us to study solvability theory for the linear operator
$$
L_{\delta}({\tt q}):=\delta J_{\Sigma}{\tt q} + 2\sqrt{2}a_{\star} e^{-\sqrt{2}{\tt v}}{\tt q}.
$$

\subsection{The approximate solution} Let us now choose the approximate solution ${\tt v}$. In this part, we will make extensive use of the {\it Lambert function} $W:[0,\infty)\to \R$ defined implicitly as the solution of the algebraic equation
$$
W(z)e^{W(z)}=z 
$$
for any given $z\geq 0$. It is well known that 
\begin{equation}\label{asympLambertfnzeroinfty}
W(z)= \left\{
\begin{aligned}
z-z^2+ \mathcal{O}(z^4), \quad \hbox{as} \quad z\to 0^+,\\
\log(z)- \log \big(\log (z)\big)+ \mathcal{O}\bigg(\frac{\log \big(\log (z)\big)}{\log (z)}\bigg), \quad \hbox{as} \quad z\to \infty
\end{aligned}
\right.
\end{equation}
and these relations can be differentiated. This is the essence of the following lemma. 

\medskip
\begin{lemma}\label{propertieslambert}
The function $W\in C^{\infty}[0,\infty)$ and satisfies that for any $i\in \mathbb{N}$, there exists a constant $C_i>0$ such that for any $z\geq 0$,
\begin{equation}
\label{est_derivatives_Lambert}
|W^{(i)}(z)|\le \frac{C_i}{(1+z)^{i}}.
\end{equation}
\end{lemma}

\begin{proof}
We briefly sketch the proof of \eqref{est_derivatives_Lambert}. Set
\begin{equation}
f_1(z):=\frac{z}{1+z} \quad \hbox{for} \quad z \geq 0
\end{equation}
and observe that for any $l\in \mathbb{N}$,
$$
f^{(l)}_1(z)=\frac{(-1)^{l+1}}{(1+z)^{l+1}}.
$$

Define recursively for $i \geq 1$,
\begin{equation*}
f_{i+1}(z)=if_i(z)-\frac{z}{z+1}f'_i(z) \quad \hbox{for} \quad z\geq 0.
\end{equation*}

We claim that for any $i \geq 1$ and any $l\ge 0$, $f^{(l)}_i\in L^\infty(0,\infty)$, for some $c_i>0$,
\begin{equation}\label{fauxlambert}
f_i(z)=c_i z^{i}(1+o(1)),\qquad z\to 0^+
\end{equation}
and this relation can be differentiated. 

\medskip
This previous claim is obviously true for $i=1$. As for the case $i=2$, a direct computation yields that
$$
f_2(z)=\frac{z^2(z+2)}{(z+1)^3}\quad \hbox{for} \quad z\geq 0
$$
and the claim holds true also in this case.

\medskip
We proceed next by induction assuming that for some $c_i>0$, $f_i(z)=i^{(i-1)} z^i(1+o(1))$ as $z\to 0^+$. Then, as $z\to 0^+$,
$$
\begin{aligned}
f_{i+1}(z) &=ic_i z^i(1+o(1))-\frac{z}{1+z}(ic_i z^{i-1}(1+o(1)))\\
&=c_{i+1}z^{i+1}(1+o(1)),
\end{aligned}
$$
for some $c_{i+1}>0$. This completes the proof of the claim.

\medskip
Again, an argument by induction shows that
\begin{equation}
\label{W_differentiation}
W^{(i)}(z)=\frac{(-1)^{i+1}}{z^i}f_i\big(W(z)\big) \quad \hbox{for} \quad z\geq 0.
\end{equation}

Thus (\ref{est_derivatives_Lambert}) follows from \eqref{fauxlambert} and (\ref{W_differentiation}). This completes the proof of the lemma.
\end{proof}

\medskip
Next lemma states that we can choose an approximate solution as accurate as needed. 

\begin{lemma}\label{lemmaapproximateslnJacToda}
For any $\delta>0$ and for any $j\ge 0$, there exist $O(m)\times O(n)-$invariant functions ${\tt w}_0,\dots,{\tt w}_j$ defined on $\Sigma$ which are smooth and such that the function ${\tt v}_j$ defined by
\begin{equation}\label{recursiveApproxslnJacToda}
{\tt v}_j:={\tt w}_0+\dots+{\tt w}_j
\end{equation}
and written in the coordinate $s=s(\pe)$ as ${\tt v}_j(\pe)=v_j(s)$, satisfies the estimate
\begin{equation}
\label{est_approx_sol}
\left|v_j-\frac{1}{\sqrt{2}}(\log(s^2+2)+|\log\delta|)\right| \le \frac{C}{|\log\delta|^\frac{1}{2}}.
\end{equation}

Moreover, the error defined in  \eqref{errorJacTodaSection3} for ${\tt v}_j$ is given by
\begin{equation}
\label{err_j}
E_\delta({\tt v}_j)=\delta\Delta_\Sigma {\tt w}_j \quad \hbox{in} \quad \Sigma
\end{equation}
and in the coordinate $s=s(\pe)$, satisfies the estimate
\begin{equation}
\label{est_err_j}
|E_\delta({v}_j)|\le \frac{C_j\delta }{(1+s)^2(\log(s+2))^{\frac{j}{2}}|\log\delta|^{\frac{j}{2}}}
\end{equation}
for some constant $C_j>0$ depending only on $j$.
\end{lemma}

%
%The value of $l>0$ will determined later, in order to apply the contraction mapping Theorem, according to the size of the norm of the right inverse of $\delta J_\Sigma-2\sqrt{2}a_\star e^{-\sqrt{2}{\tt v}}$ is of order with respect to $\eps$.
%

\begin{proof}
We proceed recursively to find ${\tt w}_0,{\tt w}_1, \ldots,{\tt w}_j$. For $j=0$, choose ${\tt w}_0$ solving the algebraic equation
\begin{equation}\label{choicew0}
\delta |A_{\Sigma}|^2 {\tt w}_0 - 2a_{\star}e^{-\sqrt{2}{\tt w}_0}=0 \quad \hbox{in} \quad \Sigma.
\end{equation}

A direct calculation yields that 
$$
{\tt w}_0:=\frac{1}{\sqrt{2}}W\left(\frac{2\sqrt{2}a_\star}{\delta|A_\Sigma|^2}\right),
$$
where $W$ is the Lambert function. 

\medskip
Write $w_0(s)={\tt w}_0(\pe)$ for $s=s(\pe)$. If $\delta>0$ is small enough then, $\delta^{-1}|A_{\Sigma}|^{-2}$ is large and using the function $\beta(s)=|A_{\Sigma}(\pe)|^2$, we find from \eqref{asympLambertfnzeroinfty} that
\begin{equation}\label{asymptw0infty}
\begin{aligned}
w_0(s)=\frac{1}{\sqrt{2}}
\log\bigg(\frac{2\sqrt{2}a_\star}{\delta\beta(s)}\bigg)- \frac{1}{\sqrt{2}}\log \bigg(\log &\bigg(\frac{2\sqrt{2}a_\star}{\delta\beta(s)}\bigg)\bigg)\\
&+ \mathcal{O}\left(\frac{\log \bigg(\log (\delta^{-1}(s^2+2))\bigg)}{\log \bigg(\delta^{-1}(s^2+2)\bigg)}\right)
\end{aligned}
\end{equation}
for any $s\in \R$.

\medskip
From \eqref{est_derivatives_Lambert} and iterating the chain rule (see \cite{HMY}), we find that for any $i\in \mathbb{N}$, there exists $C_i>0$, depending only on $\Sigma$, such that
\begin{equation}
\label{est_derivatives_w0}
|\partial^{(i)}_s w_0|\le \frac{C_i}{(1+s)^i}.
\end{equation}

In particular, there exists $C>0$, independent of $\delta>0$, such that (abusing the notation)
\begin{equation}\label{Lapw0}
|\Delta_\Sigma w_0|\le\frac{C}{(1+s)^2},
\end{equation}
so that \eqref{err_j} holds true for $j=0$.

\medskip
Next, we choose ${\tt w}_1$ solving the algebraic equation
\begin{equation}\label{choicew_1}
|A_\Sigma|^{-2}\Delta_\Sigma {\tt w}_0+{\tt w}_1-{\tt w}_0(e^{-\sqrt{2}{\tt w}_1}-1)=0
\end{equation}
from where we find that
$$
{\tt w}_1:=-{\tt w}_0-|A_\Sigma|^{-2}\Delta_\Sigma {\tt w}_0+\frac{1}{\sqrt{2}}W\left(\frac{2\sqrt{2}a_\star}{\delta|A_\Sigma|^2}e^{\sqrt{2}|A_\Sigma|^{-2}\Delta_\Sigma {\tt w}_0}\right).
$$

Using  \eqref{choicew_1},
\begin{equation*}
\begin{aligned}
E_\delta({\tt v}_1) &=E_\delta({\tt w}_0)+\delta\Delta_\Sigma {\tt w}_1+\delta|A_\Sigma|^2 {\tt w}_1-\delta |A_\Sigma|^2 {\tt w}_0(e^{-\sqrt{2}{\tt w}_1}-1)\\
&=\delta \Delta_{\Sigma}{\tt w}_1,
\end{aligned}
\end{equation*}
so that \eqref{err_j} holds true for $j=1$.

\medskip
On the other hand, the asymptotic expansion
\begin{equation}
\label{est_W_large_a}
-a-b+W(ae^{a+b})=-\frac{b}{a}+O\left(\frac{1}{a^2}\right) \qquad \hbox{as}\quad a\to\infty
\end{equation}
holds uniformly for $b$ on any compact interval of $[0,\infty)$. Also, \eqref{Lapw0} yields that the term $|A_{\Sigma}|^{-2}\Delta_{\Sigma}{\tt w}_0$ is uniformly bounded in the hypersurface$\Sigma$ and in the parameter $\delta$.

\medskip

Therefore,  from \eqref{asymptw0infty} and \eqref{est_W_large_a}, there exists $C>0$, independent of $\delta>0$ such that for any $s\in \R$,
$$
|w_1(s)| \leq \frac{C}{\log(s+2)^{\frac{1}{2}}|\log(\delta)|^{\frac{1}{2}}},
$$
where we have written $w_1(s)={\tt w}_1(\pe)$ for $s=s(\pe)$.

\medskip
Next, we show how to proceed recursively to find ${\tt w}_j$. Write 
$$
a_0:=\sqrt{2}{\tt w}_0,\qquad b_0:=\sqrt{2}|A_\Sigma|^{-2}\Delta_\Sigma {\tt w}_0
$$
and for $j\ge 1$, define
\begin{equation}\label{recursivedefwj}
\left\{
\begin{aligned}
a_j&:=a_{j-1}e^{-\sqrt{2}{\tt w}_j}\\
b_j&:=\sqrt{2}|A_\Sigma|^{-2}\Delta_\Sigma {\tt w}_j\\
\sqrt{2}w_{j+1}&:=-a_j-b_j+W(a_j e^{a_j+b_j}).
\end{aligned}
\right.
\end{equation}

We remark that ${\tt w}_{j+1}$ solves the algebraic equation 
\begin{equation}\label{explicitwj}
b_j+\sqrt{2}w_{j+1}-a_j(e^{-\sqrt{2}w_{j+1}}-1)=0
\end{equation}
or equivalently,
\begin{equation}\label{equivrecursidefwj}
\delta\Delta_\Sigma {\tt w}_{j}+\delta|A_\Sigma|^2 {\tt w}_{j+1}-\delta|A_\Sigma|^2 {\tt w}_0 e^{-\sqrt{2}({\tt w}_1+\dots+{\tt w}_j)}(e^{-\sqrt{2}{\tt w}_{j+1}}-1)=0.
\end{equation}

\medskip

We notice also that\eqref{recursivedefwj} yields
\begin{equation*}
\begin{aligned}
{\tt w}_{j+1}=&-{\tt w}_0 e^{-\sqrt{2}({\tt w}_1+\dots+{\tt w}_j)}-|A_\Sigma|^{-2}\Delta_\Sigma {\tt w}_j\\
&+\frac{1}{\sqrt{2}}W\left(\sqrt{2}{\tt w}_0\exp(-\sqrt{2}({\tt w}_1+\dots+{\tt w}_j-{\tt w}_0 e^{-\sqrt{2}({\tt w}_1+\dots+{\tt w}_j)}-|A_\Sigma|^{-2}\Delta_\Sigma {\tt w}_j))\right).
\end{aligned}
\end{equation*}

Next, we prove that for any $j\geq 0$, \eqref{err_j} holds true. The cases $j=0$ and $j=1$ are already proven. 

\medskip
Proceeding inductively, for $j\ge 1$ we assume that $E_\delta({\tt v}_j)=\delta \Delta_\Sigma {\tt w}_j$. A direct calculation yields that
\begin{equation}
\begin{aligned}
E_\delta({\tt v}_{j+1})&=E_\delta({\tt v}_j)+\delta\Delta_\Sigma {\tt w}_{j+1}+\delta|A_\Sigma|^2 {\tt w}_{j+1}-2a_\star e^{-\sqrt{2}{\tt v}_j}(e^{-\sqrt{2}{\tt w}_{j+1}}-1)\\
&=\delta\Delta_\Sigma {\tt w}_{j+1}+\delta\Delta_\Sigma {\tt w}_{j}+\delta|A_\Sigma|^2 {\tt w}_{j+1}-\delta|A_\Sigma|^2 {\tt w}_0 e^{-\sqrt{2}({\tt w}_1+\dots+{\tt w}_j)}(e^{-\sqrt{2}{\tt w}_{j+1}}-1).
\end{aligned}
\end{equation}

From \eqref{equivrecursidefwj}, we conclude that $E_\delta({\tt w}_{j+1})=\delta \Delta_\Sigma {\tt w}_{j+1}$. This proves the inductive steps and concludes the proof of \eqref{err_j}.

\medskip
Write $w_j(s)={\tt w}_j(\pe)$ for $s=s(\pe)$ and for $j\in \mathbb{N}$. Next, we show that for any $j\in \mathbb{N}$ and any $i \in \mathbb{N}\cup \{0\}$, \begin{equation}
\label{est_derivative_wj}
|\partial^{(i)}_s w_j(s)|\le \frac{C}{(s+1)^i(\log(s+2))^{\frac{j}{2}}|\log\delta|^{\frac{j}{2}}} \quad \hbox{for any} \quad s \geq 0.
\end{equation}

\medskip

From \eqref{est_derivatives_w0}, estimate \eqref{est_derivative_wj} holds true for $j=0$ and for any $i\ge 1$. Next, assume that $w_1, \ldots, w_j$ satisfy \eqref{est_derivative_wj}. We  prove that \eqref{est_derivative_wj} holds true also for $j+1$.

\medskip
Differentiating the  third equation in \eqref{recursivedefwj} and using \eqref{W_differentiation},
\begin{equation}
\label{eq_first_der}
\begin{aligned}
\partial_s(\sqrt{2}w_{j+1})&=-\frac{a_j}{1+W(a_{j} e^{a_{j}+b_{j}})}\partial_s \left(\frac{b_{j}}{a_{j}}\right)+\frac{\partial_s a_{j}}{a_{j}}\frac{\sqrt{2}w_{j+1}}{1+W(a_{j} e^{a_{j}+b_{j}})}\\
&=-\frac{a_{j}}{1+W(a_{j} e^{a_{j}+b_{j}})}\left(\frac{\partial_s b_{j}}{a_{j}}-\frac{\partial_s a_{j}}{a_{j}^2}b_j\right)+\frac{\partial_s a_{j}}{a_{j}}\frac{\sqrt{2}w_{j+1}}{1+W(a_{j} e^{a_{j}+b_{j}})}
\end{aligned}
\end{equation}
for $j\ge 1$.

\medskip
Since \eqref{est_derivative_wj} holds true for $w_j$, for some constant $C>0$, depending only on $j$ and for any $s\geq 0$,
\begin{equation}\label{estimatebj}
|b_{j}|\le \frac{C}{\log(s+2)^{\frac{j}{2}}|\log\delta|^{\frac{j}{2}}}.
\end{equation}

On the other hand, since we are assuming that $w_1, \ldots, w_j$ satisfy \eqref{est_derivative_wj} and noticing that
$$
\begin{aligned}
a_j&=a_{j-2}e^{-\sqrt{2}(w_{j-1}+w_j)}\\
&=a_0e^{-\sqrt{2}(w_1+ \cdots + w_j)},
\end{aligned}
$$
then for some constant $c>0$, that is  independent of $\delta>0$ small, 
\begin{equation}\label{estimateaj}
a_{j}\ge c a_0.
\end{equation}

Putting together \eqref{estimatebj} and \eqref{estimateaj}, we find that for any $s\geq 0$,
\begin{equation}
\label{est_ba}
\left|\frac{b_{j}}{a_{j}}\right|\le \frac{C}{\log(s+2)^{\frac{j+1}{2}}|\log\delta|^{\frac{j+1}{2}}},
\end{equation}
where again $C>0$ is a constant independent of $\delta>0$.

\medskip
Using \eqref{est_W_large_a}, \eqref{recursivedefwj} and \eqref{est_ba}, for any $s\geq 0$,
\begin{equation}\label{EST0}
|w_{j+1}| \leq \frac{C}{\log(s+2)^{\frac{j+1}{2}}|\log\delta|^{\frac{j+1}{2}}}.
\end{equation}

As a by-product of the previous analysis, we also find from \eqref{recursiveApproxslnJacToda}and \eqref{asymptw0infty} that for any $j\ge 0$, \eqref{est_approx_sol} holds true.

\medskip

Using an induction procedure over $j$ and \eqref{recursivedefwj}, it can be proven that there exists a constant $C$ depending on $j$, but not on $\delta>0$ such that 
\begin{equation}\label{EST1}
\begin{aligned}
|\partial_s b_{j}| &\le C (s^{2}|\partial^3_s w_{j}|+s|\partial^2_s w_{j}|+|\partial_s w_{j}|)\\
&\le \frac{C}{(s+1)(\log(s+2))^{\frac{j}{2}}|\log\delta|^{\frac{j}{2}}}
\end{aligned}
\end{equation}
and
\begin{equation}\label{EST2}
\begin{aligned}
\left|\frac{\partial_s a_j}{a_j}\right|&=\left|\frac{\partial_s a_0}{a_0}-\sqrt{2}\partial_s(w_1+\dots+w_j)\right|\\
&\le \frac{C}{(s+1)(\log (s+2))^{\frac{1}{2}}|\log\delta|^{\frac{1}{2}}}.
\end{aligned}
\end{equation}

Putting together \eqref{eq_first_der}, \eqref{est_ba}, \eqref{EST0}, \eqref{EST1} and \eqref{EST2}, we find that \eqref{est_derivative_wj} holds true for $i=1$.

\medskip
We finish the proof of \eqref{est_derivative_wj} for any $i\geq 1$ by differentiating \eqref{eq_first_der}, using the result in \cite{HMY} and performing an inductive procedure over $i$,  with an arbitrary, but fixed $j$.

\medskip
Estimate \eqref{est_err_j} readily follows from \eqref{coeffifirstderatzero}, \eqref{coeffifirstderatinfty},  \eqref{est_derivative_wj} and the fact that the functions we are dealing with are even and smooth. This completes the proof of the lemma.
\end{proof}

\medskip

\subsection{The linearised Jacobi-Toda operator}

Let $j \in \mathbb{N}$ be fixed, to be specified later. Consider the function ${\tt v}_j({\pe})=v_j(s({\tt p}))$ for ${\tt p}\in \Sigma$, defined in \eqref{recursiveApproxslnJacToda}.

\medskip
In this part, we follow the conventions and notations from subsections \ref{subsec_EF_Jacobi}, \ref{subsec_Jacobi_fields} and \ref{subsec_Jacobi_lin}. We study solvability theory for the linear problem
\begin{equation}
\label{eq_lin_JT_Sigma}
\delta J_\Sigma {\tt q}+2\sqrt{2}a_\star e^{-\sqrt{2}{\tt v}_j}{\tt q}=\delta {\tt f} \quad \hbox{in}\quad\Sigma,
\end{equation}
where ${\tt f}:\Sigma\to\R$ is continuous and $O(m)\times O(n)$-invariant, i.e. ${\tt f}(\pe)=f(s)$ with $f:\R\to\R$ continuous and even.

The topology of ${\tt f}$ is motivated by the behavior of the error $E_{\delta}({\tt v}_j)$ described in \eqref{est_err_j}. Thus, we introduce the following functional analytic setting.

\medskip
For $\beta \in (0,1)$, $\mu>0$ and $\varrho\in\R$, we introduce the norms
\begin{equation}\label{decayingLinftynorm}
\|{\tt f}\|_{*,\mu,\varrho}:=\sup_{\pe\in\Sigma}(s(\pe)^2+2)^{\frac{\mu}{2}} \log(s(\pe)+2)^\varrho\|{\tt f}\|_{L^\infty(B_1(\pe))}.
\end{equation}

\begin{equation}\label{decayingHoldernorm}
\|{\tt f}\|_{\mathcal{D}^{0,\beta}_{\mu,\varrho}(\Sigma)}:=\sup_{\pe\in\Sigma}(s(\pe)^2+2)^{\frac{\mu}{2}}\log(s(\pe)+2)^\varrho\|{\tt f}\|_{C^{0,\beta}(B_1(\pe))}
\end{equation}
and we consider the Banach space $\mathcal{D}^{0,\beta}_{\mu,\varrho}(\Sigma)$ defined as the space of $O(m)\times O(n)-$invariant functions ${\tt f}\in C^{0,\beta}_{loc}(\Sigma)$ for which the norm
\begin{equation}\label{functionspacerighthandsideJT}
\|{\tt f}\|_{\mathcal{D}^{0,\beta}_{\mu,\varrho}(\Sigma)}< \infty.
\end{equation}

We also consider the Banach space $\mathcal{D}^{2,\beta}_{\mu,\varrho}(\Sigma)$ defined as the space of $O(m)\times O(n)-$invariant functions ${\tt q}\in C^{2,\beta}_{loc}(\Sigma)$ for which the norm
\begin{equation}\label{eq_norm_q(s)}
\|{\tt q}\|_{\mathcal{D}^{2,\beta}_{\mu,\varrho}(\Sigma)}:=\sigma^{-1}\|D^2_\Sigma{\tt q}\|_{\mathcal{D}^{0,\beta}_{\mu+2,\varrho-1}(\Sigma)}+\|\nabla_{\Sigma}{\tt q}\|_{*,\mu +1,\varrho}+\|{\tt q}\|_{*,\mu,\varrho}<\infty.
\end{equation}

\medskip
The following proposition shows that \eqref{eq_lin_JT_Sigma} has an inverse in this functional analytic setting and also it accounts for its size respect to $\sigma$, for $\sigma$ large.

\begin{proposition}\label{propinverselinearJT}
Let $\beta \in (0,1)$ and $\varrho>\frac{3}{4}$. There exist $\sigma_0>0$ and a constant $C>0$ such that for any $\sigma>\sigma_0$ and any ${\tt f}\in \mathcal{D}^{0,\beta}_{2,\varrho}(\Sigma)$, equation \eqref{eq_lin_JT_Sigma} has a solution ${\tt q}:={\tt F}_1({\tt f})\in\mathcal{D}^{2,\beta}_{2,\varrho-\frac{1}{2}}(\Sigma)$ satisfying the estimate
\begin{equation}\label{prop_eq_lin_JT_Sigma}
\|{\tt q}\|_{\mathcal{D}^{2,\beta}_{0,\varrho-\frac{1}{2}}(\Sigma)}\le  c\red{\sigma^{\frac{3}{4}}\log(\sigma)}\|{\tt f}\|_{\mathcal{D}^{0,\beta}_{2,\varrho}(\Sigma)}.
\end{equation}
\end{proposition}

We devote the rest of the section to the proof of Proposition \ref{propinverselinearJT}, which will be carried out in several subsections.

\subsection{The Emden-Fowler change of variables}\label{subs_EF}
As in subsection \ref{subsec_EF_Jacobi}, we first write \eqref{eq_lin_JT_Sigma} in suitable coordinates. Since in our framework we deal with symmetric functions, we write ${\tt f}(\pe)=f(s)$ and ${\tt q}(\pe)=q(s)$ so that \eqref{eq_lin_JT_Sigma} reads as 
\begin{equation}
\label{eq_g_s}
\partial^2_s q+\alpha(s)\partial_s q+\beta(s)(1+\sqrt{2}w_0 e^{-\sqrt{2}(v_j-w_0)})q=f \quad \hbox{in} \quad \R,
\end{equation}
where we recall that $v_j$ is defined and estimated in Lemma \ref{lemmaapproximateslnJacToda} and also
$$
\beta(s):=|A_\Sigma|^2,\qquad\alpha(s):=(m-1)\frac{a'}{a}+(n-1)\frac{b'}{b} \quad \hbox{for} \quad  s\in \R.
$$

Next, set $\sigma:= \log\bigg(\frac{2\sqrt{2}a_\star}{\delta}\bigg)$. From \eqref{asympLambertfnzeroinfty},\eqref{est_approx_sol} and \eqref{asymptw0infty} in Lemma \ref{lemmaapproximateslnJacToda}  we find that
\begin{equation}\label{asymptvjinfty}
v_j(s)=
\frac{1}{\sqrt{2}}\sigma+ \frac{1}{\sqrt{2}}\log\big(\beta^{-1}(s)\big)- \frac{1}{\sqrt{2}}\log \bigg(\sigma + \log\big(\beta^{-1}(s)\big)\bigg)+  o\left(1\right)
\end{equation}
for any $s\in \R$ and as a corollary of the proof of Lemma \ref{lemmaapproximateslnJacToda}, this identity can be differentiated in $s$.

\medskip
We take advantage of the symmetric setting so that we study equation \eqref{eq_g_s} for $s\geq 0$ imposing the condition that $\partial_s q(0)=0$. 

It is also convenient to work again with the Emden-Fowler change of variable $s=e^t$ for $t\in \R$. Recall from \eqref{weightsJTEmdenFowler} that 
\begin{equation}\label{weightsJTEmdenFowlersec3}
\tilde{\alpha}(t):=\alpha(e^t)e^t-1, \qquad \tilde{\beta}(t):= \beta(e^t)e^{2t}
\end{equation}
and
\begin{equation}\label{Cancellingfirstderterm11}
p(t):=\exp \bigg ({-\int_0^t \frac{\tilde{\alpha}(\tau)}{2}d\tau}\bigg) \quad \hbox{for} \quad t\in\R,
\end{equation}
with $p(t)$ solving the ODE
\begin{equation}\label{Cancellingfirstderterm2}
2\frac{\partial_t p}{p}+\tilde{\alpha}(t)=0.
\end{equation}

Next, we denote
\begin{equation}\label{deffirstwegithJT}
\tilde{w}(t):=1+\sqrt{2}w_0(e^t) e^{-\sqrt{2}(v_j(e^t)-w_0(e^t))} \quad \hbox{for} \quad t\in\R
\end{equation}
and 
\begin{equation}\label{defMainweightJT}
Q:=\frac{\partial^2_t p}{p}+\tilde{\alpha}\frac{\partial_t p}{p}+\tilde{w}\tilde{\beta} \quad \hbox{in}\quad  \R
\end{equation}

Writing $q(s)=p(t)u(t)$ for $t=\log(s)$ and $s>0$, equation \eqref{eq_g_s} becomes
\begin{equation}
\label{eq_g_t}
\partial^2_t u+Q(t)u=\tilde{f} \quad \hbox{in} \quad \R,
\end{equation}
where as in \eqref{righthandandpotential},
\begin{equation}\label{defftilde}
\tilde{f}(t):=\frac{e^{2t}}{p(t)}f(e^t).
\end{equation}

Next, we solve equation \eqref{eq_g_t} with the help of the variations of parameters formula. Thus, we must study the solutions of the homogeneous equation 
\begin{equation}\label{HomogeneousLinearJTEF}
\partial^2_{t}v+Q(t)v=0 \quad \hbox{in} \quad \R.
\end{equation}

The first step in the study of solutions of \eqref{HomogeneousLinearJTEF} consists in analysing the asymptotic behavior of the potential $Q(t)$ in $\R$.

\medskip
In what follows we will make extensive use of the asymptotic behaviors of $\tilde{\alpha}(t)$, $\tilde{\beta}(t)$ and $p(t)$ in \eqref{alphatildeasymptotics}, \eqref{betatildeasymptotics} and \eqref{asymptoticsp} respectively. We also recall that these asymptotic behaviors can be differentiated in the variable $t$.

\medskip
On the other hand, from \eqref{asymptvjinfty},
\begin{equation}\label{wtildezeroasymptotics}
\tilde{w}(t)=\left\{
\begin{aligned}
\sigma + \mathcal{O}(1), & \quad \hbox{for}\quad t\leq T_0\\
\sigma + 2 t + \mathcal{O}(\ln(t)), & \quad \hbox{for}\quad t\geq T_1
\end{aligned}
\right.
\end{equation}
for $T_0<0<T_1$ as in \eqref{T0T1}, both indepedent of $\sigma$. Also, relations in \eqref{wtildezeroasymptotics} can be differentiated in $t$.

Putting together  \eqref{Cancellingfirstderterm1}, \eqref{alphatildeasymptotics}, \eqref{betatildeasymptotics}, \eqref{asymptoticsp} and \eqref{wtildezeroasymptotics}, we find that
\begin{equation}\label{Qasymptotics}
Q(t)=\left\{
\begin{aligned}
- \frac{(n-2)^2}{4} + c_0\sigma e^{2t}  + \mathcal{O}(e^{2t}) + \mathcal{O}(\sigma e^{4t}), & \quad \hbox{for}\quad t\leq T_0\\
(N-1)[\sigma + 2t] + \mathcal{O}(\ln(t))+\mathcal{O}(\sigma e^{-t})+ \mathcal{O}(te^{-t}), & \quad \hbox{for}\quad t\geq T_1
\end{aligned}
\right.
\end{equation}
and this relations can be differentiated $t$.

\medskip
From \eqref{Qasymptotics}, notice that the potential $Q(t)$ has three different qualitative regimes regarding its sign. We describe next the solutions $v(t),\tilde{v}(t)$ of \eqref{HomogeneousLinearJTEF} in each of these regions.

\medskip
Recall from \eqref{betaasymptotics} and  \eqref{capitallowercaselambda} that
$$
c_0>0 \quad \hbox{and} \quad  \lambda:= \frac{n-2}{2}
$$
and observe from \eqref{Qasymptotics} that for $\sigma$ and $|T_0|$ large, but $T_0$ still independent of $\sigma$,   
$$
\begin{aligned}
Q(t)=&-\lambda^2 + c_0\sigma e^{2t}\Bigl(1 + \mathcal{O}(\sigma^{-1}) + \mathcal{O}(e^{2t})\Bigr)\\
\partial_t Q(t)= &2c_0\sigma e^{2t}\Bigl(1 + \mathcal{O}(\sigma^{-1}) + \mathcal{O}(e^{2t})\Bigr)
\end{aligned}
$$
for $t<T_0$. Therefore, we may assume that for some constants $C_1,C_2,c_1,c_2>0$,
\begin{equation}\label{summaryregimesQ}
\left\{
\begin{aligned}
-\lambda^2 + c_0\sigma e^{2t} & \le Q(t) \le -\lambda^2 + 2c_0 \sigma e^{2t},&\qquad \partial_t Q>0, \qquad  t\le T_0\\
0<  c_1\sigma&\le Q(t) \le   C_1\sigma,& \qquad T_0< t\le T_1\\
c_2(\sigma+t) & \leq Q(t) \leq C_2(\sigma+t), \qquad & t>T_1.
\end{aligned}
\right.
\end{equation}

Notice finally from \eqref{defftilde}, that $\tilde{f}(t)$, the right side in \eqref{eq_g_t}, satisfies the estimate 
\begin{equation}\label{estimateftildeJT}
|\tilde{f}(t)|\leq C\left\{
\begin{aligned}
e^{(\frac{n}{2}+1)t}\|{\tt f}\|_{\mathcal{D}^{0,\beta}_{2,\varrho}(\Sigma)} &\quad \hbox{for} \quad  t\leq T_0,\\
e^{\frac{N-2}{2}t}t^{-\varrho}\|{\tt f}\|_{\mathcal{D}^{0,\beta}_{2,\varrho}(\Sigma)} &\quad \hbox{for} \quad  t \geq T_1.
\end{aligned}
\right.
\end{equation}

\subsection{Estimates for negative potential}\label{subs_est_neg_pot}

Write 
\begin{equation}
\label{hom_eq}
Q(t):=-(\lambda^2+q(t)) \quad \hbox{for} \quad t\in \R.
\end{equation}

Next, let $M>0$ be fixed, to be specified later. Set $t_{\sigma}=-\frac{1}{2}\log\bigl(\sigma\bigr) - M$ and consider the interval $I_0:=(-\infty, t_{\sigma})$. Notice from \eqref{summaryregimesQ}that 
\begin{equation}\label{negativepotentialsec3}
-\lambda^2 \leq Q(t) \leq  -\lambda^2 + 2c_0e^{-2M} \quad \hbox{in} \quad I_0.
\end{equation}

Let $v_0(t)$ be the solution of \eqref{HomogeneousLinearJTEF} satisfying the conditions
$$
v(t_{\sigma})=e^{\lambda M}\sigma^{\frac{\lambda}{2}} \quad \hbox{and} \quad \partial_t v(t_\sigma)=-\lambda e^{\lambda M}\sigma^{\frac{\lambda}{2}}.
$$

We proceed next, by writing $v_0(t)=e^{-\lambda t}x(t)$. and by setting $y(t):=e^{-2\lambda t} x'(t)$, we find that $x(t)$ and $y(t)$ must solve the system
\begin{equation}
\label{system_xy}
\begin{cases}
x'=e^{2\lambda t}y\\
y'=q(t)e^{-2\lambda t}x
\end{cases}
\end{equation}
with the initial conditions for $(x,y)$,
\begin{equation}
\label{initial_values}
x(t_\sigma)=1\quad  \hbox{and} \quad y(t_\sigma)=0.
\end{equation}

Integrating \eqref{system_xy}, we find that for $t\in I_0$,
\begin{equation}
\label{int_formulas_xy}
x(t)=1-\int_t^{t_\sigma}e^{2\lambda\tau}y(\tau)d\tau, \qquad y(t)=-\int_t^{t_\sigma}q(\tau)e^{-2\lambda\tau}x(\tau)d\tau
\end{equation}
and fom Fubini's Theorem and since $n>2$,
\begin{equation}\notag
x(t)=1+\int_t^{t_\sigma} x(\zeta)q(\zeta)\frac{1-e^{-2\lambda(\zeta-t)}}{2\lambda}d\zeta.
\end{equation}

We estimate for $t\leq t_{\sigma}$,
$$
\begin{aligned}
|x(t)| & \leq 1 + \int_{t}^{t_{\sigma}} |x(\zeta)||q(\zeta)| \frac{1-e^{-2\lambda(\zeta-t)}}{2\lambda}d\zeta,\\
&\leq 1 + \int_{t}^{t_{\sigma}} |x(\zeta)|\frac{|q(\zeta)|}{2\lambda}d\zeta.
\end{aligned}
$$

Using the Gronwall inequality and \eqref{summaryregimesQ},
\begin{equation*}
\begin{aligned}
|x(t)|&\le\exp\bigg(\int_t^{t_\sigma}\frac{|q(\zeta)|}{2\lambda}d\zeta\bigg)\le\exp\bigg(c\sigma\int_{-\infty}^{t_\sigma}e^{2\zeta}\bigg)\\
&\le \exp\bigg(\frac{c}{2}e^{-2M}\bigg)<\infty
\end{aligned}
\end{equation*}
for $t\leq t_{\sigma}$. We conclude that $x\in L^{\infty}(I_0)$ and by choosing $M$ sufficiently large, but independent of $\sigma$, 
$$
\|x\|_{L^{\infty}(I_0)}\leq 2,
$$ 
which is uniformly bounded in $\sigma$.

\medskip

Next, we use \eqref{summaryregimesQ} and the equation for $y(t)$ in \eqref{int_formulas_xy} to estimate
\begin{equation}\notag
|y(t)|\le c\sigma \int_t^{t_\sigma} e^{2(1-\lambda)\tau}d\tau\le
\begin{cases}
c\sigma\frac{e^{2(1-\lambda)t_{\sigma}}-e^{2(1-\lambda)t}}{2(1-\lambda)}, \qquad &\lambda\ne 1\\
c\sigma (t_\sigma-t), \qquad&\lambda=1.
\end{cases}
\end{equation}

Using again \eqref{summaryregimesQ} and the first equation in system \eqref{system_xy}, we get the estimate 
\begin{equation}
\label{est_x'}
|x'(t)|\le
c\sigma e^{2t_{\sigma}}
\left\{ 
\begin{aligned}
&e^{2\lambda (t-t_{\sigma})}-e^{2t}, \qquad &\lambda &\ne 1\\
&e^{2(t-t_{\sigma})}(t_\sigma-t), \qquad &\lambda &=1, \,
\end{aligned}
\right.
\end{equation}
for $t \in I_0$. Since, 
\begin{equation*}
|x(t)-1|=|x(t)-x(t_\sigma)|\le\int^{t_\sigma}_t |x'(\tau)|d\tau\le \int_{-\infty}^{t_{\sigma}} |x'(\tau)|d\tau,
\end{equation*}
by integrating \eqref{est_x'}, using integration by parts for the case $\lambda=1$ and taking $c$ larger if necessary but fixed, we get that for any $t<t_{\sigma}$,
\begin{equation}\label{estimatex(t)}
|x(t)-1|\leq c\sigma e^{2t_{\sigma}}
\left\{ 
\begin{aligned}
&e^{2\lambda (t-t_{\sigma})}-e^{2t}, \qquad &\lambda &\ne 1\\
&e^{2(t-t_{\sigma})}(t_\sigma-t), \qquad &\lambda &=1. \,
\end{aligned}
\right.
\end{equation}

By an elementary maximisation argument, it follows that
\begin{equation}
\label{est_x'1}
|x(t)-1|\le c\sigma e^{2t_{\sigma}} \leq ce^{-2M} \quad \hbox{for} \quad t\in I_0.
\end{equation}

Now, we fix $M$ large enough so that $x\geq\frac{1}{2}$ in $I_0$. Using \eqref{system_xy} and the fact that $q(t)<0$, we conclude that $y'<0$. 

\medskip
Proceeding in a similar fashion and since $y(t_\sigma)=0$, we have $y>0$ in $(-\infty, t_{\sigma})$ and therefore $x'>0$.

\medskip
Recall that $v_0=e^{-\lambda t}x$, so that $\partial_t v_0=(-\lambda x+x')e^{-\lambda t}$ and observe that expressions \eqref{est_x'}, \eqref{est_x'1} and \eqref{estimatex(t)} yield the inequalities
\begin{equation}
\label{est_v}
(1-c e^{-2M}) e^{-\lambda t}\le v_0(t)\le(1+ce^{-2M}) e^{-\lambda t}
\end{equation}
and
\begin{equation}
\label{est_v'}
-\lambda(1+ce^{-2M}) e^{-\lambda t}\le \partial_t v_0(t)\le -\lambda (1-ce^{-2M}) e^{-\lambda t}
\end{equation}
for any $t\in I_0$.

\medskip

We select the second linearly independent solution $\tilde{v}_0(t)$ of \eqref{HomogeneousLinearJTEF} by setting
\begin{equation}\label{secslnHJTEF}
\tilde{v}_0(t):=v_0(t)\int_{-\infty}^t v_0^{-2}(\tau)d\tau \quad \hbox{for} \quad t\in (-\infty,t_{\infty})
\end{equation}
and directly from \eqref{secslnHJTEF} we find that
\begin{equation}
\label{est_tildev}
\begin{aligned}
&\frac{1}{2\lambda}(1-ce^{-2M})e^{\lambda t}\le \tilde{v}_0(t)\le \frac{1}{2\lambda}(1+ce^{-2M})e^{\lambda t},\\
&\frac{1}{2}(1-ce^{-2M})e^{\lambda t}\le \partial_t\tilde{v}_0(t)\le \frac{1}{2}(1+ce^{-2M})e^{\lambda t}
\end{aligned}
\end{equation}
for $t\in I_0$. 

\medskip
Observe that 
$$
\begin{aligned}
\tilde{v}_0(t_{\sigma})&=\frac{1}{2\lambda}\sigma^{-\frac{\lambda}{2}}\big(1 + \mathcal{O}(e^{-2M})\big)\\\partial_t \tilde{v}_0(t_\sigma)&=\frac{1}{2}\sigma^{-\frac{\lambda}{2}}\big(1 + \mathcal{O}(e^{-2M})\big)
\end{aligned}
$$
so that the Wronskian of $v_0(t)$ and $\tilde{v}_(t)$ is given by
\begin{eqnarray}\label{WronskianEF}
\begin{aligned}
W(t)&=v_0(t_\sigma)\partial_t\tilde{v}_0(t_\sigma)-\tilde{v}_0(t_\sigma)\partial_t v_0(t_\sigma)\\
&=\sigma^{\frac{\lambda}{2}}\frac{\sigma^{-\frac{\lambda}{2}}}{2}
-\frac{\sigma^{-\frac{\lambda}{2}}}{2\lambda}(-\lambda\sigma^{\frac{\lambda}{2}})+\mathcal{O}(e^{-2M})
\\
&=1+\mathcal{O}(e^{-2M}).
\end{aligned}
\end{eqnarray}

We finished this section by  solving \eqref{eq_g_t} in $I_0$ via the integral formula
$$
u_0(t):= v_0(t)\int_{-\infty}^{t}W^{-1}(\tau)\tilde{v}_0(\tau)\tilde{f}(\tau)d\tau - \tilde{v}_0(t)\int_{-\infty}^{t}W^{-1}(\tau){v}_0(\tau)\tilde{f}(\tau)d\tau
$$
for $t\in I_0$ from where it is direct to verify that 
\begin{equation}\label{eq:estimatesU0I0}
|u_0(t)|+ |\partial_t u_0(t)| \leq c\|{\tt f}\|_{\mathcal{D}^{0,\beta}_{2,\varrho}(\Sigma)}e^{\bigl(\frac{n}{2}+1\bigr)t} \quad t\in I_0.
\end{equation}

\subsection{Estimates for the first transition region}\label{subs_trans}
Next, we fix $m>0$ small and independent of $\sigma$, so that 
$$
-\lambda + c_0 e^{-2M}\approx Q(t_\sigma) < m^2.
$$ 

Using that since $Q(t)$ is monotone increasing, we let $T_{\sigma}\in (t_{\sigma}, 0)$ be the unique solution of $Q(T_{\sigma})=m^2$. From \eqref{Qasymptotics} we find that
\begin{equation}\label{eq:asymptotics_Tsigmas}
T_{\sigma}\approx -\frac{1}{2}\log(\sigma) + \frac{1}{2}\log\Bigl(\frac{\lambda^2+ m^2}{c_0}\Bigr).
\end{equation}

In this subsection we consider the interval $I_1:=(t_{\sigma},T_{\sigma})$. Observe that 
$$
{\rm lenght}(I_1)=M-\frac{1}{2}\log\Bigl(\frac{\lambda^2+m^2}{c_0}\Bigr)=\mathcal{O}(1) \quad \hbox{independent of $\sigma$}.
$$ 

Let $v(t)$ be any arbitrary solution of \eqref{HomogeneousLinearJTEF} and set $x(t)=v(t)$ and $y(t)=\partial_t v(t)$ for $t\in I_1$, with
$$
x(t_{\sigma})=x_0 \quad \hbox{and} \quad y(t_{\sigma})=y_0. 
$$

\medskip
Set 
$$
{\tt A}(t):= 
\left(
\begin{array}{cc}
0& 1\\
-Q(t) & 0 
\end{array}
\right) \quad \hbox{for} \quad t\in I_1
$$
and observe that 
$$
\left(
\begin{array}{c}
x'\\
y'
\end{array}
\right) = A(t)\left(
\begin{array}{c}
x\\
y
\end{array}
\right).
$$

Observe also that for $t\in I_1$
$$
\begin{aligned}
\|{\tt A}(t)\|\leq & \sup \limits_{\|{\tt p}\|=1}|{\tt A}(t){\tt p}|\\
\leq & \sqrt{1+Q^2(t)}\\
\leq &C, 
\end{aligned}
$$
where $C>0$ is independent of $\sigma$.

Consequently, 
$$
\begin{aligned}
|(x(t),y(t))| &\leq |(x(t_{\sigma}),y(t_{\sigma}))| + \int_{t_{\sigma}}^t \|{\tt A}(\tau)\| | (x(\tau),y(\tau))|d\tau\\
& \leq |(x(t_{\sigma}),y(t_{\sigma}))| + C\int_{t_{\sigma}}^t| (x(\tau),y(\tau))|d\tau.
\end{aligned}
$$

The Gronwall inequality yields that 
$$
\begin{aligned}
|(x(t),y(t))| &\leq |(x(t_{\sigma}),y(t_{\sigma}))|\exp\left(C\int_{t_{\sigma}}^t d\tau\right)\\
& \leq |(x(t_{\sigma}),y(t_{\sigma}))|\exp\bigg(C(T_{\sigma}-t_{\sigma})\bigg)\\
&\leq \tilde{C} |(x(t_{\sigma}),y(t_{\sigma}))|.
\end{aligned}
$$

We conclude that for some constant $C>0$ independent of $\sigma$, 
$$
|v(t)| + |\partial_t v(t)| \leq C\Bigl(|v(t_{\sigma})| + |\partial_t v(t_{\sigma})|\Bigl) \quad \hbox{for} \quad t\in I_1.
$$

Next, we fix solutions $v_1(t)$ and $\tilde{v}_1(t)$ of \eqref{HomogeneousLinearJTEF} in $I_1$, satisfying the initial conditions
$$
\begin{aligned}
v_1(t_{\sigma})=& 1 \quad \partial_{t}v_1(t_{\sigma})=&0\\
\tilde{v}_1(t_{\sigma})=& 0 \quad \partial_{t}\tilde{v}_1(t_{\sigma})=&1
\end{aligned}
$$
so that the Wronski determinant $W(t)=1$ in $I_1$ and 
$$
|v_1(t)| + |\tilde{v}_1(t)| + |\partial_t v_1(t)|+|\partial_t \tilde{v}_1(t)| \leq C \quad \hbox{for} \quad t\in I_1.
$$

We solve equation \eqref{eq_g_t} in $I_1$ via the formula
$$
\begin{aligned}
u_1(t)=& u_0(t_{\sigma})v_1(t) + \partial_t u_0(t_{\sigma})\tilde{v}_1(t)\\
& + v_1(t)\int_{t_{\sigma}}^t \tilde{v}_1(\tau)\tilde{f}(\tau)d\tau - \tilde{v}_1(t)\int_{t_{\sigma}}^t {v}_1(\tau)\tilde{f}(\tau)d\tau 
\end{aligned}$$
for $t\in I_1$.

From \eqref{eq:estimatesU0I0} we find that 
$$
|u_1(t_{\sigma})|+ |\partial_t u_1(t_{\sigma})| \leq c\|{\tt f}\|_{\mathcal{D}^{0,\beta}_{2\varrho}(\Sigma)}\sigma^{-\frac{n+2}{4}}
$$
and hence we find that 
$$
\begin{aligned}
|u_1(t)|+ |\partial_t u_1(t)| \leq & c\|{\tt f}\|_{\mathcal{D}^{0,\beta}_{2,\varrho}(\Sigma)}\Bigr(\sigma^{\bigl(-\frac{n+2}{4}\bigr)}+e^{\bigl(\frac{n}{2}+1\bigr)t}\Bigr)\\ 
\leq & c\|{\tt f}\|_{\mathcal{D}^{0,\beta}_{2,\varrho}(\Sigma)}\Bigr(\sigma^{\bigl(-\frac{n+2}{4}\bigr)}e^{-\bigl(\frac{n}{2}+1\bigr)T_{\sigma}}+1\Bigr)e^{\bigl(\frac{n}{2}+1\bigr)t}\\ 
\end{aligned}
$$
for $t\in I_1$. From our choice of $T_{\sigma}\sim -\frac{1}{2}\log(\sigma)$ we find that
\begin{equation}\label{eq:estimatesU1I1}
|u_1(t)|+ |\partial_t u_1(t)| \leq c\|{\tt f}\|_{D^{0,\beta}_{2,\varrho}(\Sigma)}e^{\bigl(\frac{n}{2}+1\bigr)t} \quad t\in I_1.
\end{equation}

\subsection{Estimates for the second transition region}\label{subs_transI}
In this subsection, we consider the interval $I_2:=(T_{\sigma},T_0)$, where the potential $Q(t)$ has the properties that 
$$
0<m^2<Q(t)\leq c_1\sigma, \qquad  \partial_t Q(t)>0 \quad \hbox{in} \quad I_2 \quad \hbox{see e.g. \eqref{summaryregimesQ}}. 
$$

Observe also that $${\rm lenght}(I_2)=T_0 - T_\sigma=\frac{1}{2}\log(\sigma) + \mathcal{K}$$
for some $\mathcal{K}>0$, independent of $\sigma$.

\medskip

As in the previous subsection, we study the solutions of \eqref{HomogeneousLinearJTEF} in $I_2$. Let $v(t)$ solve \eqref{HomogeneousLinearJTEF} and consider the Lyapunov energy
$$
H(t):=\frac{1}{Q(t)}\frac{\bigl(\partial_t v(t)\bigr)^2}{2} + \frac{v^2(t)}{2} \quad t\in I_2. 
$$

The functional $H(t)$ is differentiable with
$$
\partial_t H(t)=-\frac{\partial_t Q(t)}{Q^2(t)}\frac{\bigl(\partial_t v(t)\bigr)^2}{2} \leq 0 \quad \hbox{in} \quad I_2
$$
and thus 
$$
0 \leq H(t) \leq H(t_{\sigma})\leq c\bigl(|v(T_{\sigma})|^2 + |\partial_t v(T_\sigma)|^2 \bigr) \quad \hbox{for} \quad t\in I_2.
$$ 

\medskip
We estimate for $t\in I_2$
\begin{equation}\label{eq:estimate_generic_kernelI2}
|v(t)| \leq c\bigl(|v(T_{\sigma})| + |\partial_t v(T_\sigma)|\bigr),\qquad |\partial_t v(t)| \leq  c\sqrt{\sigma}\bigl(|v(T_{\sigma})| + |\partial_t v(T_\sigma)|\bigr). 
\end{equation}
 
Next, we proceed exactly in the same fashion as we did in the previous subsection, first by fixing two solutions $v_2(t)$ and $\tilde{v}_2(t)$ of \eqref{HomogeneousLinearJTEF} in $I_2$, satisfying the initial conditions
$$
\begin{aligned}
v_2(T_{\sigma})=& 1 \quad \partial_{t}v_2(T_{\sigma})=&0\\
\tilde{v}_2(T_{\sigma})=& 0 \quad \partial_{t}\tilde{v}_2(T_{\sigma})=&1
\end{aligned}
$$
with Wronskii determinant $W(t)=1$ in $I_2$. 

\medskip
Next,we set the variations of parameters formula
$$
\begin{aligned}
u_2(t)=& u_1(T_{\sigma})v_2(t) + \partial_t u_1(T_{\sigma})\tilde{v}_2(t)\\
& + v_2(t)\int_{T_{\sigma}}^t \tilde{v}_2(\tau)\tilde{f}(\tau)d\tau - \tilde{v}_2(t)\int_{T_{\sigma}}^t {v}_2(\tau)\tilde{f}(\tau)d\tau 
\end{aligned}$$
for $t\in I_2$.

\medskip
Again, proceeding similarly as 
we did for the estimate\eqref{eq:estimatesU1I1}, we obtain
\begin{equation}\label{eq:estimatesU2I2}
|u_2(t)|+ |\partial_t u_2(t)| \leq c\sqrt{\sigma}\log(\sigma)\|{\tt f}\|_{\mathcal{D}^{0,\beta}_{2,\varrho}(\Sigma)}e^{\bigl(\frac{n}{2}+1\bigr)t} \quad t\in I_2.
\end{equation}

\medskip
\subsection{Estimates for the potential having linear growth }\label{subs_pos_pot}

Now, we solve equation \eqref{eq_g_t} in the final interval $I_3:=(T_0, \infty)$, where $Q(t)\ge c_1\sigma$. Let us then introduce the change of variable
$$
\xi(t)=\int_{T_0}^t Q(\tau)^{\frac{1}{2}}d\tau, \qquad\forall t\in I_3.
$$

Let $v(t)$ be an arbitrary solution of \eqref{HomogeneousLinearJTEF} in $I_3$. Write
\begin{equation}
\label{def_w}
v(t)=Q(t)^{-\frac{1}{4}}{\tt w}(\xi(t)) \quad \hbox{for} \quad t \in I_3 
\end{equation}
and observe from \eqref{HomogeneousLinearJTEF} that ${\tt w}$ must solve
\begin{equation}
\label{eq_w}
\partial^2_\xi {\tt w}+(1+{\tt V}(\xi)){\tt w}=0 \quad \hbox{in} \quad (0,\infty),
\end{equation}
where
$$
{\tt V}(\xi):=-\frac{\partial^2_t Q(t(\xi))}{4Q^2(t(\xi))}+\frac{5(\partial_t Q)^2(t(\xi))}{16Q^3(t(\xi))}.
$$

We find from \eqref{Qasymptotics} that $\partial_t Q$ and $\partial^2_t Q$ are bounded in $I_3$. Therefore, using again \eqref{Qasymptotics}, we estimate
\begin{equation}\label{V_L1-norm}
\begin{aligned}
\int_0^\infty |{\tt V}(\xi)|d\xi &=\int_{T_0}^\infty |{\tt V} (\xi(t))|Q^{\frac{1}{2}}(t)dt\\
&\le \int_{T_0}^{T_1}cdt+\int_{T_1}^\infty\frac{c}{(\sigma+t)^{\frac{3}{2}}}dt\\
&\le {\emph{K}}.
\end{aligned}
\end{equation}

Moreover, differentiating (\ref{def_w}) and evaluating at $t=T_0$, we get
$$
{\tt w}(0)=Q(T_0)^{\frac{1}{4}}v(T_0) \quad\hbox{and} \quad \partial_\xi {\tt w}(0)=\frac{\partial_t Q(T_0)}{4Q(T_0)^{\frac{5}{4}}}v(T_0)-\frac{\partial_t v(T_0)}{Q(T_0)^{\frac{1}{4}}},
$$
so that
\begin{equation}
\label{w_initial}
\begin{aligned}
c\sigma^\frac{1}{4}(|v(T_0)|+|\partial_t v(T_0)|)\leq |{\tt w}(0)|+|\partial_\xi {\tt w}(0)|& \le C\sigma^\frac{1}{4}(|v(T_0)|+|\partial_t v(T_0)|) \end{aligned}
\end{equation}

Now, we aim to estimate ${\tt w}$ and $\partial_\xi {\tt w}$ in $(0,\infty)$. In order to do so, we multipliy \eqref{eq_w} by $\partial_{\xi} {\tt w}$ to find that 
$$
\partial_{\xi}\big( \partial_{\xi } {\tt w}\big)^2 + \partial_{\xi }\big( {\tt w} \big)^2  = -2 {\tt V}(\xi){\tt w} \partial_{\xi}{\tt w}. 
$$

A direct integration yields that 
$$
\begin{aligned}
\big( \partial_{\xi } {\tt w}(\xi)\big)^2 + \big( {\tt w}(\xi)\big)^2  &= \int_{0}^{\xi}-2 {\tt V}(z){\tt w} \partial_{z}{\tt w}dz\\
& \leq \int_{0}^{\xi}|{\tt V}(z)|\bigg(\big( \partial_{\xi} {\tt w}\big)^2 + \big( {\tt w}\big)^2 \bigg)dz. 
\end{aligned}
$$

%In order to do so we rewrite equation (\ref{eq_w}) as a system
%\begin{equation}\notag
%\label{syst_aux_glob_def-}
%\begin{cases}
%\partial_\xi w=z\\
%%\partial_\xi z=-(1+{\tt V}(\xi))w
%\end{cases}
%\end{equation}
%and we introduce the rotation
%$$\Phi(\xi)=
%\begin{bmatrix}
%\cos(\xi) & \sin(\xi)\\
%-\sin(\xi) & \cos(\xi)
%\end{bmatrix} \quad \quad \hbox{for}%\quad \xi\in\R.%
%$$

%Observe that the vector function
%$$\zeta(\xi):=\Phi(\xi)^{-1}(w,z)^t$$
%satisfies the ODE
%$$\zeta'=\Phi(\xi)
%\begin{bmatrix}
%0 & 0\\
%-{\tt V}(\xi) & 0
%\end{bmatrix}
%\Phi^{-1}(\xi)\zeta.$$
%Integrating this last relation we get
%$$
%\zeta(\xi) = \zeta(0) + \int_{0}^\xi%\Phi^{-1}(s)%
%\begin{bmatrix}
%  0 & 0 \\
%  -q(s) & 0 
%\end{bmatrix}
%\Phi(s)\zeta(s) ds.
%$$

The Gronwall inequality, \eqref{V_L1-norm} and \eqref{w_initial} yield that
\begin{equation}
\label{est_w_dx}
\begin{aligned}
|\partial_\xi {\tt w}(\xi)|+|{\tt w}(\xi)| &\le c(|{\tt w}(0)|+|\partial_\xi {\tt w}(0)|)\quad \hbox{for}\quad \xi>0.
\end{aligned}
\end{equation}

\medskip
Pulling back the estimates, we estimate $v$ and $\partial_t v$,
\begin{equation}\label{AAA}
|v(t)|+ |\partial_t v(t)|\le c\sigma^{\frac{1}{4}}Q(t)^{-\frac{1}{4}}(|v(T_0)|+|\partial_t v(T_0)|)
\end{equation}
for any $t\in I_3$.

\medskip
Now we proceed as in the previous subsections by selecting two solutions $v_3(t)$ and $\tilde{v}_3(t)$ of \eqref{HomogeneousLinearJTEF} in $I_3$, with 
$$
\begin{aligned}
v_3(T_0)=& 1 \quad \partial_{t}v_3(T_0)=&0\\
\tilde{v}_3(T_0)=& 0 \quad \partial_{t}\tilde{v}_3(T_0)=&1
\end{aligned}
$$
with Wronskii determinant $W(t)=1$ in $I_3$. Clearly, $v_3(t)$ and $\tilde{v}_3(t)$ satisfy the estimate \eqref{AAA}. 

\medskip
Next, we solve equation \eqref{eq_g_t} in $I_3$ setting the variations of parameters formula
$$
\begin{aligned}
u_3(t)=& u_2(T_0)v_3(t) + \partial_t u_2(T_0)\tilde{v}_3(t)\\
& + v_3(t)\int_{T_{0}}^t \tilde{v}_3(\tau)\tilde{f}(\tau)d\tau - \tilde{v}_3(t)\int_{T_{0}}^t {v}_3(\tau)\tilde{f}(\tau)d\tau 
\end{aligned}$$
for $t\in I_3$.

\medskip
From \eqref{eq:estimatesU2I2} we find that 
$$
|u_2(t)|+ |\partial_t u_2(t)| \leq c\sqrt{\sigma}\log(\sigma)\|{\tt f}\|_{\mathcal{D}^{0,\beta}_{2,\varrho}(\Sigma)}.
$$

This estimate together with \eqref{estimateftildeJT} yields 
\begin{equation}\label{eq:estimatesU3I3}
|u_3(t)|+ |\partial_t u_3(t)| \leq c\sigma^\frac{3}{4}\log(\sigma)\|{\tt f}\|_{\mathcal{D}^{0,\beta}_{2,\varrho}(\Sigma)}e^{\frac{N-2}{2}t}t^{-\varrho +\frac{1}{2}} \quad t\in I_3.
\end{equation}

\subsection{Proof of Proposition \ref{propinverselinearJT} completed}\label{subs_rhs}

In this subsection we finished the detailed proof of Proposition \ref{propinverselinearJT}.

Given any interval $I\subset \R$, let $\chi_I:\R \to \R$ be the characteristic function of $I$. Consider the functions $u_j$ defined in the intervals $I_j$ for  $j=0,1,2,3$. It is direct to verify that the function 
$$
u(t)=\sum_{j=0}^3 \chi_{I_j}u_i(t) \quad t\in \R
$$
is a smooth solution of \eqref{eq_g_t}. In addition, from the estimates \eqref{eq:estimatesU0I0}, \eqref{eq:estimatesU1I1}, \eqref{eq:estimatesU2I2} and
\eqref{eq:estimatesU3I3}, the function $u(t)$ satisfies the estimate
\begin{equation}\notag
|u(t)|+|\partial_{t}u(t)|\le c\sigma^{\frac{3}{4}}\log(\sigma)\|{\tt f}\|_{\mathcal{D}^{0,\beta}_{2,\varrho}(\Sigma)}\left\{
\begin{aligned}
e^{\frac{n+2}{2}t}&, \qquad &t\le T_\sigma,\\
1&, \qquad & T_\sigma < t<T_0,\\
t^{\frac{1}{2}-\varrho}e^{\frac{N-2}{2}t}&, \qquad &t>T_0.
\end{aligned}
\right.
\end{equation}

Next, we estimate the norm in \eqref{eq_norm_q(s)} for ${\tt q}$. Setting ${q}(s):=p(t(s))u(t(s))$ for $s\geq 0$, we find that ${q}(s) \sim s^2$ as $s\to 0$, so that we can extend ${q}$ smoothly to $\R$.

\medskip 
From the previous analysis and setting ${\tt q}(\pe)=q(s)$ for $s=s(\pe)$, we find that 
\begin{equation}\label{inverseJT1}
\|{\tt q}\|_{*,0,\varrho - \frac{1}{2}} \leq  c\sigma^{\frac{3}{4}}\log(\sigma)\|{\tt f}\|_{\mathcal{D}^{0,\beta}_{2,\varrho}(\Sigma)}.
\end{equation}

On the other hand, since 
$$
\partial_s {q}(s)=\frac{1}{s}\bigg(\partial_t p(t(s))\, u(t(s))+p(t(s))\,\partial_t u(t(s))\bigg)
$$
with $t(s)=\log(s)$, we find that $\partial_s{q}(s) \sim s$ as $s \to 0$ and the gradient of the function ${\tt q}(\pe)=q(s)$ satisfies the estimate
\begin{equation}\label{inverseJT2}
\|\nabla_{\Sigma} {\tt q}\|_{*,1,\varrho -\frac{1}{2}} \leq c\sigma^{\frac{3}{4}}\log(\sigma)\|{\tt f}\|_{\mathcal{D}^{0,\beta}_{2,\varrho}(\Sigma)}.
\end{equation}

Since $\partial_s q(0)=0$, the even extension of $q$ yields a $C^2(\R)$ solution of \eqref{eq_g_s} so that ${\tt q}(\pe)=q(s)$ is a solution of \eqref{eq_lin_JT_Sigma} in $\Sigma$.

\medskip
On the other hand, from \eqref{eq_g_s}, \eqref{inverseJT1} and \eqref{inverseJT2},
\begin{equation}
\label{est_C2}
\sigma^{-1}\|D^2_\Sigma{\tt q}\|_{*,2,\varrho-\frac{3}{2}}
+\|\nabla_\Sigma{\tt q}\|_{1,\varrho-\frac{1}{2}}+\|{\tt q}\|_{*,0,\varrho-\frac{1}{2}}\le c\sigma^{\frac{3}{4} }\log(\sigma)\|{\tt f}\|_{\mathcal{D}^{0,\beta}_{2,\varrho}(\Sigma)}.
\end{equation}

We conclude the proof of the proposition as follows. Since the coefficient $\alpha(s)$ in ODE \eqref{eq_g_s} is singular at $s=0$, we need to pass through the elliptic PDE on $\Sigma$ in order to estimate the H\"{o}lder norm of the second derivative near $s=0$.
\\

From the standard local H\"{o}lder estimates  applied to the equation \eqref{eq_lin_JT_Sigma},
$$
\Delta_\Sigma{\tt q}+|A_\Sigma|^2{\tt q}={\tt f}-\sqrt{2}|A_{\Sigma}|^2{\tt w}_0 e^{-\sqrt{2}({\tt v}_j-{\tt w}_0)}{\tt q} \quad \hbox{in} \quad \Sigma, 
$$
we find that 
\begin{equation}\notag
\begin{aligned}
\|{\tt q}\|_{C^{2,\beta}(B_{\frac{3}{2}}(\zeta))}&\le c(\|{\tt q}\|_{L^\infty(B_2(\zeta))}+\|{\tt f}\|_{C^{0,\beta}(B_2(\zeta))}+\sigma\|{\tt q}\|_{C^{0,\beta}(B_2(\zeta))})\\
&\le c(\|{\tt q}\|_{L^\infty(B_2(\zeta))}+\|{\tt f}\|_{C^{0,\beta}(B_2(\zeta))}+\sigma\|\nabla_\Sigma{\tt q}\|_{L^\infty(B_2(\zeta))})\\&\le c\sigma^{\frac{7}{4}}\log(\sigma)\|{\tt f}\|_{\mathcal{D}^{0,\beta}_{2,\varrho}(\Sigma)}.
\end{aligned}
\end{equation}

This yields that, 
$$
\|{\tt q}\|_{\mathcal{D}^{2,\beta}_{0,\varrho-\frac{1}{2}}(\Sigma)}\le  c\sigma^{\frac{3}{4}}\log(\sigma)\|{\tt f}\|_{\mathcal{D}^{0,\beta}_{2,\varrho}(\Sigma)}
$$
and this completes the proof.

\subsection{A fixed point argument and the proof of Theorem \ref{th_Liouville}}

In this part, we use the linear theory studied in the previous subsections to prove Theorem \ref{th_Liouville} and the norms defined in \eqref{decayingHoldernorm} and \eqref{eq_norm_q(s)}. The ideas here will also be used in the the section 5.

\medskip

Let $\sigma_0>0$ be as in Proposition \ref{prop_eq_lin_JT_Sigma} and let $\sigma:= \log\bigg(\frac{2\sqrt{2}a_\star}{\delta}\bigg)$ with $\sigma>\sigma_0$.

\medskip
Fix $j \in \mathbb{N}$, to be specified later and let ${\tt v}_j({\pe})= v_j(s(\pe))$ be the appro\-ximate solution of \eqref{JacTodaSection3} described in Lemma \ref{lemmaapproximateslnJacToda}.

\medskip
From \eqref{est_err_j},
\begin{equation}\label{est_err_jsec3}
\|E_{\delta}({\tt v}_j)\|_{\mathcal{D}^{0,\beta}_{2,\frac{j}{2}}(\Sigma)} \leq C \delta \sigma^{- \frac{j}{2}}.
\end{equation}

We look for a solution of \eqref{JacTodaSection3} having the form ${\tt v}={\tt v}_j + {\tt q}$ so that we solve equation (\ref{eq_JT}) for ${\tt q}$. Using the operator
$$
L_{\delta}({\tt q}):=\delta J_{\Sigma}{\tt q} + 2\sqrt{2}a_{\star} e^{-\sqrt{2}{\tt v}_j}{\tt q},
$$
this equation reads as
\begin{equation}\label{linearJTwith quadractitermsec3}
L_{\delta}({\tt q})= -E_\delta({\tt v}_j)+\sqrt{2}\delta |A_\Sigma|^2{\tt v}_j Q({\tt q}),
\end{equation}
where $Q$ is the quadratic term defined in \eqref{JacTodaQuadracticerror}.

\medskip
Set 
$$
\mathcal{B}:=\bigg\{{\tt q}\in \mathcal{D}^{2,\beta}_{2,\frac{j-1}{2}}(\Sigma)\,:\, \|{\tt q}\|_{\mathcal{D}^{2,\beta}_{2,\frac{j-1}{2}}(\Sigma)}\leq A\sigma^{\frac{7}{4}-\frac{j}{2}}\log(\sigma)\bigg\},
$$
where $A>0$ is a constant independent of $\sigma$ and to be specified later.

\medskip
Next, observe there exists a constant $c>0$ such that for any ${\tt q}\in \mathcal{B}$, 
$$
|Q({\tt q})|\leq c{\tt q}^2
$$
and consequently,
\begin{equation}\label{est_Quadraticerrsec3}
\big\|\delta |A_{\Sigma}|^2 {\tt v}_jQ({\tt q})\big\|_{\mathcal{D}^{0,\beta}_{6,\frac{j-2}{2}}(\Sigma)} \leq cA^2 \delta \sigma^{\frac{7}{2} -j}\log(\sigma)^2.
\end{equation}

On the other hand, for some $c>0$ independent of $j$ and $\sigma$ and for any ${\tt q}, \tilde{{\tt q}}\in \mathcal{B}$,
$$
|Q({\tt q}) -Q(\tilde{\tt q})|\leq c|{\tt q}+ \tilde{\tt q}||{\tt q}- \tilde{\tt q}| \quad \hbox{in} \quad \Sigma 
$$
so that 
\begin{equation}\label{est_QuadraticLipschitzerrsec3}
\bigg\|\delta |A_{\Sigma}|^2 {\tt v}_j\big (Q({\tt q}) - Q(\tilde{\tt q})\big)\bigg\|_{\mathcal{D}^{0,\beta}_{6,\frac{j-2}{2}}(\Sigma)} \leq 2 A c \delta \sigma^{\frac{7}{4}-\frac{j}{2}}\log(\sigma)\|{\tt q} - \tilde{\tt q}\|_{\mathcal{D}^{2,\beta}_{2,\frac{j-1}{2}}(\Sigma)}.
\end{equation}

Next, let ${\tt R}: {\mathcal{D}^{2,\beta}_{2,\frac{j-1}{2}}(\Sigma)} \mapsto {\mathcal{D}^{2,\beta}_{2,\frac{j-1}{2}}(\Sigma)}$ be defined by
$$
{\tt R}({\tt q}):= \delta^{-1}L_{\delta}^{-1}\bigg(-E_\delta({\tt v}_j)+\sqrt{2}\delta |A_\Sigma|^2{\tt v}_j Q({\tt q})\bigg).
$$

Observe that ${\tt R}$ is the resolvent operator of the nonlinear equation \eqref{linearJTwith quadractitermsec3}.

\medskip
We fix $j \geq 8$, $\sigma>0$ large and $A>2C$ large, but independent of $\sigma$
 and $j$, so that from \eqref{est_err_jsec3} and  \eqref{est_Quadraticerrsec3}, we find that
$$
\begin{aligned}
\|{\tt R}({\tt q})\|_{{\mathcal{D}^{2,\beta}_{2,\frac{j-1}{2}}(\Sigma)}} & \leq  C\sigma^{\frac{7}{4}}\log(\sigma)\left(\sigma^{-\frac{j}{2}} + cA^2\sigma^{\frac{7}{2}-j}\log(\sigma)^2\right)\\
& \leq C\sigma^{\frac{7}{4}-\frac{j}{2}}\log(\sigma)\bigg(1 +\mathcal{O}(\sigma^{\frac{7}{2}-\frac{j}{2}}\log(\sigma)^2)\bigg)\\
&\leq 2C\sigma^{\frac{7}{4}-\frac{j}{2}}\log(\sigma)\\
& \leq A\sigma^{\frac{7}{4}-\frac{j}{2}}\log(\sigma)
\end{aligned}
$$
for any ${\tt q}\in \mathcal{B}$. Thus, ${\tt R}:\mathcal{B} \to \mathcal{B}$ is well defined.

\medskip
Next, we verify the contractive character of ${\tt R}$ restricted to $\mathcal{B}$. Using \eqref{est_QuadraticLipschitzerrsec3}, we find that for some constant $M>0$, independent of $\sigma$ and $j$,
\begin{equation*}
\begin{aligned}
\|{\tt R}({\tt q})- {\tt R}(\tilde{\tt q})\|& \leq \|(\delta^{-1}L_{\delta})^{-1}\|
\bigg\||A_{\Sigma}|^2 {\tt v}_j\big (Q({\tt q}) - Q(\tilde{\tt q})\big)\bigg\|_{\mathcal{D}^{0,\beta}_{6,\frac{j-2}{2}}(\Sigma)} \\
& \leq M
\sigma^{\frac{7}{4}-\frac{j}{2}}\log(\sigma)\|{\tt q} - \tilde{\tt q}\|_{\mathcal{D}^{2,\beta}_{2,\frac{j-1}{2}}(\Sigma)}
\end{aligned}
\end{equation*}
for any ${\tt q},\tilde{\tt q}\in \mathcal{B}$. Since $\sigma>0$ is large enough and $j > 8$, ${\tt R}:\mathcal{B} \to \mathcal{B}$ is a contraction.  

\medskip
A direct application of the contraction mapping principle yields the existence of a unique ${\tt q}\in \mathcal{B}$ solving \eqref{linearJTwith quadractitermsec3}. This completes the proof of Theorem \ref{th_Liouville}.

\section{The approximate solution to the Allen-Cahn equation}\label{approxslnAC}

In this part we find an appropriate approximate solution to \eqref{Allen-Cahn-eq} and compute its error in a suitable coordinate system.

\subsection{Fermi coordinates} Recall from Section \ref{JacToda Section} that in our developments we are $\Sigma=\Sigma_{m,n}^-$.
First, we introduce the system of coordinates that we will use to descri\-be the Laplacian near a dilated and translated version of the hypersurface $\Sigma$.

\medskip
Using \eqref{Lawson_cones} and the fact that $\Sigma$ is asymptotic to the cone $C_{m,n}$ stressed out in  \eqref{asymptoticsabatinfinity}, we find $\delta_0>0$ small and $\eta_0$ with $0<\eta_0 < \frac{1}{2\sqrt{N-1}}\min (\sqrt{m-1},\sqrt{n-1})$, such that in the tubular neighbourhood
$$
\mathcal{N}:=\left\{\pe+{\tt z}\nu_\Sigma(\pe)\,:\,
\pe\in\Sigma, \quad |{\tt z}|<\delta_0+ \eta_0|s(\pe)|\right\}
$$ of $\Sigma$, the mapping
\begin{equation}\label{FermicoordSigma}
X(\pe,{\tt z}):=\pe+{\tt z}\nu_\Sigma(\pe)
%\quad  \hbox{for} \quad \pe\in\Sigma,\,{\tt z}%\in\R
\end{equation}
defines a system of local coordinates, 
known as the \textit{Fermi coordinates}, where $\nu_\Sigma$ is computed in \eqref{param_O(m)O(n)_invariantnormalvec}. 

\medskip
Let $\eps>0$ be small, but fixed and consider the dilated hypersurface $\Sigma_\eps:=\eps^{-1}\Sigma$. Observe first that for any $\pe\in\Sigma_\eps$, there exists a unique 
$$
(\s,\x,\y)=(\s(\pe),\x(\pe),\y(\pe))\in \R \times S^{m-1}\times S^{n-1}
$$ 
such that
$$
\pe=\eps^{-1}\left(a(\eps\s)\x,b(\eps\s)\y\right)
$$
and 
\begin{equation}\notag
s(\eps\pe)=\eps\s(\pe), \quad \x(\pe)=\x(\eps\pe),\quad \y(\pe)=\y(\eps\pe).
\end{equation}

\medskip

The Fermi coordinates of $\Sigma_{\eps}$ are defined by
$$
X_\eps({\pe},{\tt z}):=\pe+{\tt z} \,\nu_\Sigma(\eps\pe)
% \quad \hbox{for}\quad \pe\in\Sigma_\eps,\,%{\tt z}\in\R
$$
in the dilated neighbourhood
$$\mathcal{N}_\eps:=\left\{\pe+{\tt z}\,\nu_\Sigma(\eps\pe)\,:\,\pe\in\Sigma_\eps, \quad 
| {\tt z}|<\frac{\delta_0}{\eps}+\eta_0|\s(\pe)|\right\}.
$$

\medskip
Fix {\color{red}$\alpha\in\left(0,\frac{1}{9}\right)$} and consider two smooth $O(m)\times O(n)$-invariant functions ${\tt h}_1, {\tt h}_2 : \Sigma \to \R$ with ${\tt h}_l(\pe)=h_l(s(\pe))$ for $\pe \in \Sigma$ and $l=1,2$ and such that
\begin{equation}
\begin{small}
\label{cond_growth_hl}
\frac{1}{\sqrt{2}}\left(l-\frac{3}{2}-\alpha\right)\left(\log(s^2+2)+2|\log\eps|\right)<h_l(s)
<\frac{1}{\sqrt{2}}\left(l-\frac{3}{2}+\alpha\right)\left(\log(s^2+2)+2|\log\eps|\right)
\end{small}
\end{equation}
for $s\in \R$. Assume also that ${\tt h}_l(\pe)=h_{l}(s)$ is even in the variable $s\in \R$.

\medskip
For $l=1,2$, the mapping
$$X_{\eps,\h_l}(\pe,t):=\pe+(t+{\tt h}_l(\eps\pe))\nu_\Sigma(\eps\pe)
$$
defines a diffeomorphism onto the tubular neighbourhood 
\begin{equation}
\label{metric_FC}
\begin{aligned}
\mathcal{N}_{l,\eps}
&=\left\{X_{\eps,\h_l}(\pe,t):|t|<\frac{1}{4\sqrt{2}}\left(\log(s(\eps\pe)^2+2)+2|\log\eps|\right)\right\}.\\
&:=\left\{X_\eps(\pe,{\tt z}):|{\tt z}-{\tt h}_l(\eps\pe)|<\frac{1}{4\sqrt{2}}\left(\log(s(\eps\pe)^2+2)+2|\log\eps|\right)\right\}.
\end{aligned}
\end{equation}

\medskip
Next, let $l\in \{1,2\}$. We compute the euclidean Laplacian in $\mathcal{N}_{l,\eps}$ for $O(m)\times O(n)-$invariant functions.

\medskip
Let $g=(g_{ij})_{N\times N}$ be a Riemannian  metric on $\Sigma$  with inverse $g^{-1}=(g^{ij})_{N\times N}$. Using the Fermi coordinates in \eqref{FermicoordSigma}, the metric $g$ induces a metric $G=(G_{ij})_{(N + 1)\times (N + 1)}$ on $\mathcal{N}$ whose entries are determined by the formulae
$$
G_{ij}=g_{ij}-2A_{ij}{\tt z}+{\tt z}^2\partial_i\nu_\Sigma\cdotp\partial_j\nu_\Sigma \quad \hbox{for} \quad i,j =1, \ldots,N,$$
$$ 
G_{i{\tt z}}=G_{{\tt z}i}=0 \quad \hbox{for} \quad i=1,\ldots,N, \quad G_{{\tt z}{\tt z}}=1.
$$

\medskip
Writing $G^{-1}=(G^{ij})_{(N+1)\times (N+1)}$, the Laplace operator in the set $\mathcal{N}$ takes the form
$$
\Delta=\frac{1}{\sqrt{\det G}}\partial_i\left(\sqrt{\det G}\,G^{ij}\partial_j\right),
$$
where summation over repeated indexes is understood.

\medskip
A direct computation yields that 
\begin{equation}\notag
\begin{aligned}
\Delta&=\partial^2_{\tt z}+\partial_{\tt z}(\log\sqrt{\det G})\partial_{\tt z}+G^{ij}\partial_{ij}+(\partial_i G^{ij}+\partial_i(\log\sqrt{\det G})G^{ij})\partial_i\\
&=\partial^2_{\tt z}-H_{\Sigma_{\tt z}}\partial_{\tt z}+\Delta_{\Sigma_{\tt z}},
\end{aligned}
\end{equation}
where $H_{\Sigma_{\tt z}}$ and $\Delta_{\Sigma_{\tt z}}$ are the mean curvature and the Laplace-Beltrami operator of the normally translated hypersurface
$$
\Sigma_{\tt z}:=\{\pe+{\tt z}\nu_\Sigma(\pe):\, \pe\in\Sigma\}
$$
for ${\tt z}$ fixed  and such that for every $\pe \in \Sigma$, $(\pe,{\tt z})\in X^{-1}(\mathcal{N})$. We observe that for $|{\tt z}|>\delta_0$, $\Sigma_{\tt z}$ is only defined outside a compact set and has two connected components.

\medskip
Let ${\tt z}$ be arbitrary and as in the previous paragraph. Define ${\tt Q}$ by
$$
{\tt Q}(\pe,z):= H_{\Sigma_z}(\pe)- H_\Sigma(\pe)-{\tt z}|A_\Sigma(\pe)|^2
$$
so that 
$$
|(1+s(\pe))D_{\Sigma}{\tt Q}|+ |\partial_{\tt z}{\tt Q}|+|{\tt Q}| \leq C|{\tt z}|^2(1 + s(\pe))^{-3}.
$$

\medskip

For $i,j=1,\ldots,N$, define also $a^{ij}$ and $b^j$ as 
\begin{equation}\notag
a^{ij}(\pe,{\tt z}):=G^{ij}(\pe,{\tt z})-g^{ij}(\pe),
\end{equation}
$$
b^j(\pe,{\tt z}):=\partial_i G^{ij}(\pe,{\tt z})+\partial_i(\log\sqrt{\det G(\pe,{\tt z})})G^{ij}(\pe,{\tt z})-\partial_i g^{ij}(\pe)-\partial_i(\log\sqrt{\det g})g^{ij}(\pe).
$$

\medskip

The formal expansion for the euclidean Laplacian in the coordinates $X(\pe,{\tt z})$ in the set $\mathcal{N}$ reads as
\begin{equation}\notag
\Delta=\partial^2_{\tt z}-{\tt z}|A_\Sigma|^2\partial_{\tt z}+\Delta_\Sigma-{\tt Q}(\pe,{\tt z})\partial_{\tt z}+a^{ij}(\pe,{\tt z})\partial_{ij}+b^j(\pe,{\tt z})\partial_j.
\end{equation}

\medskip

Introducing the change of variables
\begin{equation}
\label{eps_shift_coord}
z=\eps(t+{\tt h}_l(\eps\pe))
\quad \hbox{for} \quad (\pe,t)\in X_{\eps,{\tt h}_l}^{-1}(\mathcal{N}_{l,\eps}),
\end{equation}
the euclidean Laplacian can be computed in the set $\mathcal{N}_{l,\eps}$ in the coordinates $X_{\eps,{\tt h}_l}(\pe,t)$. These computations are collected in the following Lemma.

\begin{lemma}
\label{lemma_coord_Fermi}
In the neighbourhood $\mathcal{N}_{l,\eps}$, the Laplacian in the $(\pe,t)$ coordinates is given by
\begin{equation}
\label{Laplacian_Fermi}
\begin{aligned}
\Delta&=\Delta_{\Sigma_\eps}+\partial^2_t-\eps^2\big(\Delta_\Sigma {\tt h}_l+|A_\Sigma|^2{\tt h}_l\big)\partial_t- \eps^2 t|A_{\Sigma}|^2 \partial_{t} - 2 \eps\nabla_\Sigma {\tt h}_l\cdotp\partial_t\nabla_\Sigma+\eps^2|\nabla_\Sigma {\tt h}_l|^2\partial^2_t
\\
&-\eps{\tt Q}\partial_t +(a^{ij}\partial^2_{ij}+\eps b^j\partial_j)-\eps^2(a^{ij} \partial_{ij}{\tt h}_l+b^j \partial_j{\tt h}_l)\partial_t-2\eps a^{ij} \partial_i{\tt h}_l\partial_{tj}+\eps^2a^{ij}\partial_i{\tt h}_l\partial_j{\tt h}_l \partial^2_t,
\end{aligned}
\end{equation}
where ${\tt h}_l$ and its derivatives are evaluated at $\eps\pe$, while ${\tt Q},\, a^{ij}$ and $b^j$ are evaluated at $(\eps\pe,\eps(t+{\tt h}_l(\eps\pe)))$.
\end{lemma}

\medskip
We remark that the computations in Lemma \ref{lemma_coord_Fermi} started off from and are valid for an arbitrary metric on $\Sigma$.

\medskip
Next, we consider the Riemannian metric on $\Sigma$ that is induced by the parametrisation in \eqref{param_O(m)O(n)_invariant}. After rescaling and translating the neighbourhood $\mathcal{N}$ to $\mathcal{N}_{l,\eps}$ in this particular coordinates, it follows from Lemma \ref{FermicoordSigma} that for smooth functions defined in $\mathcal{N}_{l,\eps}$, expressed in the coordinates $(\pe,t)$ with $\s=\s(\pe)\in\R$ and being even in $\s$, we have
\begin{equation}
\label{Laplacian_Fermi_st}
\begin{aligned}
\Delta= & \Delta_{\Sigma_{\eps}}+\partial^2_t-\eps^2(\Delta_{\Sigma}{\tt h}_l +|A_\Sigma|^2{\tt h}_l)\partial_t - \eps^2t|A_{\Sigma}|^2 \partial_t\\
&-\eps^2(h''_l+\alpha(\eps \s)h'_l)\partial_t-2\eps h'_l\partial_{t\s}+\eps^2(h'_l)^2\partial^2_t\\
&-\eps\Q\partial_t+\ta\partial^2_\s+\eps\bi\partial_\s-\eps^2(\ta h''_l+\bi h'_l)\partial_t-2\eps\ta h'_l\partial_{t\s}+\eps^2\ta(h'_l)^2 \partial^2_t,
\end{aligned}
\end{equation}
where
$$
\Delta_{\Sigma_{\eps}}=\partial^2_{\s}+\eps\alpha(\eps\s)\partial_\s
$$
and
$$
\Delta_{\Sigma}{\tt h}_l + |A_{\Sigma}|^2 {\tt h}_l = \partial_{ss}h_l(s) + \alpha(s)\partial_s h_l(s) + \beta(s) h_l(s)
$$
with $\alpha(s)$ and $\beta(s)$ described in \eqref{Laplace_Sigma_O(m)O(n)_invariant}, \eqref{JTincoordPsi}, \eqref{alphaasymptotics} and \eqref{betaasymptotics} and where 
$h_l$ and its derivatives are evaluated at $\eps\s$.

\medskip
Also, for $(\pe,{\tt z})\in\Sigma\times\R$, we have denoted 
\begin{equation}\notag
\begin{aligned}
\Q(s(\pe),{\tt z}):={\tt Q}(\pe,{\tt z}), \qquad
\ta(s(\pe),{\tt z}):=a^{ss}(\pe,{\tt z}),
\qquad\bi(s(\pe),{\tt z}):=b^s(\pe,{\tt z})
\end{aligned}
\end{equation}
and observe that $\Q,\,\ta$ and $\bi$ are evaluated at $(\eps\s,\eps(t+h_l(\eps\s)))$.

\medskip
Proceeding as in the Appendix in \cite{DELPINOKOWALCZYKWEI2013I} and as in Section 3 in \cite{AGUDELODELPINOWEI2016}, using the variables $(\s,t)$ and setting ${\tt h}_l(\pe)=h_l(s(\pe))$ for $\pe \in \Sigma$, we find that
$$
|(1+s(\pe))D_{\Sigma}a^{ss}|+ |\partial_{\tt z}a^{ss}|+|a^{ss}| \leq C|{\tt z}|(1 + s(\pe))^{-1}
$$
and
$$
|(1+s(\pe))D_{\Sigma}b^s|+ |\partial_{\tt z}b^s|+|b^s| \leq C|{\tt z}|(1 + s(\pe))^{-2}.
$$

\medskip
In particular, for every $(\s,t)=(\s(\pe),t)$ with $(\pe,t)\in X^{-1}_{\eps,\h_l}(\mathcal{N}_{\eps,l})$,
\begin{equation}
\label{est_remainder_laplacian}
\begin{aligned}
|\Q(\eps\s,\eps(t+h_l(\eps\s)))|&\le C\frac{\eps^2(t+h_l(\eps\s))^2}{(1+|\eps \s|)^3}\\
|\ta(\eps\s,\eps(t+h_l(\eps\s)))|&\le C\frac{\eps(t+h_l(\eps\s))}{1+|\eps\s|}
\\
|\bi(\eps\s,\eps(t+h_l(\eps\s)))|&\le C\frac{\eps(t+h_l(\eps\s))}{(1+|\eps\s|)^2}.
\end{aligned}
\end{equation}

\subsection{First approximation}
In this part we choose the first approximation of the solution of equation \eqref{Allen-Cahn-eq}. We focus to the region near the approximate zero level set using the Fermi coordinates $(\pe,{\tt z})$ of $\Sigma_\eps \times \R$.

\medskip
Recall that we have fixed two $O(m)\times O(n)$-invariant functions ${\tt h}_1, {\tt h}_2:\Sigma\to\R$, ${\tt h}_l(\pe)=h_l(s)$, with $h_1,h_2$ even in $s$, ${\tt h}_1, {\tt h}_2 \in C^2(\Sigma)$ and satisfying \eqref{cond_growth_hl}. Assume further that 
\begin{equation}\label{orderedinterfaces}
-\infty \equiv {\tt h}_0 < {\tt h}_1 < {\tt h}_2 < {\tt h}_3\equiv + \infty
\end{equation}
and that for any $s\in \R$,
\begin{equation}
\label{cond_growth_der_hl}
|h'_l(s)|\le\frac{c}{|s|+1},\quad |h''_l(s)|\le\frac{c}{(|s|+1)^2}, \quad l=1,\,2.
\end{equation}

\medskip
First, for $l=1,2$ and $t\in \R$ set $w_l(t):=(-1)^{l-1}v_\star(t),$ where $v_{\star}(t)$ is the heteroclinic solution to \eqref{Allen-Cahn-eqn 1D} described in \eqref{Allen-Cahn-eqn 1D Sln}. For $(\pe,{\tt z})\in\Sigma_\eps\times\R$ define
\begin{equation}\label{firstapproxAC}
U_0(\pe,{\tt z}):=w_1({\tt z}-{\tt h}_l(\eps \pe))+ w_2({\tt z}-{\tt h}_2(\eps \pe))-1.
\end{equation}

Set also
$$
S(u)=\Delta u+F(u), \qquad F(u)=u(1-u^2)
$$
and let us now compute the error $S(U_0)$ near the normal graphs of the functions $\pe\mapsto{\tt h}_l(\eps\pe)$ over $\Sigma_\eps$.

\begin{lemma}\label{lemmaerrorSUzero}
For $l\in \{1,2\}$ and for any $(\pe,t)\in X^{-1}_{\eps, {\tt h}_l}(\mathcal{N}_{l,\eps})$,
\begin{equation}\label{errorSUzero}
\begin{aligned}
(-1)^{l-1}S(U_0)=&-\eps^2(\Delta_\Sigma \h_l+|A_\Sigma|^2 \h_l)v'_\star-\eps^2|A_\Sigma|^2tv'_\star+\eps^2|\nabla_\Sigma \h_l|^2v''_\star\\
&+6(1-v_\star^2)(e^{-\sqrt{2}t}e^{-\sqrt{2}(\h_l-\h_{l-1})}-e^{\sqrt{2}t}e^{-\sqrt{2}(\h_{l+1}-\h_l)})
%&-\eps^2\sum_{|j-l|\ge 1}(\Delta_\Sigma h_j+|A_\Sigma|^2(t+h_l))w'_j(t+h_l-h_j)\\
%&+\eps^2\sum_{|j-l|\ge 1}(h'_j)^2 w''_j(t+h_l-h_j)
+R_\eps(\h_1,\h_2),
\end{aligned}
\end{equation}
where $R_\eps(\h_1,\h_2)$ is such that
\begin{equation}\label{estimateReps}
|R_\eps(\h_1,\h_2)|\le C\eps^{2+\gamma}(s(\eps\pe)^2+2)^{-\frac{2+\gamma}{2}}e^{-\rho|t|} \quad \hbox{in} \quad X_{\eps,\h_l}^{-1}(\mathcal{N}_{l,\eps})
\end{equation}
for some $\gamma\in(0,\frac{1}{2})$ and $\rho\in(0,\sqrt{2})$.
\end{lemma}

\begin{proof}
Our calculations are done for $O(m)\times O(n)$-invariant functions and hence we use the coordinates $(\s,t)$, where $\s=\s(\pe)$ and $t= {\tt z} - {\tt h}_l(\eps \pe)$.  Recall also that $\h_l(\eps\pe)=h_l(\eps\s)$.

\medskip
We write
$$(-1)^{l-1}S(U_0)=E_1+E_2,\qquad E_1:=(-1)^{l-1}F(U_0),\qquad E_2:=(-1)^{l-1}\Delta U_0.
$$
First we consider the term $E_1$. Observe that for
\begin{equation}
F(U_0)=\sum_{j=1}^2 F(w_j(t+h_l-h_j))+F(U_0)-\sum_{j=1}^2 F(w_j(t+h_l-h_j)).
\end{equation}
Since $F$ is odd for $l\in \{1,2\}$, 
$$
(-1)^{l-1}F((-1)^{l-1}v)=F(v), \qquad (-1)^{l-1}F((-1)^{l-2}v)=(-1)^{l-1}F((-1)^l v)=-F(v).
$$
Thus,
\begin{equation}
\label{est_E1_1}
\begin{aligned}
&(-1)^{l-1}\left(F(U_0)-\sum_{j=1}^2 F(w_j(t+h_l-h_j))\right)\\
&=(-1)^{l-1}F(U_0)-F(v_\star)+F(v_\star(t+h_l-h_{l-1}))+F(v_\star(t+h_l-h_{l+1}))%-\sum_{|j-l|\ge 2} F(w_j(t+h_l-h_j))
\end{aligned}
\end{equation}
By the Mean Value Theorem, there exists $\xi_1$ between $w_{l}(t)$ and $U_0(\pe,t)$ such that
\begin{equation}\notag
F(U_0)-F(w_l)=F'(w_l)(U_0-w_l)+\frac{1}{2}F''(w_l+\xi_1(U_0-w_l))(U_0-w_l)^2.
\end{equation}
Using that $F$ is odd, we get
\begin{equation}\notag
\begin{aligned}
(-1)^{l-1}F(U_0)-F(v_\star)=&F'(v_\star)((-1)^{l-1}U_0-v_\star)\\
&+\frac{1}{2}F''(v_\star+\xi_1((-1)^{l-1}U_0-v_\star))((-1)^{l-1}U_0-v_\star)^2.
\end{aligned}
\end{equation}
Since,
$$(-1)^{l-1}U_0-v_\star=\sgn(l-j)-v_\star(t+h_l-h_j), \qquad j\in\{1,2\},\,\,\,j\ne l,
$$%\sum_{|j-l|\ge 1}(-1)^{j+l}(v_\star(t+h_l-h_j)-\sgn(l-j)),$$
performing a Taylor expansion we find that
\begin{equation}\label{ACerrorNonlin}
\begin{aligned}
F(v_\star(t+h_l-h_{l-1}))=&F(1)+F'(1)(v_\star(t+h_l-h_{l-1})-1)\\
&+\frac{1}{2}F''(1+\xi_2 (v_\star(t+h_l-h_{l-1})-1))(v_\star(t+h_l-h_{l-1})-1)^2
\end{aligned}
\end{equation}
and
\begin{equation}\label{ACerrorNonlin1}
\begin{aligned}
F(v_\star(t+h_l-h_{l+1}))=&F(-1)+F'(-1)(v_\star(t+h_l-h_{l+1})+1)\\
&+\frac{1}{2}F''(-1+\xi_3 (v_\star(t+h_l-h_{l+1})+1))(v_\star(t+h_l-h_{l+1})+1)^2.
\end{aligned}
\end{equation}
Putting together \eqref{est_E1_1}, \eqref{ACerrorNonlin} and \eqref{ACerrorNonlin1} we find that for $j\in\{1,2\},\,j\ne l$,
\begin{equation}
\begin{aligned}
&(-1)^{l-1}\left(F(U_0)-\sum_{j=1}^k F(w_j(t+h_l-h_j))\right)=\\
&-(2+F'(v_\star))(v_\star(t+h_l-h_{l+1})+v_\star(t+h_l-h_{l-1}))\\
%&+F'(v_\star)\left(\sum_{|l-j|\ge 2}(-1)^{l+j}(v_\star(t+h_l-h_j)-\sgn(l-j))\right)\\
&+\frac{1}{2}F''(v_\star+\xi_1((-1)^{l-1}U_0-v_\star))(v_\star(t+h_l-h_j)-\sgn(l-j))^2\\
%\left(\sum_{|j-l|\ge 1}(-1)^{j+l}(v_\star(t+h_l-h_j)-\sgn(l-j))\right)^2\\
&+\frac{1}{2}F''(1+\xi_2 (v_\star(t+h_l-h_{l-1})-1))(v_\star(t+h_l-h_{l-1})-1)^2\\
&+\frac{1}{2}F''(-1+\xi_3 (v_\star(t+h_l-h_{l+1})+1))(v_\star(t+h_l-h_{l+1})+1)^2,\qquad .
\end{aligned}
\end{equation}

Finally, using the asymptotic behaviour of $v_\star$, we have
\begin{equation}
E_1=(-1)^{l-1}\sum_{j=1}^2 F(w_j(t+h_l-h_j))+6(1-v_\star^2)(e^{-\sqrt{2}t}e^{-\sqrt{2}(h_l-h_{l-1})}-e^{\sqrt{2}t}e^{-\sqrt{2}(h_{l+1}-h_l)})+R_1,
\end{equation}
where
\begin{equation}
\begin{aligned}
R_{1,\eps}:=&-3(1+v_\star^2)(v_\star(t+h_l-h_{l-1})-1+2e^{-\sqrt{2}t}e^{-\sqrt{2}(h_l-h_{l-1})})\\
&-3(1+v_\star^2)(v_\star(t+h_l-h_{l+1})+1-2e^{\sqrt{2}t}e^{-\sqrt{2}(h_{l+1}-h_l)})\\
%&+F'(v_\star)\left(\sum_{|l-j|\ge 2}(-1)^{l+j}(v_\star(t+h_l-h_j)-\sgn(l-j))\right)\\
&+\frac{1}{2}F''(v_\star+\xi_1((-1)^{l-1}U_0-v_\star))(v_\star(t+h_l-h_j)-\sgn(l-j))^2\\
%\left(\sum_{|j-l|\ge 1}(-1)^{j+l}(v_\star(t+h_l-h_j)-\sgn(l-j))\right)^2\\
&+\frac{1}{2}F''(1+\xi_2 (v_\star(t+h_l-h_{l-1})-1))(v_\star(t+h_l-h_{l-1})-1)^2\\
&+\frac{1}{2}F''(-1+\xi_3 (v_\star(t+h_l-h_{l+1})+1))(v_\star(t+h_l-h_{l+1})+1)^2
\end{aligned}
\end{equation}
for $j\in\{1,2\},\,j\ne l$.

\medskip
Next, we compute $E_2$. Observe that 
$$
(-1)^{l-1}\Delta U_0=\Delta v_\star-\Delta v_\star(t+h_l-h_j)\quad \hbox{for} \quad j\in\{1,2\},\,j\ne l.
$$
Moreover, from \eqref{Laplacian_Fermi_st} we find that
\begin{equation}\notag
\begin{aligned}
\Delta v_\star=&v''_\star-\eps^2(h''_l+\alpha h'_l+\beta h_l)v'_\star
-\eps^2\beta tv'_\star+\eps^2(h'_l)^2 v''_\star\\
&-\eps\Q v'_\star-\eps^2(\ta h''_l+\bi h'_l)v'_\star+\eps^2\ta(h'_l)^2v''_\star\\
=&v''_\star-\eps^2(h''_l+\alpha h'_l+\beta h_l)v'_\star
-\eps^2\beta tv'_\star+\eps^2(h'_l)^2 v''_\star+R_{2,\eps},
\end{aligned}
\end{equation}
where $\alpha$, $\beta$ and $h_l$ and its derivatives are all evaluated at $\eps\s$, while $\Q,\,\ta$ and $\bi$ are evaluated at $(\eps\s,\eps(t+h_l(\eps\s)))$. 
Similarly,
\begin{equation}
\begin{aligned}
(-1)^l\Delta v_\star(t+h_l-h_j)=&\Delta w_j(t+h_l-h_j)\\
=&w''_j(t+h_l-h_j)-\eps^2(h''_j+\alpha h'_j+\beta(t+h_l))w'_j(t+h_l-h_j)\\
%&\sum_{|j-l|\ge 1}\left(w''_j(t+h_l-h_j)-\eps^2(\Delta_\Sigma h_j+|A_\Sigma|^2(t+h_l))w'_j(t+h_l-h_j)+\eps^2(h'_j)^2 w''_j(t+h_l-h_j)\right)\\
&+\eps^2(h'_j)^2 w''_j(t+h_l-h_j)+\eps^2\left(\ta w''_j(t+h_l-h_j)(h'_j)^2\right)\\
%&+\eps^3(t+h_l)\sum_{|j-l|\ge 1}\left(\ta w''_j(t+h_l-h_j)(h'_j)^2-(\ta h''_j+\bi h'_j)w'_j(t+h_l-h_j)\right)\\
&-\eps^2(\ta h''_j+\bi h'_j)w'_j(t+h_l-h_j)-\eps\Q w'_j(t+h_l-h_j).\\
= &w''_j(t+h_l-h_j)-\eps^2(h''_j+\alpha h'_l+\beta(t+h_l))w'_j(t+h_l-h_j)\\
&+\eps^2(h'_j)^2 w''_j(t+h_l-h_j)+R_{3,\eps}
\end{aligned}
\end{equation}
for $j\in\{1,2\}$ with $j\ne l$.
Therefore, 
\begin{equation}
\begin{aligned}
(-1)^l\Delta v_\star(t+h_l-h_j)
= w''_j(t+h_l-h_j)-\eps^2(h''_j+\alpha h'_l+\beta(t+h_l))w'_j(t+h_l-h_j)\\
+\eps^2(h'_j)^2 w''_j(t+h_l-h_j)+R_{3,\eps}
\end{aligned}
\end{equation}
for $j\in\{1,2\}$ with $j\ne l$ and consequently,
\begin{equation}\notag
\begin{aligned}
E_2 =v''_\star-&\eps^2(h''_l+\alpha h'_l+\beta h_l)v'_{\star}-\eps^2\beta tv'_\star +\eps^2(h'_l)^2 v''_\star\\
&+w''_j(t+h_l-h_j)-\eps^2(h''_j+\alpha h'_l+\beta(t+h_l))w'_j(t+h_l-h_j)\\
& \hspace{1cm}+\eps^2(h'_j)^2 w''_j(t+h_l-h_j)+R_{2,\eps}+R_{3,\eps}.
\end{aligned}
\end{equation}
Setting $R_\eps:=R_{1,\eps}+R_{2,\eps}+R_{3,\eps}$, we find that \eqref{errorSUzero} holds true.

\medskip 
It remains to prove the estimate \eqref{estimateReps} for the remainder term $R_{\eps}$. To do so, we proceed as follows.
Using the inequalities in \eqref{cond_growth_hl}, we find the lower estimate
\begin{equation}\notag
\begin{aligned}
|t+h_j-h_l|\ge &|h_j-h_l|-|t|\\
=&|h_j-h_l|-(1+\alpha)|t|+\alpha|t|\\
\ge& \frac{1}{\sqrt{2}}\left(1-2\alpha-\frac{1+\alpha}{4}\right)(\log((\eps\s)^2+2)+2|\log\eps|)+\alpha|t|
\end{aligned}
\end{equation}
and hence we conclude that
\begin{equation}\notag
\begin{aligned}
e^{-2\sqrt{2}|t+h_j-h_l|}\le & e^{-2\sqrt{2}\left(1-2\alpha-\frac{1+\alpha}{4}\right)(\log((\eps s)^2+2)+2|\log\eps|)}e^{-2\sqrt{2}\alpha|t|}\\ 
\le & \eps^{3(1-3\alpha)}((\eps\s)^2+2)^{-\frac{3}{2}(1-3\alpha)}e^{-2\alpha|t|}.
\end{aligned}
\end{equation}
We note that
$$3(1-3\alpha)=:2+\gamma>2$$
provided $\alpha>0$ is small enough. 

\medskip
Putting together the previous estimate and the estimates in \eqref{est_remainder_laplacian}, the estimate in \eqref{estimateReps} follows. This completes the proof of the Lemma.
\end{proof}

\subsection{Improvement of the approximation}
In this subsection we improve the approximation $U_0$ defined in \eqref{firstapproxAC}. More precisely, we cancel the terms of order $\eps^2$ or smaller.

\medskip
We begin by writing
\begin{equation}\label{AClinearOrthCond1}
6(1-v_\star^2)e^{-\sqrt{2}t}=a_{\star} v'_\star+g_{0}(t) \quad  \hbox{with} \quad \int_\R g_{0}(t)v'_\star(t)dt=0
\end{equation}
and by noticing that
\begin{equation}\label{AClinearOrthCond2}
\int_{\R} v''_{\star}(t)v'_{\star}(t)dt =\int_\R t(v'_\star(t))^2 dt=0.
\end{equation}

In order to improve the approximation $U_0$, we solve the ODE's
\begin{equation}
\psi''_0+(1-3v_\star^2)\psi_0=g_0, \quad  \psi''_1+(1-3v_\star^2)\psi_1=-v''_\star, \quad 
\psi''_2+(1-3v_\star^2)\psi_2=tv'_\star \quad \hbox{in} \quad \R.
\end{equation}

Since $v'_{\star}$ is a positive solution of the equation
$$
\psi''(t) + (1 - 3v_{\star}^2)\psi(t)=0 \quad\hbox{in} \quad \R,
$$ 
directly from the variation of parameters formula and the orthogonality conditions in \eqref{AClinearOrthCond1} and \eqref{AClinearOrthCond2}, we find that 
$$
\psi_0(t)=v'_\star(t)\int_0^t (v'_\star(\tau))^{-2} \left(\int_{\tau}^\infty v'_\star(\xi)g_0(\xi)d\xi\right) d\tau,
$$
$\psi_1(t)=\frac{1}{2}tv'_{\star}(t)$ and 
$$
\psi_2(t)=v'_\star(t)\int_0^t (v'_\star(\tau))^{-2} \left(\int_\tau^\infty \xi v'_\star(\xi)^2d\xi\right) d\tau.
$$

We also find that for any $i\in \mathbb{N}\cup \{0\}$, there exists $C_i>0$ such that
$$
\big\|(1 + e^{2\sqrt{2}|t|}\chi_{t>0})\partial_i \psi_0\big\|_{L^{\infty}(\R)}\leq C_i
$$
and for any $\varrho\in (0\sqrt{2})$ and for any $j=1,2$ and any $i\in \mathbb{N}\cup \{0\}$, there exists $\tilde{C}_i>0$ 
$$
\big\|e^{\varrho |t|}\partial_i \psi_j\big\|_{L^{\infty}(\R)}\leq \tilde{C}_i.
$$
Furthermore, since $\psi_1$ and $\psi_2$ are odd functions in the variable $t$, then
$$
\int_\R\psi_1(t) v'_\star(t) dt=\int_\R\psi_2(t) v'_\star(t) dt=0.
$$

\medskip
Next, we proceed as in subsection 5.2 in \cite{AGUDELODELPINOWEI2015}. First, define $\eta_j: \Sigma_{\eps}\times \R \to \R$ by the formula 
\begin{equation}
\begin{aligned}
(-1)^{j-1}\eta_j(\pe,\z):=&-e^{-\sqrt{2}(\h_j(\eps\pe)-\h_{j-1}(\eps\pe))}\psi_0(\h_j(\eps\pe)-\z)
+e^{-\sqrt{2}(\h_{j+1}(\eps\pe)-\h_j(\eps\pe))}\psi_0(\z-\h_j(\eps\pe))\\
&\hspace{3cm}+\eps^2\psi_1(\z-\h_j(\eps\pe))+\eps^2|A_\Sigma(\eps\pe)|^2\psi_2(\z-\h_j(\eps\pe))
\end{aligned}
\end{equation}
and set $\tilde{\eta}_j(\pe,t)=(-1)^{j-1}{\eta}_j(\pe,t+\h_j(\eps\pe))$.
We consider the approximation
\begin{equation}\notag
U_1(\pe,\z):=U_0(\pe,\z)+\eta(\pe,\z), \qquad \eta(\pe,\z):=\sum_{j=1}^2\eta_{j}(\pe,\z). 
\end{equation}
Using the variable $t:=\z-\h_l(\eps\pe)$ in $X^{-1}_{\eps,\h_l}(\mathcal{N}_{l,\eps})$, we have for $  j\in\{1,2\}$ with $j\ne l$, that
\begin{equation}\notag
\begin{aligned}
(-1)^{l-1}S(U_1)%&=(-1)^{l-1}S(U_0)+(-1)^{l-1}(\Delta_{\mathcal{N}_{l,\eps}}\eta+F'(U_0)\eta)+(-1)^{l-1}Q_{U_0}(\eta)=\\
&=(-1)^{l-1}S(U_0)+\Delta_{\mathcal{N}_{l,\eps}}\tilde{\eta}_l+F'(w_l)\tilde{\eta}_l+(F'(U_0)-F'(w_l))\tilde{\eta}_l\\
&+\Delta_{\mathcal{N}_{j,\eps}}\tilde{\eta}_j(t+h_l-h_j)+F'(U_0)\tilde{\eta}_j(t+h_l-h_j)+(-1)^{l-1}Q_{U_0}(\eta),
\end{aligned}
\end{equation}
where we have denoted
$$
Q_{U_0}(\eta)=F(U_0+\eta)-F(U_0)-F'(U_0)\eta.
$$

%\begin{equation}\notag
%\begin{aligned}
%(-1)^{l-1}(\Delta_{\mathcal{N}_{l,\eps}}\eta_l+F'(w_l)\eta_l)=\Delta_{\mathcal{N}_{l,\eps}}\tilde{\eta}_l+F'(w)\tilde{\eta}_l
%\end{aligned}
%\end{equation}

\medskip
Using \eqref{Laplacian_Fermi_st} and \eqref{est_remainder_laplacian} in $\mathcal{N}_{l,\eps}$, we find that
\begin{equation}\notag
\begin{aligned}
(\Delta_{\mathcal{N}_{l,\eps}}+F'(w))(\eps^2(h'_l(\eps\s))^2\psi_1(t)+\eps^2\beta(\eps\s)\psi_2(t))=\eps^2\beta(\eps\s) tv'_\star-\eps^2(h'_l)^2 v''_\star+R_{4,\eps}(\h_1,\h_2)
\end{aligned}
\end{equation}
and
\begin{equation}\notag
\begin{aligned}
&(\Delta_{\mathcal{N}_{l,\eps}}+F'(w))\left(-e^{-\sqrt{2}(h_l-h_{l-1})}\psi_0(-t)
+e^{-\sqrt{2}(h_{l+1}-h_l)}\psi_0(t)\right)=\\
&-e^{-\sqrt{2}(h_l-h_{l-1})}g_0(-t)+e^{-\sqrt{2}(h_{l+1}-h_l)}g_0(t)+R_{5,\eps}(\h_1,\h_2),
\end{aligned}
\end{equation}
with $$|R_{i,\eps}(\h_1,\h_2)|\le C \eps^3(s(\eps\pe)^2+2)^{-\frac{2+\gamma}{2}}e^{-\rho|t|}, \qquad i=4,5.
$$

The following lemma summarises the computations of the error $S(U_1)$.

\begin{lemma}
\label{lemma_new_error}
Assume the hypothesis in Lemma \eqref{errorSUzero}. The error $S(U_1)$ in $X^{-1}_{\eps,\h_l}(\mathcal{N}_{\eps,l})$ is given by
\begin{equation}\notag
\begin{aligned}
(-1)^{l-1}S(U_1)=\eps^2(\Delta_\Sigma \h_l+|A_\Sigma|^2 \h_l)v'_\star+a_\star(e^{-\sqrt{2}(\h_l-\h_{l-1})}-e^{-\sqrt{2}(\h_{l+1}-\h_l)})v'_\star+R_{6,\eps}(\h_1,\h_2),
\end{aligned}
\end{equation}
with 
\begin{equation}
|R_{6,\eps}(\h_1,\h_2)|\le C \eps^{2+\gamma}(s(\eps\pe)^2+2)^{-\frac{2+\gamma}{2}}e^{-\rho|t|}.
\end{equation}
\end{lemma}
\section{The Lyapunov-Schmidt reduction}\label{LSreduction}

In this section, we perform an infinite dimensional Lyapunov-Schmidt reduction procedure and finish the construction of the solution of \eqref{Allen-Cahn-eq} predicted in Theorem \ref{2ndMainTheo}. Again many of the developments are in the lines of those in \cite{AGUDELODELPINOWEI2015, DELPINOKOWALCZYKWEI2013I}.

\subsection{A gluing procedure}\label{subs_gluing}
We begin the developments in this part by construction a global approximation of \eqref{Allen-Cahn-eq}. We introduce a smooth cutoff function $\chi\in C^\infty(\R)$ such that $0\le\chi\le 1$ and
\begin{equation}\notag
\chi(t)=\begin{cases}
1 \qquad t\le 1\\
0 \qquad t\ge 2.
\end{cases}
\end{equation}
For $(\pe,{\tt z})\in\Sigma_\eps\times\R$, we set
$$\zeta(\pe,{\tt z}):=\chi\left(|{\tt z}|-\frac{4}{\sqrt{2}}(\log(s(\eps\pe)^2+2)+2|\log\eps|)+2\right)$$
and 
\begin{equation}
\tilde{\zeta}(\xi):=
\left\{\begin{aligned}
\zeta\circ X_\eps^{-1}(\xi) , \qquad &\hbox{for} \quad \xi\in \mathcal{N}_\eps\\
0 \,\,\,\,\,\,, \qquad &\hbox{otherwise.}
\end{aligned}
\right.
\end{equation}
%being $$\Pi:(s,\x,\y,\xi)\in\R\times S^{m-1}\times S^{n-1}\times\R\mapsto(s,\xi)\in\R^2$$ 
%the projection onto the first and the last components. 
Similarly, we set
\begin{equation}\label{globalapproxsect5}
\tilde{U}_1(\xi):=
\left\{\begin{aligned}
U_1\circ X_\eps^{-1}(\xi) , \qquad &\hbox{for} \quad \xi\in \mathcal{N}_\eps\\
0 \,\,\,\,\,\,, \qquad &\hbox{otherwise.}
\end{aligned}
\right.
\end{equation}
We define our global approximation as
%\begin{equation}\notag
%\mathbb{H}_\eps(\xi):=
%\begin{cases}
%-1\qquad\text{for $\xi\in S_\eps$}\\
%(-1)^{k+1}\qquad\text{for $\xi\in \tilde{S}_\eps$.}
%\end{cases}
%\end{equation} 
%so that we set
\begin{equation}
w(\xi):=\tilde{\zeta}(\xi)\tilde{U}_1(\xi)-(1-\tilde{\zeta}(\xi)),
\end{equation}
and we look for a solution to the Allen-Cahn equation of the form
$$u=w+\varphi,$$
where $\varphi:\R^{N+1}\to \R$ is a small correction in an appropriate topology to be determined later.

\medskip
We introduce next some useful notation. For $l=1,\,2$ and a function $u:\Sigma_\eps\times\R\to\R$, we set
\begin{equation}
u^{\natural}_l(\xi):=
\begin{cases}
u\circ X_{\eps,{\tt h}_l}^{-1}(\xi),\qquad &\text{for $\xi\in\mathcal{N}_{\eps,l}\subset\R^{N+1}$}\\
\hspace{1cm}0, \qquad &\text{otherwise}.
\end{cases}
\end{equation}

On the other hand, given any $O(m)\times O(n)$-invariant function $v:\R^{N+1}\to\R$, we define for $1\le l\le 2$ and $(\pe,t)\in\Sigma_\eps\times\R$,
\begin{equation}
v_{l}^{\sharp}(\pe,t):=\begin{cases}
v\circ X_{\eps,{\tt h}_l}(\pe,t)\qquad&\text{if $(\pe,t)\in X^{-1}_{\eps,{\tt h}_l}(\mathcal{N}_{\eps,l})$}\\
\hspace{1cm}0,\qquad &\text{otherwise.}
\end{cases}
\end{equation}

The notation just introduced can be explained as follows: $u^\natural_l$ refers to an expression of $u$ in {\it{natural}} or euclidean coordinates while $u^\sharp$ refers to a function expressed in natural coordinates {\it{pushed forward}} to  Fermi coordinates.

\medskip
For any integer $i\ge 1$ and $(\pe,t)\in\Sigma_\eps\times\R$, we set
\begin{equation}\notag
\chi_i(\pe,t):=\chi\left(|t|-\frac{1}{4\sqrt{2}}(\log(s(\eps\pe)^2+2)+2|\log\eps|)+i\right)
\end{equation}
and we look for a correction of the form
\begin{equation}\label{globalsmallterm}
\varphi=\sum_{l=1}^2 \chi^{\natural}_{3,l}\varphi_l+\psi.
\end{equation}

Using the fact that $\chi^{\natural}_{4,l}\chi^{\natural}_{3,l}=\chi^{\natural}_{4,l}$, the Allen-Cahn equation can be written as
\begin{equation}\notag
\begin{aligned}
&0=S(w+\varphi)=S(w)+\Delta\varphi+F'(w)\varphi+Q_w(\varphi)=\\
&\sum_{l=1}^2\chi^{\natural}_{3,l}(\Delta\varphi_l+F'(w)\varphi_l+\chi^{\natural}_{4,l}S(w)+\chi^{\natural}_{4,l}
Q_w(\varphi_l+\psi)+\chi^{\natural}_{4,l}(F'(w)+2)\psi)\\
&+\Delta\psi+\big(2-(1-\sum_{l=1}^2 \chi^{\natural}_{4,l})(F'(w)+2)\big)\psi+(1-\sum_{l=1}^2 \chi^{\natural}_{4,l})S(w)\\
&+\sum_{l=1}^k 2\nabla\chi^{\natural}_{3,l}\cdotp\nabla\varphi_l+\Delta\chi^{\natural}_{3,l}\varphi_l
+(1-\chi^{\natural}_{4,l})Q_w(\psi+\sum_{i=1}^2 \chi^{\natural}_{3,i}\varphi_i),
\end{aligned}
\end{equation}
where we have denoted
$$
Q_w(\varphi)=F(w+\varphi)-F(w)-F'(w)\varphi.
$$

Since we look for an $O(m)\times O(n)$-invariant solution, also $\psi$ and $\varphi_1,\varphi_2$ have to satisfy these symmetries. In other words, $\varphi_l=\phi^{\natural}_l$, for some functions $\phi_l:\Sigma_\eps\times\R\to\R$ of the $(\pe,t)$-variables, which are $O(m)\times O(n)$-invariant. Therefore we have to solve the system given by the equations
\begin{equation}
\label{eq_nl_psi}
\begin{aligned}
&\Delta\psi+\left(2-(1-\sum_{l=1}^2 \chi^{\natural}_{4,l})(F'(w)+2)\right)\psi+\left(1-\sum_{l=1}^2 \chi^{\natural}_{4,l}\right)S(w)\\
&+\sum_{l=1}^2 2\nabla\chi^{\natural}_{3,l}\cdotp\nabla\varphi_l+\Delta\chi^{\natural}_{3,l}\varphi_l
+(1-\chi^{\natural}_{4,l})Q_w\left(\psi+\sum_{i=1}^2 \chi^{\natural}_{3,i}\varphi_i\right)=0 \quad \hbox{in} \quad \R^{N+1},\quad l=1,2
\end{aligned}
\end{equation}
and
\begin{equation}
\label{eq_nl_phi_j}
\begin{aligned}
&\Delta_{\Sigma_\eps}\phi_l+\partial^2_t\phi_l+F'(v_\star)\phi_l+\chi_4 S(w^{{\sharp}}_l)+\chi_4
Q_{w^{\sharp}_{l}}(\phi_l+\psi^{\sharp}_{l})+\chi_4(F'(w^{\sharp}_{l})+2)\psi^{\sharp}_{l}\\
&\qquad+\chi_2(\Delta_{\mathcal{N}_{l,\eps}}-\partial^2_t-\Delta_{\Sigma_\eps})\phi_l+\left(F'(w^{\sharp}_{l})-F'(v_\star)\right)\chi_2\phi_l=0 \quad \hbox{in} \quad {\Sigma_{\eps} \times \R},\quad  l=1,2,
\end{aligned}
\end{equation}
where $\Delta_{\mathcal{N}_{l,\eps}}$ represents the Laplacian in the $(\pe,t)$-coordinates, given by Lemma \ref{lemma_coord_Fermi}.

\medskip
Using the decay of the error both along the surface and in the orthogonal direction, it is possible to prove the following result about the behaviour of the error far from the interfaces.

\begin{lemma}
\label{lemma_error_far}
In the previous notations, for some $\bar{\gamma}>0$,
\begin{equation}
\left|\left(1-\sum_{l=1}^2 \chi^{\natural}_{4,l}\right)S(w)\right|\le C\eps^{2+\bar{\gamma}}(|\eps\xi|^2+2)^{-\frac{2+\bar{\gamma}}{2}},\qquad\forall\,\xi\in\R^{N+1}.
\end{equation}
\end{lemma}

\begin{proof}
Here we estimate the error far from the interfaces. TO be more precise, let $\xi=X_{\eps,\h_l}(\pe,t)$ with $\pe\in\Sigma_\eps$ and
$$
\frac{1}{4\sqrt{2}}\left(\log(s(\eps\pe)^2+2)+2|\log\eps|\right)-3\le|t|
\le\frac{1}{4\sqrt{2}}\left(\log(s(\eps\pe)^2+2)+2|\log\eps|\right)-2.
$$

Then, for $\s=\s(\pe)$
\begin{equation}\notag
\begin{aligned}
e^{-\sqrt{2}|t|}(e^{-\sqrt{2}({\tt h}_{j+1}-{\tt h}_j)}+e^{-\sqrt{2}({\tt h}_j-{\tt h}_{j-1})})\le &
e^{-(1-\alpha)\sqrt{2}|t|}e^{-(1-2\alpha)\left(\log(s(\eps\pe)^2+2)+2|\log\eps|\right)}e^{-\alpha\sqrt{2}|t|}\\
\le &e^{-\left(\frac{1-\alpha}{4}+1-2\alpha\right)\left(\log(s(\eps\pe)^2+2)+2|\log\eps|\right)} e^{-\alpha\sqrt{2}|t|}\\
= & \eps^{\frac{5}{2}-\frac{9}{2}\alpha}(s(\eps\pe)^2+2)^{-\frac{5}{4}+\frac{9}{4}\alpha} e^{-\alpha\sqrt{2}|t|}\\
\le &C\eps^{\frac{5}{2}-4\alpha}(|s(\eps\pe)|^2+2)^{-\frac{5}{4}+\frac{9}{4}\alpha}e^{-\alpha\sqrt{2}|t|}\\
= &C\eps^{\frac{5}{2}-4\alpha}(|\eps\s|^2+2)^{-\frac{5}{4}+\frac{9}{4}\alpha}e^{-\alpha\sqrt{2}|t|}.
\end{aligned}
\end{equation}

Thus the conclusion follows by setting $\min\left\{\frac{5}{2}-4\alpha,\frac{5}{2}-\frac{9}{2}\alpha\right\}=:2+\bar{\gamma}>2$ with any  $\alpha \in (0, \frac{1}{9})$. 
\end{proof}
In view of Lemma \ref{lemma_error_far}, we can find a solution $\psi=\psi(\phi_1,\phi_2,{\tt h}_1,{\tt h}_2)$ to equation (\ref{eq_nl_psi}), for any fixed $\phi_1,\phi_2$ and ${\tt h}_1,{\tt h}_2$. This will be done in subsection \ref{subs_eq_far}, thanks to coercivity of the bilinear form associated with the linear operator. After that, we will plug this solution into system (\ref{eq_nl_phi_j}), which will be solved with respect to $(\phi_1,\phi_2,{\tt h}_1,{\tt h}_2)$.

\medskip
To solve (\ref{eq_nl_phi_j}) we will rely on the infinite dimensional Lyapunov-Schmidt reduction. This allows to overcome the fact that  the operator $\Delta_{\Sigma_\eps}+\partial^2_t+F'(v_\star)$ has a one-dimensional kernel spanned by $v'_\star$.  In order to set up the reduction scheme we set
\begin{equation}
\begin{aligned}
{\tt N}_{l}^{\sharp}(\psi,\phi_1,\phi_2,\h_1,\h_2)&:=\chi_4Q_{w_{l}^{{\sharp}}}(\phi_l+\psi_{l}^{{\sharp}})+\chi_4(F'(w_{l}^{{\sharp}})+2)\psi_{l}^{{\sharp}}+\chi_2(\Delta_{\mathcal{N}_{l,\eps}}-\partial^2_t-\Delta_{\Sigma_\eps})\phi_l\\
&\qquad+(F'(w_{l}^{{\sharp}})-F'(v_\star))\chi_2\phi_l,\\
P_{l}^{\sharp}(\psi,\phi_1,\phi_2,\h_1,\h_2)(\pe)&:=\int_\R \left(\chi_4 S(w_{l}^{{\sharp}})+{\tt N}_{l}^{\sharp}(\psi,\phi_1,\phi_2,\h_1,\h_2)\right)v'_\star(t) dt \quad\hbox{for} \quad\pe\in\Sigma_\eps.
\end{aligned}
\end{equation}

First we fix ${\tt h}_1,{\tt h}_2$ and we find a solution $(\phi_1,\phi_2)=(\phi_1({\tt h}_1,{\tt h}_2),\phi_2({\tt h}_1,{\tt h}_2))$ to the system
\begin{equation}
\label{eq_nl_phi_j_pr}
\begin{aligned}
\Delta_{\Sigma_\eps}\phi_l+\partial^2_t\phi_l+F'(v_\star)\phi_l=-\chi_4 S(w_{l}^{{\sharp}})-&{\tt N}_{l}^{\sharp}(\psi,\phi_1,\phi_2,\h_1,\h_2)+P_{l}^{\sharp}(\psi,\phi_1,\phi_2,\h_1,\h_2)(\pe)v'_\star, \\
\int_\R \phi_l(\pe,t)v'_\star(t)dt=&0,\qquad\forall\, \pe\in\Sigma_\eps, \quad  l=1,\,2
\end{aligned}
\end{equation}
where $\psi=\psi(\varphi_1,\varphi_2,\h_1,\h_2)$ is the solution found above. In other words, for any ${\tt h}_1,\, {\tt h}_2$ fixed, equation (\ref{eq_nl_phi_j}) can be solved up to Lagrange multipliers $P_{l}^{\sharp}(\psi,\phi_1,\phi_2,\h_1,\h_2)$, which depend on ${\tt h}_1$ and ${\tt h}_2$ (we recall that also $\psi,\,\phi_1$ and $\phi_2$ do depend on ${\tt h}_1$ and ${\tt h}_2$). This system, known as the \textit{auxiliary equation} will be treated in Subsection \ref{subs_aux_eq}. Finally we will determine $\h_1,\h_2$ by solving the system
\begin{equation}\notag
P_{l}^{\sharp}(\psi,\phi_1,\phi_2,\h_1,\h_2)=0, \qquad l=1,\,2,
\end{equation}
known as the \textit{bifurcation equation}. Equivalently, we have to choose ${\tt h}_1$ and ${\tt h}_2$ in order for the Lagrange multipliers to vanish. Integrating over $\R$ and using Lemma \ref{lemma_new_error}, it is possible to see that this is equivalent to solve a non-linear system of the form
\begin{equation}\label{JTNonhomogsystemsect5}
\begin{aligned}
&\eps^2 J_\Sigma{\tt h}_1+a_\star e^{-\sqrt{2}({\tt h}_2-{\tt h}_1)}=\eps^2 f_1(\pe,\h_1,\h_2)\\
&\eps^2 J_\Sigma{\tt h}_2-a_\star e^{-\sqrt{2}({\tt h}_2-{\tt h}_1)}=\eps^2 f_2(\pe,\h_1,\h_2)
\end{aligned}
\quad \hbox{in} \quad  \Sigma
\end{equation}
for ${\tt h}_1$ and ${\tt h}_2$ satisfying \eqref{cond_growth_hl} and \eqref{cond_growth_der_hl} with 
\begin{equation}
\label{dec_f_l}
|f_l(\pe,\h_1,\h_2)|\le c\eps^{2\mu}(s(\pe)^2+2)^{2+\mu} \quad \hbox{for} \quad \pe\in\Sigma,\quad\,l=1,2,
\end{equation}
where $\,\mu:=\min\{\gamma,\bar{\gamma}\}>0$ with $\gamma$ and $\bar{\gamma}$ begin the constants in  {Lemmas \ref{lemmaerrorSUzero} and \ref{lemma_error_far}} and  where from  
$$
a_{\star}= \|v'_{\star}\|^{-2}_{L^2(\R)}\int_{\R}6(1-v_\star^2)e^{-\sqrt{2}t}v'_{\star}(t)dt>0
$$
is the constant in \eqref{AClinearOrthCond1}.

\medskip
Setting
\begin{equation}\notag
\begin{cases}
{\tt v}_1:=\h_1+\h_2,\\
{\tt v}_2:=\h_2-\h_1,
\end{cases}\qquad
\begin{cases}
g_1:=f_1+f_2,\\
g_2:=f_2-f_1,
\end{cases}
\end{equation}
we reduce the bifurcation system  \eqref{JTNonhomogsystemsect5} reduces to to \begin{equation}\notag
\begin{aligned}
J_\Sigma{\tt v}_1=&g_1\left(\pe,\frac{{\tt v}_1-{\tt v}_2}{2},\frac{{\tt v}_1+{\tt v}_2}{2}\right),\\
\eps^2 J_\Sigma{\tt v}_2-2a_\star e^{-\sqrt{2}{\tt v}_2}=&\eps^2 g_2\left(\pe,\frac{{\tt v}_1-{\tt v}_2}{2},\frac{{\tt v}_1+{\tt v}_2}{2}\right).
\end{aligned}
\quad \hbox{in} \quad \Sigma.
\end{equation}
We look for a solution of the form 
$${\tt v}_l={\tt v}_{0,l}+\q_l,$$
where the pair $({\tt v}_{0,1},{\tt v}_{0,2})$ solves the homogeneous Jacobi-Toda system 
\begin{equation}
\begin{aligned}
J_\Sigma{\tt v}_{0,1}=&0\\
\eps^2 J_\Sigma{\tt v}_{0,2}-2a_\star e^{-\sqrt{2}{\tt v}_{0,2}}=&0.
\end{aligned}
\quad \hbox{in} \quad \Sigma.
\end{equation}
From Proposition \ref{prop_Jacobi_lin}, ${\tt v}_{0,1}=0$, while Theorem \ref{th_Liouville} with $
\delta =\eps^2 >0$, provides the existence  of ${\tt v}_{0,2}$. 
Consequently, the pair $(\q_1,\q_2)$ is determined by solving the nonlinear system
\begin{equation}
\label{syst_q}
\begin{aligned}
J_\Sigma{\tt q}_1&=\tilde{g}_1({\pe},\q_1,\q_2)\\
\eps^2 J_\Sigma \q_2+2\sqrt{2}a_\star e^{-\sqrt{2}{\tt v}}\q_2&=\eps^2\tilde{g}_2(\pe,\q_1,\q_2)
\end{aligned}
\quad \hbox{in} \quad \Sigma,
\end{equation}
where ${\tt v}={\tt v}_{0,2}$ has the asymptotic behavior of the approximate solution to the Jacobi-Toda equation described in Lemma \ref{lemmaapproximateslnJacToda}, i.e.
$$
{\tt v}_{0,2}\sim \frac{1}{\sqrt{2}}W\left(\frac{2\sqrt{2}a_\star}{\eps^2|A_{\Sigma}|^2}\right).
$$

The right-hand sides in \eqref{syst_q} are given by
\begin{equation}\notag
\begin{aligned}
\tilde{g}_1(\pe,\q_1,\q_2)&=g_1\left(\pe,-\frac{{\tt v}_{0,2}}{2}+\frac{{\tt q}_1-{\tt q}_2}{2},\frac{{\tt v_{0,2}}}{2}+\frac{{\tt q}_1+{\tt q}_2}{2}\right)\\
\tilde{g}_2(\pe,\q_1,\q_2)=&g_2\left(\pe,-\frac{{\tt v}_{0,2}}{2}+\frac{{\tt q}_1-{\tt q}_2}{2},\frac{{\tt v}_{0,2}}{2}+\frac{{\tt q}_1+{\tt q}_2}{2}\right)\\
&+2a_\star\eps^{-2} e^{-\sqrt{2}{\tt v}_{0,2}} Q(\q_2)-2\sqrt{2}a_\star\eps^{-2}e^{-\sqrt{2}{\tt v}}(e^{-\sqrt{2}({\tt v}_{0,2}-{\tt v})}-1)\q_2.
\end{aligned}
\end{equation}

For further details about system (\ref{syst_q}), we refer to subsection \ref{subs_bifo_eq}.

\medskip
\begin{remark}
{\rm The solution $({\tt h}_1, {\tt h}_2)$ of the system  \eqref{JTNonhomogsystemsect5} will have the form
\begin{equation}
\label{def_h_l}
\h_1=-\frac{{\tt v}_{0,2}}{2}+\frac{{\tt q}_1-{\tt q}_2}{2},\qquad\h_2=\frac{{\tt v_{0,2}}}{2}+\frac{{\tt q}_1+{\tt q}_2}{2}.
\end{equation}

From \eqref{est_approx_sol} in Lemma \ref{lemmaapproximateslnJacToda} and the developments in subsection 3.9, we are able to conclude that $({\tt h}_1, {\tt h}_2)$ 
satisfies \eqref{cond_growth_hl} and \eqref{cond_growth_der_hl}, since $({\tt q}_1, {\tt q}_2)$ are determined as small perturbations of solutions to the Jacobi-Toda system.

\medskip
Roughly speaking, the two connected components of the zero level set of ${u_\eps}$ diverge logarithmically from the cone $C_{m,n}$ at infinity.}
\end{remark}

\subsection{Function spaces}\label{subs_function_sp}

In order to treat equation (\ref{eq_nl_psi}) and system  \eqref{eq_nl_phi_j_pr} we introduce some function spaces.

For $\beta\in(0,1)$, $\mu>0$ and functions $g\in C^{0,\beta}_{loc}(\R^{N+1})$, we introduce the norm (c.f (\ref{normsJacLinTheory})):
\begin{equation}\notag
\|g\|_{{\infty, \mu}}:=\|(|\eps\xi|^2+2)^{\frac{\mu}{2}}g\|_{L^\infty(\R^{{N+1}})}.
\end{equation}
Moreover, we say that $g\in{\tt Y}^{{\natural,\beta}}_{{\mu}}$ if it is $O(m)\times O(n)$-invariant and the norm
\begin{equation}
\|g\|_{{\tt Y}^{{\natural,\beta}}_{{\mu}}}:=\sup_{\xi\in\R^{N+1}}(2+|\eps\xi|^2)^{\frac{2+\mu}{2}}\|g\|_{C^{0,\beta}(B_1(\xi))}
\end{equation}
is finite. %In these spaces, sometimes it will be useful to use the norm
%\begin{equation}
%\|g\|_{C^{0}_{p,\mu}(\R^{N+1})}:=\sup_{\xi\in\R^{N+1}}(1+|\eps\xi|^2)^{1+\mu}\|g\|_{L^p(B_1(\xi))}, \qquad p>1.
%\end{equation}

\medskip
We also say that a function $\psi\in C^{2,\beta}_{loc}(\R^{N+1})$ is in ${\tt X}^{{\natural,\beta}}_{{\mu}}$ if it is $O(m)\times O(n)$-invariant and the norm
\begin{equation}
\label{norm_C2pmu}
\|\psi\|_{{\tt X}^{{\natural,\beta}}_{{\mu}}}:=\sup_{\xi\in\R^{N+1}}(2+|\eps\xi|^2)^{\frac{2+\mu}{2}}\|D^2\psi\|_{C^{0,\beta}(B_1(\xi))}
+\|\nabla\psi\|_{{\infty, 2+\mu}}+\|\psi\|_{{\infty, 2+\mu}}
\end{equation}
is finite.

\medskip
Given $\beta\in (0,1)$, $\rho\in(0,\sqrt{2})$, $\mu>0$ and a function $f\in C^{0,\beta}_{loc}(\Sigma_\eps\times\R)$, we define the norm
\begin{equation}
\|f\|_{{\infty, \mu,\rho}}:=\|(s(\eps\pe)^2+2)^{\frac{2+\mu}{2}}\cosh(t)^\rho f\|_{L^\infty(\Sigma_\eps\times\R)},
\end{equation}
%We note that, since for any $y\in\Sigma_\eps$, we have $y=\eps^{-1}\Psi(\eps s(y),\x,\y)$, then we have $s(\eps y)=\eps s(y)$. 
Furthermore, we say that $f\in Y^{{\sharp,\beta}}_{{\mu, \rho}}$ if it is $O(m)\times O(n)$-invariant and
\begin{equation}
\|f\|_{Y^{{\sharp,\beta}}_{{\mu, \rho}}}:=\sup_{\pe\in \Sigma_\eps,\,t\in\R}(s(\eps\pe)^2+2)^{\frac{2+\mu}{2}}\cosh(t)^\rho \|f\|_{C^{0,\beta}(I_{\pe,t})},\qquad I_{\pe,t}:=B_1(\pe)\times(t,t+1) 
\end{equation}
is finite. Moreover, for $O(m)\times O(n)$-invariant functions $\phi\in C^{2,\beta}_{loc}(\Sigma_\eps\times\R)$, we say that $\phi\in X^{{\sharp,\beta}}_{{\mu, \rho}}$ if
\begin{equation}
\|\phi\|_{X^{{\sharp,\beta}}_{{\mu, \rho}}}:=\sup_{\pe\in\Sigma_\eps,\,t\in\R}(s(\eps\pe)^2+2)^{\frac{2+\mu}{2}}\cosh(t)^\rho \|D^2\phi\|_{C^{0,\beta}(I_{\pe,t})}+\|\nabla\phi\|_{{\infty, 2+\mu,\rho}}
+\|\phi\|_{{\infty, 2+\mu,\rho}}
\end{equation}
is finite.

\subsection{The equation far from the nodal set}\label{subs_eq_far}

The aim of this subsection  is to solve equation (\ref{eq_nl_psi}), with $\h_1,\,\h_2,\,\phi_1$ and $\phi_2$ fixed. Recall that we have set $\mu:=\min\{\gamma,\bar{\gamma}\}>0$, with $\gamma$ and $\bar{\gamma}$ defined as in {Lemmas \ref{lemmaerrorSUzero} and \ref{lemma_error_far}.} We also refer the reader back to subsections 2.7 and 3.8 for the respective definition of the spaces $\mathcal{C}^{2,\beta}_{\infty,\mu}(\Sigma)$ and $\mathcal{D}^{2,\beta}_{\mu,\frac{1}{2}}(\Sigma)$.

\medskip
\begin{proposition}
Let $\beta\in(0,\frac{1}{2})$, $\rho>0$ {be given} and let $\Lambda_0,\,\Lambda_1>0$ be fixed constants. Let also $\h_1,\,\h_2$ be %$O(m)\times O(n)$-invariant functions fulfilling (\ref{cond_growth_hl}) and (\ref{cond_growth_der_hl}).
of the form (\ref{def_h_l}), with $\q_1\in\mathcal{C}^{2,\beta}_{\infty,\mu}(\Sigma),\,\q_2\in\mathcal{D}^{2,\beta}_{\mu,\frac{1}{2}}(\Sigma)$ such that
$$\|\q_1\|_{\mathcal{C}^{2,\beta}_{\infty,\mu}(\Sigma)}<\Lambda_0\eps^\mu,\qquad
\|\q_2\|_{\mathcal{D}^{2,\beta}_{\mu,\frac{1}{2}}(\Sigma)}<\Lambda_0\eps^{\mu}.$$
Let $\phi_1,\,\phi_2\in X^{{\sharp,\beta}}_{{\mu, \rho}}$ be such that 
$$\|\phi_1\|_{X^{{\sharp,\beta}}_{{\mu, \rho}}}<\Lambda_1\eps^{2+\mu},$$ 
 {Set  $\varphi_l=\phi^{\natural}_l$}. Then there exists a unique solution $\psi:=\psi({\varphi_1,\varphi_2},\h_1,\h_2)\in {\tt X}^{{\sharp,\beta}}_{{\mu, \rho}}$ to equation (\ref{eq_nl_psi}) satisfying
\begin{equation}\notag
\begin{aligned}
\|\psi({\varphi_1,\varphi_2},\h_1,\h_2)\|_{{\tt X}^{{\sharp,\beta}}_{{\mu, \rho}}}  &\le\Lambda_2\eps^{2+\mu},\\
\|\psi({\varphi^1_1,\varphi^1_2},\h_1,\h_2)-\psi({\varphi^2_1,\varphi^2_2},\h_1,\h_2)\|_{{\tt X}^{{\sharp,\beta}}_{{\mu, \rho}}}&\le c\eps^{2+\mu}\left(\|{\varphi^1_1-\varphi^2_1}\|_{{\tt X}^{{\sharp,\beta}}_{{\mu, \rho}}}+\|{\varphi^1_2-\varphi^2_2}\|_{{\tt X}^{{\sharp,\beta}}_{{\mu, \rho}}}\right)\\
\|\psi({\varphi_1,\varphi_2},\h^1_1,\h^2_2)-\psi({\varphi_1,\varphi_2},\h^2_1,\h^2_2)\|_{{\tt X}^{{\sharp,\beta}}_{{\mu, \rho}}}&
\le c\eps^{2+\mu}\left(\|\q^1_1-\q^2_1\|_{\mathcal{C}^{2,\beta}_{\infty,\mu}(\Sigma)}
+\|\q^1_2-\q^2_2\|_{\mathcal{D}^{2,\beta}_{0,\varrho+\frac{1}{2}}(\Sigma)}\right),
\end{aligned}
\end{equation}
for some $c,\,\Lambda_2>0$. 
\label{prop_psi}
\end{proposition}

The proof basically consists of two steps. First we construct a right inverse of the operator 
$$-\Delta+V_\eps, \qquad V_\eps:=2-(1-\sum_{l=1}^2 \chi^{{\natural}}_{4,l})(F'(w)+2)$$ 
using the fact that the potential $V_\eps$ is positive and bounded away from $0$, uniformly in $\eps$ (see Proposition \ref{prop_right-inv-far}). After that, we will use a perturbation argument, based on the contraction mapping theorem. 

\medskip
To carry out the first step we consider the equation:
\begin{eqnarray}
\label{eq_far_lin}
-\Delta\psi+V_\eps(\xi)\psi=g.
\end{eqnarray}
\begin{proposition}
Let $\beta\in(0,1)$, $\mu>0$, and $g\in {\tt Y}^{{\natural}}$. Then, for $\eps>0$ small enough, there exists a solution $\psi:={\tt F}_3(g)\in {{\tt X}^{{\sharp,\beta}}_{{\mu, \rho}}}$ to (\ref{eq_far_lin}). Moreover, it satisfies
\begin{equation}
\|\Psi(g)\|_{{\tt X}^{{\sharp,\beta}}_{{\mu, \rho}}}\le c\|g\|_{{\tt Y}^{{\sharp,\beta}}_{{\mu, \rho}}},
\end{equation}
for some constant $c>0$.
\label{prop_right-inv-far}
\end{proposition}
First we state the following a priori estimate, which will be used to treat the linear problem.
%\begin{lemma}
%\label{lemma_apriori_est_far}
%Let $\beta\in(0,1)$, $\mu>0$ and $\psi$ be a bounded solution to (\ref{eq_far_lin}) with $g\in{\tt Y}$. Then 
%\begin{equation}
%\label{apriori_est_far}
%\|\psi\|_{L^\infty(\R^{N+1})}\le c\|g\|_{L^\infty(\R^{N+1})},
%\end{equation}
%for some constant $c>0$.
%\end{lemma}
%\begin{proof}
%If $\|\psi\|_{L^\infty(\R^{N+1})}=|\psi(\xi_0)|$ and, for instance, $\psi(\xi_0)>0$ is a maximum, then
%$$0<(2-\delta)\psi(\xi_0)\le -\Delta\psi(\xi_0)+V_\eps(\xi_0)\psi(\xi_0)=g(\xi_0)\le \|g\|_{L^\infty(\R^{N+1})}.$$
%If $\psi(\xi_0)<0$ is a minimum, then the previous argument applied to $-\psi$ gives (\ref{apriori_est_far}). On the other hand, if the maximum is not achieved, we consider a sequence $|\xi_k|\to\infty$ such that 
%$$|\psi(\xi_k)|\to\|\psi\|_{L^\infty(\R^{N+1})}$$
%and we set
%\begin{equation}\notag
%\psi_k(\xi):=\psi(\xi+\xi_k),\qquad V_k(\xi)=V_{\eps_k}(\xi+\xi_k),\qquad g_k(\xi)=g_k(\xi+\xi_k).
%\end{equation}
%Since $\|\psi_k\|_{L^\infty(\R^{N+1})}=\|\psi\|_{L^\infty(\R^{N+1})}$for any $k$, then by the Ascoli-Arzel? theorem there exists a subsequence of $\psi_k$ converging uniformly on compact subsets to a bounded solution $\psi_\ast$ to 
%$$-\Delta\psi_\ast+V_\ast(x)\psi_\ast=g_\ast$$
%in $R^{N+1}$, with $$\psi_\ast(0)=\|\psi_\ast\|_{L^\infty(\R^{N+1})}=\|\psi_\ast\|_{L^\infty(\R^{N+1})},
%\qquad\|g_\ast\|_{L^\infty(\R^{N+1})}=\|g\|_{L^\infty(\R^{N+1})}$$ 
%and $0<2-\delta\le V_\ast\le 2+\delta$. As a consequence, arguing as above, we have $$\|\psi_\ast\|_{L^\infty(\R^{N+1})}\le c\|g_\ast\|_{L^\infty(\R^{N+1})}.$$
%\end{proof}
\begin{lemma}
\label{lemma_dec_far}
Let $\beta\in(0,1)$, $\mu>0$ and $\psi$ be a bounded solution to (\ref{eq_far_lin}) with $g\in{{\tt Y}^{{\sharp,\beta}}_{{\mu, \rho}}}$. Then
$$\|\psi\|_{{\infty},2+\mu}\le c\|g\|_{{\infty},2+\mu}.$$
\end{lemma}
\begin{proof}
The idea is to take a point $\xi_0\in\R^{N+1}$ and to prove that
$$|\psi(\xi_0)|\le \lambda (|\eps\xi_0|^2+2)^{-\frac{2+\mu}{2}}.$$
In order to do so, we take $\nu>0$ arbitrary, $\lambda>0$ (to be fixed) and we compare $\psi$ with the barrier
$$v_{\lambda,\nu}(\xi):=\lambda (|\eps\xi|^2+2)^{-\frac{2+\mu}{2}}+\nu (|\eps\xi|^2+2)^{\frac{2+\mu}{2}}$$
in a ball of radius $R_0>|\xi_0|$ so large that 
$${|\psi(\xi)|}\le\|\psi\|_{L^\infty(\R^{N+1})}\le \nu ((\eps R_0)^2+2)^{\frac{2+\mu}{2}},\qquad \forall\, \xi\in\partial B_{R_0}.$$
In fact, in such a ball, we have
$$(-\Delta+V_\eps)(\psi-v_{\lambda,\nu})\le (\|g\|_{{\infty}, 2+\mu}-\lambda(2-\delta))(|\eps\xi|^2+2)^{-\frac{2+\mu}{2}}=0$$
with $0<2-\delta<2$, provided $\eps$ is small enough and 
$$\lambda=\frac{\|g\|_{{\infty}, 2+\mu}}{2-\delta}.$$
Therefore, applying the maximum principle and letting $\nu\to 0$, we have the statement.
\end{proof}
Now we can prove Proposition \ref{prop_right-inv-far}.
\begin{proof}[Proof of Proposition \ref{prop_right-inv-far}]
We take a smooth cutoff function $\chi:\R\to\R$ such that $\chi(t)=1$ for $t<1$ and $\chi(t)=0$ for $t>2$. For $k\ge 0$, we set $\chi_R(x)=\chi(|\eps x|-R)$ and $g_R:=\chi_R g$, so that $g_R\to g$ uniformly as $R\to\infty$ and we consider the unique solution $\psi_R$ to the problem
\begin{equation}\notag
\begin{cases}
-\Delta \psi_R+V_\eps(x)\psi_R=g_R, \qquad\text{in $\R^{N+1}$}\\
\psi_R\in H^1(\R^{N+1}).
\end{cases}
\end{equation} 
Since $g$ is $O(m)\times O(n)$-invariant, then, by uniqueness, so is  $\psi_R$. Moreover, by a bootstrap argument, it is possible to prove that $\psi_R\in L^\infty(\R^{N+1})$. Thus Lemma \ref{lemma_dec_far}, gives the bound 
$$\|\psi_R\|_{L^\infty(\R^{N+1})}\le\|\psi_R\|_{{\infty},2+\mu}\le c\|g_R\|_{2+\mu,\ast}\le \|g\|_{{\infty},2+\mu}.$$

By the standard {{ellitpic regularity and}} and the Ascoli-Arzel? theorem, there exists a sequence $R_k\to \infty$ such that $\psi_{R_k}$ converges uniformly on compact subsets to a bounded solution $\psi$ to (\ref{eq_far_lin}), which itself satisfies 
$$\|\psi\|_{{\infty}, 2+\mu}\le c\|g\|_{{\infty}, 2+\mu}.$$
To conclude, the estimate in the norm $\|\cdotp\|_{{\tt X}^{{\sharp,\beta}}_{{\mu, \rho}}}$ follows from the elliptic estimates.
\end{proof}
Now we can prove Proposition \ref{prop_psi}.
\begin{proof}[Proof of Proposition \ref{prop_psi}]
Equation (\ref{eq_nl_psi}) can be reduce to the fixed point problem
\begin{equation}
%\begin{aligned}
\psi=-{\tt F}_3\left((1-\sum_{l=1}^2 \chi^{{\natural}}_{4,l})S(w)+\sum_{l=1}^2 2\nabla\chi^{{\natural}}_{3,l}\cdotp\nabla\varphi_l+\Delta\chi^{{\natural}}_{3,l}\varphi_l
+(1-\chi^{{\natural}}_{4,l})Q_w(\psi+\sum_{i=1}^2 \chi^{{\natural}}_{3,i}\varphi_i)\right),
%\end{aligned}
\end{equation}
which can be uniquely solved in the ball
$$B_{\Lambda_1}:=\{\psi\in {{\tt X}^{{\sharp,\beta}}_{{\mu, \rho}}}:\,\|\psi\|_{{\tt X}^{{\sharp,\beta}}_{{\mu, \rho}}}<\Lambda_1\eps^{2+\mu}\},$$
using  the contraction mapping theorem and thanks to Lemmas \ref{lemma_new_error} and \ref{lemma_error_far}. {We leave the details to the reader and refer to \cite{DELPINOKOWALCKZYKWEI2011}, section $5$ for details on similar developments.}
\end{proof}

%\red{From this point on we should reorganize the argument, more care is needed. To solve (\ref{eq_nl_phi_j}) need to introduce Lagrange multipliers.}\green{There was just a mistake in the reference below. What we solve is (\ref{eq_nl_phi_j_pr}), not (\ref{eq_nl_phi_j})}. \blue{Look at the introductory paragraph below. I think with that we can leave the previous paragraphs as they are.}

\subsection{The auxiliary equation}\label{subs_aux_eq}

In this subsection we will deal with the auxiliary system (\ref{eq_nl_phi_j_pr})  with $\h_1,\,\h_2$ fixed.

\medskip
First we introduce the spaces
$$
\begin{aligned}
&X_{\bot,\mu,\rho}^{{\sharp},\beta}:=\left\{\phi\in X^{{\sharp},\beta}_{\mu,\rho}:\,\int_\R \phi(\pe,t)v'_\star(t)dt=0,\,\forall\, \pe\in\Sigma_\eps\right\},\\ 
&Y_{\bot,\mu,\rho}^{{\sharp},\beta}:=\left\{f\in Y^{{\sharp},\beta}_{\mu,\rho}:\,\int_\R f(\pe,t)v'_\star(t)dt=0,\,\forall\, \pe\in\Sigma_\eps\right\}.
\end{aligned}$$

\medskip
\begin{proposition}
Let $\beta\in(0,\frac{1}{2})$, $\rho>0$ be given and let $\h_1,\,\h_2$ be %$O(m)\times O(n)$-invariant functions fulfilling (\ref{cond_growth_hl}) and (\ref{cond_growth_der_hl}).
of the form (\ref{def_h_l}), with $\q_1\in\mathcal{C}^{2,\beta}_{\infty,\mu}(\Sigma),\,\q_2\in\mathcal{D}^{2,\beta}_{\mu,\frac{1}{2}}(\Sigma)$ such that
$$\|\q_1\|_{\mathcal{C}^{2,\beta}_{\infty,\mu}(\Sigma)}<\Lambda_0\eps^{\mu},\qquad
\|\q_2\|_{\mathcal{D}^{2,\beta}_{\mu,\frac{1}{2}}(\Sigma)}<\Lambda_0\eps^{\mu},$$
with $\Lambda_0>0$ fixed. Then there exists a unique solution $(\phi_1,\phi_2)=(\phi_1(\h_1,\h_2),\phi_2(\h_1,\h_2))\in X_{\bot,\mu,\rho}^{\sharp,\beta}\times X_{\bot,\mu,\rho}^{\sharp,\beta}$ to system (\ref{eq_nl_phi_j_pr}) satisfying
\begin{equation}
\begin{aligned}
\|\phi_1(\h_1,\h_2)\|_{X_{\mu,\rho}^{\sharp,\beta}}+\|\phi_2(\h_1,\h_2)\|_{X_{\mu,\rho}^{\sharp,\beta}}&\le \Lambda_1 \eps^{2+\mu},\\
\|\phi_l(\h^1_1,\h^1_2)-\phi_l(\h^2_1,\h^2_2)\|_{X_{\mu,\rho}^{\sharp,\beta}}&\le c\eps^{2+\mu}(\|\q^1_1-\q^2_1\|_{\mathcal{C}^{2,\beta}_{\infty,\mu}(\Sigma)}
+\|\q^1_2-\q^2_2\|_{\mathcal{D}^{2,\beta}_{\mu,\frac{1}{2}}(\Sigma)}), \qquad l=1,2,
\end{aligned}
\end{equation}
for some $c,\,\Lambda_1>0$.
\label{prop_phi_l}
\end{proposition}
Once again, first we will construct a right inverse of the operator $-\Delta_{\Sigma_\eps}-\partial^2_t+(3v_\star^2-1)$ under the suitable orthogonality condition, then we will apply a fixed point argument. For this purpose, we consider the linear problem
\begin{equation}
\label{eq_lin_product}
-\Delta_{\Sigma_\eps}\phi-\partial^2_t \phi+(3v_\star^2-1)\phi=f \qquad\text{in $\Sigma_\eps\times\R$.}
\end{equation}
In order to be able to solve equation (\ref{eq_lin_product}), $f$ has to satisfy the orthogonality condition
\begin{equation}
\label{ort_v'star}
\int_\R f(\pe,t)v'_\star(t) dt=0 \quad \hbox{for} \quad \pe\in\Sigma_\eps,
\end{equation}
and we look for a solution $\phi$ which satisfies (\ref{ort_v'star}). The aim of this subsection is to prove the following result.
\begin{proposition}
\label{prop_eq_lin_product}
Let $\beta\in(0,1)$, $\rho\in(0,\sqrt{2})$ and $\mu\in(0,1)$. Then, for any $f\in Y_{\bot,\mu,\rho}^{\sharp,\beta}$, there exists a unique solution $\phi:={\tt F}_4(f)\in X_{\bot,\mu,\rho}^{\sharp,\beta}$ to (\ref{eq_lin_product}). Moreover
$$\|\phi\|_{X^{\sharp,\beta}_{\mu,\rho}}\le c\|f\|_{Y^{\sharp,\beta}_{\mu,\rho}}.$$
\end{proposition}
The proof of Proposition \ref{prop_eq_lin_product} involves several steps, which will be dealt with the aid of some Lemmas and Remarks.\\

Since the decay of $f$ along the surface is slow, $f$ is not necessarily in $L^2(\Sigma_\eps\times\R)$, thus we start with a truncated problem with right-hand side $f_R:=f\chi_R$, where $\chi_R(\pe):=\chi(s(\eps\pe)-R)$ and $\chi:\R\to\R$ is a smooth cutoff function such that 
$$\chi(\tau)=
\begin{cases}
1,\qquad \tau>2\\
0,\qquad \tau<1
\end{cases}, \qquad 0\le \chi\le 1.$$
\begin{lemma}[Existence for a truncated problem]
\label{lemma_existence}
Let $\beta\in(0,1)$, $\rho\in(0,\sqrt{2})$, $\gamma\in(0,1)$ and $f\in Y_{\bot,\mu,\rho}^{\sharp,\beta}$. Then, for any $R>0$, there exists a unique solution $\phi_R\in H^1(\Sigma_\eps\times\R)$ to (\ref{eq_lin_product}) such that
\begin{equation}\notag
\int_\R \phi_R(\pe,t)v'_\star(t)dt=0, \qquad\forall\, \pe\in\Sigma\eps.
\end{equation}
Moreover, $\phi\in C^{2,\beta}_{loc}(\Sigma_\eps\times\R)$, it is $O(m)\times O(n)$-invariant and bounded.
\end{lemma}
\begin{proof}
Since we have multiplied by a cutoff function an we have exponential decay in $t$, $f_R\in L^2(\Sigma\times\R)$, thus equation (\ref{eq_lin_product}) can be attached with variational techniques. In fact, due to the spectral decomposition of the ordinary differential operator $-\partial^2_t+(3v_\star^2-1)$, the corresponding functional
$$\int_{\Sigma_\eps\times\R} |\nabla_{\Sigma_\eps}\phi|^2+\int_{\Sigma_\eps\times\R}(\partial_t\phi)^2+\int_{\Sigma_\eps\times\R}(3v_\star^2-1)\phi^2 d\sigma(\pe)dt-\int_{\Sigma_\eps\times\R} f_R\phi d\sigma(\pe)dt$$
is coercive and lower semicontinuous on the closed subspace
$$Z:=\left\{\phi\in H^1(\Sigma_\eps\times\R):\,\int_\R \phi(\pe,t)v'_\star(t) dt=0,\,\forall\, \pe\in\Sigma_\eps\right\},$$
thus it has a unique minimiser $\phi_R\in Z$, which fulfils
\begin{equation}
\label{weak_sol}
\int_{\Sigma_\eps\times\R}\left(\langle\nabla_{\Sigma_\eps}\phi_R,\nabla_{\Sigma_\eps}\psi\rangle_{\Sigma_\eps}
+\partial_t\phi_R\partial_t\psi+(3v_\star^2-1)\phi_R\psi\right) d\sigma(\pe)dt=\int_{\Sigma_\eps\times\R}f_R\psi d\sigma(\pe)dt, \qquad\forall\,\psi\in Z.
\end{equation}
In order to prove that $\phi$ is a true weak solution we have to show that (\ref{weak_sol}) is satisfied for any $\psi\in H^1(\Sigma_\eps\times\R)$. This follows from a direct computation after writing any $\psi\in H^1(\Sigma_\eps\times\R)$ as 
$$\psi=\tilde{\psi}+a(\pe)v'_\star(t), \qquad\tilde{\psi}\in Z,\qquad a(\pe):=\frac{\int_\R \psi(\pe,t)v'_\star(t)dt}{\int_\R v'_\star(t)^2dt}$$
and we use the fact that $v'_\star$ is in the kernel of $-\Delta_{\Sigma_\eps}-\partial^2_t+(3v_\star^2-1)$. Symmetry and regularity for $\phi_R$ follow from symmetry and regularity of $f$ and uniqueness. The fact that $\phi\in L^\infty(\Sigma_\eps\times\R)$ follows from a bootstrap argument. %In fact $H^1(\Omega_{y,t,1})\subset L^p(\Omega_{y,t,1})$, for any $p\in[2,2^\star]$, $2^\star:=\frac{2(N+1)}{N-1}$ and for any set of the form $\Omega_{y,t,1}:=B_1(y)\times(t-1,t+1)\subset\Sigma_\eps\times\R$. Moreover, by the elliptic estimates, 
%\begin{equation}
%\label{ell_estimates}
%\begin{aligned}
%&\|\phi_R\|_{W^{2,2^\star}(\Omega_{y,t,\frac{1}{2}})}\le c(\|\phi_R\|_{L^{2^\star}(\Omega_{y,t,1})}+\|f_R\|_{L^2(\Omega_{y,t,1})})\\
%&c(\|\phi_R\|_{H^1(\Sigma_\eps\times\R)}+\|f_R\|_{L^2(\Sigma_\eps\times\R)})<\infty
%\end{aligned}
%\end{equation}
%By the Sobolev embeddings, we can iterate the argument to show that $\phi_R\in W^{2,p_k}_{loc}(\Sigma_\eps\times\R)$ for any $k\ge 0$, where the sequence $p_k$ is defined inductively by $p_0:=2^\star$ and $p_{k+1}:=\frac{Np_k}{N-2 p_k}$, therefore $\phi_R\in L^\infty_{loc}(\Sigma_\eps\times\R)$. Since the right-hand side of (\ref{ell_estimates}) is independent of $(y,t)$, we conclude that $\phi\in L^\infty(\Sigma_\eps\times\R)$.
\end{proof}
In order to prove the decay of the solution and the estimates in the required weighted norms we need an a priori estimate, which relies on the following well-known result, which is proved, for instance in \cite{DELPINOKOWALCKZYKWEI2011}.
\begin{lemma}
Let $\phi$ be a bounded solution to 
$$-\Delta_{\R^N}\phi-\partial^2_t\phi+(3v_\star^2-1)\phi=0, \qquad\forall\,(\mathbf{y},t)\in\R^N\times\R.$$
Then $\phi(\mathbf{y},t)=cv'_\ast(t)$, for some constant $c\in\R$.
\end{lemma}
\begin{lemma}[A priori estimate]
\label{lemma_Apriori_est}
Let $\beta\in(0,1)$, $\rho\in(0,\sqrt{2})$, $\mu\in(0,1)$ and $f\in Y_{\bot,\mu,\rho}^{\sharp,\beta}$. Let $\phi$ be a bounded solution to (\ref{eq_lin_product}). Then
$$\|\phi\|_{L^\infty(\Sigma_\eps\times\R)}\le c\|f\|_{L^\infty(\Sigma_\eps\times\R)}.$$
\end{lemma}
\begin{proof}
Assume, by contradiction, that there exist a sequence $\eps_n\to 0$ and sequences $\phi_n,\, f_n\in L^\infty(\Sigma_\eps\times\R)$ such that
\begin{equation}\notag
\begin{aligned}
&-\Delta_{\Sigma_{\eps_n}}\phi_n-\partial^2_t \phi_n+(3v_\star^2-1)\phi_n=f_n\\
&\int_\R \phi_n(\pe,t)dt=0\qquad\forall\, \pe\in\Sigma_{\eps_n}\\
&\|\phi_n\|_{L^\infty(\Sigma_\eps\times\R)}=1,\qquad\|f_n\|_{L^\infty(\Sigma_\eps\times\R)}\to 0.
\end{aligned}
\end{equation}
Let us consider a sequence $(\pe_n,t_n)\in\Sigma_{\eps_n}\times\R$ such that $|\phi_n(\pe_n,t_n)|\ge \frac{1}{2}$ and a parametrisation $Y:B_\theta\subset\R^N\to\Sigma$ of $\Sigma$ such that $Y(0)=\eps_n \pe_n$ and the metric at $0$ is the identity. In these coordinates, the laplacian reads
$$\Delta_{\Sigma_{\eps_n}}:=g^{ij}(\eps_n \mathbf{y})\partial_{ij}+\eps_n b^i(\eps_n \mathbf{y})\partial_i$$
so that the equation reads
$$-g^{ij}(\eps_n \mathbf{y})\partial_{ij}\tilde{\phi}_n-\eps_n b^i(\eps_n \mathbf{y})\partial_i\tilde{\phi}_n-\partial^2_t\tilde{\phi}_n+(3v_\star^2-1)\tilde{\phi}_n=\tilde{f}_n, \qquad\forall\,(\mathbf{y},t)\in B_{\theta\eps^{-1}}\times\R,$$
where $\tilde{\phi}_n(\mathbf{y},t):=\phi_n(\pe,t)$ and $\tilde{f}_n(\mathbf{y},t):=f_n(\pe,t)$. Note that $|\tilde{\phi}_n(0,t_n)|=|\phi_n(\pe_n,t_n)|\ge \frac{1}{2}$.\\

Let us assume first that $t_n$ is bounded. Then, up to a subsequence, $t_n\to t_\infty\in\R$ and, by the Ascoli-Arzel? theorem, $\tilde{\phi}_n$ converges uniformly on compact subsets to a bounded solution $\phi_\infty$ to
$$-\Delta_{\R^N}\phi_\infty-\partial^2_t \phi_\infty+(3v_\star^2-1)\phi_\infty=0$$
in $\R^{N+1}$ such that $|\phi_\infty(0,t_\infty)|\ge\frac{1}{2}$, which yields that $\phi_\infty=C v'_\star(t)$, with $C\ne 0$. However, the orthogonality condition
$$\int_\R \phi_\infty v'_\star dt=0,\qquad\forall\mathbf{y}\in\R^N$$
yields that $C=0$, a contradiction.\\

If $t_n$ is unbounded, say $t_n\to \infty$, the situation is similar. In this case we set
$$\phi^\sharp_n(\mathbf{y},t):=\tilde{\phi}_n(\mathbf{y},t+t_n)$$
and we get, in the limit, a bounded solution $\phi_\infty$ to
$$-\Delta_{\R^N}\phi_\infty-\partial^2_t \phi_\infty+2\phi_\infty=0$$
in $\R^{N+1}$ such that $|\phi_\infty(0,\infty)|\ge\frac{1}{2}$, which contradicts the maximum principle and the boundedness of $\phi_\infty$.
\end{proof}
\begin{lemma}[Decay in $t$]
\label{lemma_dec_t}
Let $\beta\in(0,1)$, $\rho\in(0,\sqrt{2})$, $\mu\in(0,1)$ and $f\in Y_{\bot,\mu,\rho}^{\sharp,\beta}$. Let $\phi$ be a bounded solution to (\ref{eq_lin_product}). Then
$$|\phi(\pe,t)|\cosh(t)^\rho\le c(\|\phi\|_{L^\infty(\Sigma_\eps\times\R)}+\|f \cosh(t)^\rho\|_{L^\infty(\Sigma_\eps\times\R)}), \qquad\forall \, (\pe,t)\in\Sigma_\eps\times\R.$$
\end{lemma}
\begin{proof}
The proof relies on comparing the solution with a barrier. For any $\nu>0$ and $\lambda>0$, we set
$$v_{\lambda,\nu}(t):=(\lambda\cosh^{-\rho}(t)+\nu\cosh^\rho(t))e^{\nu R_0}.$$
The idea is to apply the maximum principle to bounded domain 
$$\Omega:=\Sigma_{\eps,\eps^{-1}R_0}\times(t_1,t_2), \qquad\Sigma_{\eps,\eps^{-1}R_0}:=\Sigma_\eps\cap B_{\eps^{-1}R_0},$$
with $R_0>0,t_1>0,t_2>0$ to be determined. First we note that, in order to apply the maximum principle, $f'(v_\star(t))$ must be uniformly negative, thus we choose $t_1>0$ such that $1-3v_\star^2<-1-\frac{\gamma^2}{2}$ for $|t|>t_1$. Then we observe that, since $e^{\nu R_0}\ge 1$,  for $\pe\in\Sigma_\eps$ and $t=t_1$ we have
$$\phi(\pe,t_1)\le\|\phi\|_{L^\infty(\Sigma_\eps\times\R)}\le \lambda \cosh^{-\rho}(t_1) \le v_{\lambda,\nu}(t_1)$$
provided $\lambda\ge\|\phi\|_{L^\infty(\Sigma_\eps\times\R)}\cosh^{\rho}(t_1)$. For $|t|=t_2$, we have
$$\phi(\pe,t_2)\le \|\phi\|_{L^\infty(\Sigma_\eps\times\R)}\le \nu\cosh^\rho(t_2)\le v_{\lambda,\nu}(t_2)$$
provided $t_2$ is large enough. Moreover, on $\partial \Sigma_{\eps,\eps^{-1}R_0}\times (t_1,t_2)$, we have
$$v_{\lambda,\nu}(t)\ge (\lambda\cosh^{-\rho}(t_2)+\nu\cosh^\rho(t_1)) e^{\nu R_0}\ge \|\phi\|_{L^\infty(\Sigma_\eps\times\R)}\ge \phi(\pe,t),$$
provided $R_0>0$ is large enough. Differentiation shows that in $\Omega$ the following differential inequality holds
\begin{equation}\notag
(-\Delta_{\Sigma_\eps}-\partial^2_t+(3v_\star^2-1))(\phi-v_{\lambda,\nu})\le \left(\|f\cosh(t)^\rho\|_{L^\infty(\Sigma_\eps\times\R)}-\lambda e^{\nu R_0}\left(1-\frac{\rho^2}{2}\right)\right) \cosh(t)^{-\rho}\le 0
\end{equation}
provided
$$\lambda\ge \frac{2\|f\cosh(t)^\rho\|_{L^\infty(\Sigma_\eps\times\R)}}{e^{\nu R_0}(2-\rho^2)}.$$
In conclusion, applying the maximum principle and letting $\nu\to 0$, we have
$$\phi\le \lambda \cosh^{-\rho}(t) \text{ in $\Sigma_\eps\times(t_1,\infty)$}, \qquad \lambda=\|\phi\|_{L^\infty(\Sigma_\eps\times\R)}\cosh^{\rho}(t_1)
+\frac{2\|f\cosh^\rho(t)\|_{L^\infty(\Sigma_\eps\times\R)}}{2-\rho^2}.$$
If $\phi$ satisfies the hypothesis of the Lemma, then also $-\phi$ and $\phi(\pe,-t)$ do, hence the proof is concluded.
\end{proof}
\begin{lemma}[Decay in $\pe$]
\label{lemma_dec_y}
Let $\beta\in(0,1)$, $\rho\in(0,\sqrt{2})$, $\mu\in(0,1)$ and $f\in Y_{\bot,\mu,\rho}^{\sharp,\beta}$. Let $\phi$ be a bounded solution to (\ref{eq_lin_product}). Then
$$\|\phi\|_{2+\mu,\rho,\infty}\le c(\|\phi\|_{L^\infty(\Sigma_\eps\times\R)}+\|f\|_{2+\mu,\rho,\infty}).$$
\end{lemma}
\begin{proof}
For $\pe\in\Sigma_\eps$, we define 
$$\varphi(\pe):=\int_\R \phi(\pe,t)^2dt.$$
Due to the exponential decay in $t$ provided by Lemma \ref{lemma_dec_t}, $\psi$ is well defined and bounded with
\begin{equation}\notag
\|\varphi\|_{L^\infty(\Sigma_\eps\times\R)}\le c(\|\phi\|_{L^\infty(\Sigma_\eps\times\R)}+\|f \cosh(t)^\rho\|_{L^\infty(\Sigma_\eps\times\R)}).
\end{equation}
A computation shows that
\begin{equation}\notag
\Delta_{\Sigma_\eps}\varphi=\int_\R (2\phi\Delta_{\Sigma_\eps}\phi+2|\nabla_{\Sigma_\eps}\phi|^2)dt,
\end{equation}
thus, multiplying (\ref{eq_lin_product}) by $\phi$ and integrating over $\R$,
$$\Delta_{\Sigma_\eps}\varphi= 2\int_\R |\nabla_{\Sigma_\eps}\phi|^2 dt+2\int_\R (\partial_t\phi)^2 dt+2\int_\R(3v_\star^2-1)\phi^2 dt-2\int_\R \phi f dt,$$
therefore, by the spectral properties of the ordinary differential operator $-\partial^2_t+(3v_\star^2-1)$, we have
$$\Delta_{\Sigma_\eps}\varphi\ge 3\int_\R \phi^2 dt-2\int_\R f\phi dt,$$
and hence, by the Young inequality, $\psi\ge 0$ satisfies the differential inequality
$$-\Delta_{\Sigma_\eps}\varphi(\pe)+2\varphi(\pe)\le\int_\R f(\pe,t)^2 dt,\qquad\forall\, \pe\in\Sigma_\eps.$$

Therefore, using the barrier
$$w_{\lambda,\nu}(\pe):=\lambda(s(\eps\pe)^2+2)^{-(2+\mu)}+\nu(s(\eps\pe)^2+2)^{2+\mu},$$
with $$\lambda:=\frac{2\|f\|_{2+\mu,\rho,\infty}^2}{2-\delta}+c(\mu)\|\varphi\|_{L^\infty(\Sigma_\eps\times\R)}$$
and $\nu>0$ arbitrarily small, the maximum principle gives
$$0\le \varphi(\pe)\le c\|f\|_{2+\mu,\rho,\infty}^2(s(\eps\pe)^2+2)^{-\frac{2+\mu}{2}}$$
or equivalently
$$(s(\eps\pe)^2+2)^{2+\mu}\int_\R \phi(\pe,t)^2 dt\le c\|f\|_{2+\mu,\rho,\infty}^2.$$

By the elliptic estimates, we have
$$(s(\eps\pe)^2+2)^{\frac{2+\mu}{2}}\|\phi\|_{L^\infty(I_{\pe,t})}\le c\|f\|_{2+\mu,\rho,\infty}, \qquad\forall \pe\in\Sigma_\eps.$$

In order to prove that the decay along the surface is uniform in $t$, we use the barrier
$$\bar{w}_{\lambda,\nu}(\pe,t):=\bar{c}(\mu,\rho)\|f\|_{2+\mu,\rho,\infty}(s(\eps\pe)^2+2)^{-\frac{2+\mu}{2}}\cosh(t)^{-\rho}
-\nu(s(\eps\pe)^2+2)^{\frac{2+\mu}{2}}\cosh(t)^\rho$$
where $\bar{c}(\mu,\rho)>0$ is a suitable constant, in a region of the form
$$\left\{(\pe,t)\in\Sigma_\eps\times\R:t_0<t<t_1,\, |s(\eps\pe)|\le s_0\right\},$$
with $t_0$ so large that $3v_\star^2-1>1+\frac{\rho^2}{2}$ for $t\in (t_0,\infty)$. This concludes the proof.
\end{proof}

\medskip
Now we can conclude the proof of Proposition \ref{prop_eq_lin_product}.
\begin{proof}
Given $f\in Y_{\bot,\mu,\rho}^{\sharp,\beta}$, for any $R>0$, by Lemma \ref{lemma_existence} it is possible to find a bounded $O(m)\times O(n)$-invariant solution $\phi_R\in H^1(\Sigma_\eps\times\R)$ to the truncated equation
$$-\Delta_{\Sigma_\eps}\phi_R-\partial^2_t\phi_R+(3v_\star^2-1)\phi_R=f_R$$
which satisfies the orthogonality condition (\ref{ort_v'star}) and, by the a priori estimate provided in Lemma \ref{lemma_Apriori_est},
$$\|\phi_R\|_{L^\infty(\Sigma_\eps\times\R)}\le c\|f_R\|_{L^\infty(\Sigma_\eps\times\R)}\le c\|f\|_{L^\infty(\Sigma_\eps\times\R)}.$$

Therefore, by the Ascoli-Arzel? theorem, there exists a sequence $R_k\to\infty$ such that $\phi_{R_k}$ converges uniformly on compact subsets to a bounded solution $\phi$ to (\ref{eq_lin_product}).

\medskip
Since $\phi_R$ satisfies the orthogonality condition (\ref{ort_v'star}) and is $O(m)\times O(n)$-invariant, then also $\phi$ does. Moreover, by Lemmas \ref{lemma_dec_t} and $\ref{lemma_dec_y}$, $\phi$ has the suitable decay both in $t$ and in $\pe$ and 
$$\|\phi\|_{2+\mu,\rho,\infty}\le c(\|\phi\|_{L^\infty(\Sigma_\eps\times\R)}+\|f\|_{2+\mu,\rho,\infty}).$$

Once again by the a priori estimate, namely by Lemma \ref{lemma_Apriori_est},
$$\|\phi\|_{2+\mu,\rho,\infty}\le \|f\|_{2+\mu,\rho,\infty}.$$

By the elliptic estimates, we can see that $\phi\in X_0$ and
$$\|\phi\|_{X_{\bot,\mu,\rho}^{\sharp,\beta}}\le c\|f\|_{Y_{\bot,\mu,\rho}^{\sharp,\beta}}.$$
\end{proof}
Now we can prove Proposition \ref{prop_phi_l}.
\begin{proof}
System \eqref{eq_nl_phi_j_pr} can be formulated as a fixed point problem in the form
$$\phi_l={\tt F}_4\left(-\chi_4 S(w)_{\ast,l}-{\tt N}_l(\psi,\phi_1,\phi_2,\h_1,\h_2)+P_l(\psi,\phi_1,\phi_2,\h_1,\h_2)(\pe)v'_\star\right),\qquad l=1,2,$$
where $\psi=\psi(\phi_1,\phi_2,\h_1,\h_2)$ is the correction found in Proposition \ref{prop_psi}.

\medskip
Thanks to Lemma \ref{lemma_new_error} and to the size of $\psi$ in $\eps$, the above problem has a unique solution in the ball
$$B_{\Lambda_1}=\{(\phi_1,\phi_2)\in X_{\bot,\mu,\rho}^{\sharp,\beta}\times X_{\bot,\mu,\rho}^{\sharp,\beta}:\,\|\phi_1\|_{X_{\mu,\rho}^{\sharp,\beta}}+\|\phi_2\|_{X_{\mu,\rho}^{\sharp,\beta}}<\Lambda_1\eps^{2+\mu}\},$$
provided $\Lambda_1>0$ is large enough. The details of the nonlinear argument are similar to the proof of Proposition \ref{prop_psi}. A similar proof can be found in Section $6$ from \cite{DELPINOKOWALCKZYKWEI2011}. This concludes the proof.
\end{proof}

\subsection{The bifurcation equation}\label{subs_bifo_eq}

We recall that the bifurcation equation is actually a system, given by
$$P_{{l}}^{\sharp}(\psi,\phi_1,\phi_2,\h_1,\h_2)=0, \qquad l=1,\,2,$$
where $\psi,\,\phi_1$ and $\phi_2$ are were constructed in subsections \ref{subs_eq_far} and \ref{subs_aux_eq}.

\medskip
As we shown in subsection \ref{subs_gluing}, this turns out to be equivalent to a system of the form (\ref{syst_q}), which can be solve using a fixed point argument in the ball
$$B_{\Lambda_0}:=\{(\q_1,\q_2)\in\mathcal{C}^{2,\beta}_{\infty,\mu}(\Sigma)\times\mathcal{D}^{2,\beta}_{\mu,\frac{1}{2}}(\Sigma):\,
\|\q_1\|_{\mathcal{C}^{2,\beta}_{\infty,\mu}(\Sigma)}<\Lambda_0\eps^{\mu},
\,\|\q_2\|_{\mathcal{D}^{2,\beta}_{\mu,\frac{1}{2}}(\Sigma)}<\Lambda_0\eps^{\mu}\},$$
provided $\Lambda_1>0$ is large enough, but indepedent of $\eps>0$.

\medskip
This is due mainly to the fact that, thanks to Proposition \ref{propinverselinearJT}, the right inverse of $\Delta_\Sigma+|A_\Sigma|^2+2\sqrt{2}a_\star\eps^{-2}e^{-\sqrt{2}{\tt v}}$ is of order $|\log\eps|^{2\emph{K}}$ and Lemma \ref{lemma_new_error}, providing the size of the error. Once again, the details are left to the reader. Similar proofs can be found in \cite{DELPINOKOWALCKZYKWEI2011}.

\section{Estimate of the energy on a ball}\label{sec_energy_ball}

In this section we will prove point (\ref{est_energy_grad}) and hence This concluding  the proof of Theorem \ref{2ndMainTheo}. Recall that the developments in Section 5 have yielded a solution $u_{\eps}=w + \varphi$ to \eqref{Allen-Cahn-eq}, where $w$ is described in \eqref{globalapproxsect5} and $\varphi$ is described in \eqref{globalsmallterm}. 

To show 
\[
\int_{B_R}\frac{1}{2} |\nabla u_\eps|^2+\frac{1}{4}(1-u_\eps^2)^2\le c R^N,
\]
we first claim  that 
\[
\int_{B_R} |\nabla u_\eps|^2\leq C R^N.
\]

\medskip
Our developments make extensive use of the coordinates $(\pe,t)$-coordinates. First we observe that the volume element in the Fermi coordinates $(\pe,{\tt z})$ of $\Sigma$ is given by
$$
\sqrt{|\det G|}=\sqrt{|\det g|}+P(\pe,{\tt z}),
$$
being $P$ a polynomial in ${\tt z}$ such that $P(\pe,0)=0$ (see \eqref{metric_FC} for the definition of $G$ and $g$).

\medskip
With the change of variables in  \eqref{eps_shift_coord}, the volume element in the $(\pe,t)$-coordinates in $X_{\eps,\h_l}^{-1}(\mathcal{N}_{\eps,\h_l})$ is given by
\begin{multline*}
\sqrt{|\det g_\eps|}+P(\eps\pe,\eps(t+\h_l(\eps\pe)))=\eps^{-(m-1)}a(s(\eps\pe))^{m-1}\eps^{-(n-1)}b(s(\eps\pe))^{n-1}\\
+P(\eps\pe,\eps(t+\h_l(\eps\pe))),
\end{multline*}
being $g_\eps$ the metric of $\Sigma_\eps=\eps^{-1}\Sigma$.

\medskip
Let $\pe$ is any point in $\Sigma_\eps\cap\partial  B_R$ and set $\s(R):=\s(\pe)$ and
$$t_\eps:=\frac{1}{4\sqrt{2}}\left(\log(s(\eps\pe)^2+2)+2|\log\eps|\right).$$

Assume that $R>2\eps^{-1}$. Then
\begin{equation}\notag
\begin{aligned}
\int_{B_R}|\nabla w|^2 d\xi&\le c\int_{0}^{\s(R)}\eps^{-(m-1)}a(\eps\s)^{m-1}\eps^{-(n-1)}b(\eps\s)^{n-1}d\s\int_{0}^{t_\eps}v'_\star(t)^2dt\\
&\le c\eps^{-(N-1)}\int_0^{\s(R)}(1+\eps\s)^{m-1}(\eps\s)^{n-1}d\s\\
&=c\eps^{-N}\int_0^{\eps\s(R)}(1+s)^{m-1}s^{n-1} ds\\
&=c\eps^{-N}\left(\int_0^{1}(1+s)^{m-1}s^{n-1} ds+\int_1^{\eps\s(R)}(1+s)^{m-1}s^{n-1} ds\right)\\
&\le c\eps^{-N}(1+(\eps \s(R))^N)\le c\eps^{-N}(1+(\eps R)^N)\le cR^N.
\end{aligned}
\end{equation}

\medskip
Setting $\tilde{\varphi}=\sum_{l=1}^2\chi^{\natural}_{3,l}\varphi_l$, 
\begin{equation}
\begin{aligned}
\int_{B_R}|\nabla\tilde{\varphi}|^2 d\xi&\le c\eps^{4+2\mu}\int_0^{\s(R)}\frac{\eps^{-(m-1)}a(\eps\s)^{m-1}\eps^{-(n-1)}b(\eps\s)^{n-1}}{((\eps \s)^2+2)^{2+\mu}}d\s \int_0^{t_\eps}e^{-2\rho t}dt\\
&\le c\eps^{4+2\mu-(N-1)}\int_0^{\s(R)}\frac{(1+\eps\s)^{m-1}(\eps\s)^{n-1}}{((\eps \s)^2+2)^{2+\mu}}d\s\\
&=c\eps^{4+2\mu-N}\int_0^{\eps\s(R)}\frac{(1+s)^{m-1}s^{n-1}}{(s^2+2)^{2+\mu}}ds\\
&=c\eps^{4+2\mu-N}\left(\int_0^{1}\frac{(1+s)^{m-1}s^{n-1}}{(s^2+2)^{2+\mu}}ds
+\int_1^{\eps\s(R)}\frac{(1+s)^{m-1}s^{n-1}}{(s^2+2)^{2+\mu}}ds\right)\\
&\le c\eps^{4+2\mu-N}(1+(\eps\s(R))^{N-4-2\mu})\\
&\le c\eps^{4+2\mu-N}(1+(\eps R)^{N-4-2\mu})\\
&\le cR^{N-4-2\mu}.
\end{aligned}
\end{equation}

Finally, we estimate the gradient of $\psi$, to find that
\begin{equation}\notag
\begin{aligned}
\int_{B_R}|\nabla\psi|^2d\xi&\le c\eps^{4+2\mu}\int_{B_R}\frac{1}{(|\eps\xi|+2)^{2+\mu}}d\xi\le\\
&=c\eps^{4+2\mu}\int_0^{\s(R)}\frac{r^N}{((\eps r)^2+2)^{2+\mu}}dr\\
&=c\eps^{4+2\mu}\int_0^{\eps\s(R)}\frac{(\eps^{-1}\rho)^N}{(\rho^2+2)^{2+\mu}}\frac{d\rho}{\eps}\\
&=c\eps^{4+2\mu-N-1}\left(\int_0^1\frac{\rho^N}{(\rho^2+2)^{2+\mu}}d\rho
+\int_1^{\eps\s(R)}\frac{\rho^N}{(\rho^2+2)^{2+\mu}}d\rho\right)\\
&\le c\eps^{4+2\mu-N-1}\left(1+\int_1^{\eps\s(R)}\rho^{N-4-2\mu}d\rho\right)\\
&\le c\eps^{4+2\mu-N-1}(1+(\eps\s(R))^{N-3-2\mu})\\
&\le c\eps^{4+2\mu-N-1}(1+(\eps R)^{N-3-2\mu})\\
& \le cR^{N-3-2\mu}.
\end{aligned}
\end{equation}

Similar estimates hold for the mixed terms and this completes the proof of the claim. To show that 
\[
\int_{B_R}(1-u_\eps^2)^2\leq CR^N
\]
a similar argument as in the proof of the above claim can be used. We leave the details to the reader. The proof of the theorem is now complete.


\begin{thebibliography}{100}
\bibitem{AGUDELODELPINOWEI2015} \textsc{O. Agudelo, M. del Pino, J. Wei.} Solutions with multiple catenoidal ends to the Allen-Cahn equation in R3. J. Math. Pures Appl. (9) 103 (2015), no. 1, 142-218.


\bibitem{AGUDELODELPINOWEI2016} \textsc{O. Agudelo, M. del Pino, J. Wei.} Higher-dimensional catenoid, Liouville equation, and Allen-Cahn equation. Int. Math. Res. Not. IMRN 2016, no. 23, 7051-7102.  



\bibitem{ALLENCAHN1979} \textsc{S. Allen and J. W. Cahn.} A microscopic theory for antiphase boundary motion and its application
to antiphase domain coarsening, Acta. Metall. 27 (1979), 1084-1095.



\bibitem{ABPRS} \textsc{H. Alencar, A. Barros, O. Palmas, G.  Reyes, W. Santos.} $O(m)\times O(n)$-invariant minimal hypersurfaces in $\R^{m+n}$. \emph{Ann. Global Anal. Geom.} {27} (2005), no. 2, 179--199. 



\bibitem{ALMGREN1966} \textsc{F.J. Almgren.} Some interior regularity theorems for minimal surfaces and an extension of Bernstein's theorem, Ann. Math. (2) 84 (1966) 277-292.

\bibitem{AMBROSIOCABRE2000} \textsc{L. Ambrosio, X. Cabr?.} Entire solutions of semilinear elliptic equations in R3 and a conjecture of De Giorgi, J. Am. Math. Soc. 13(2000) 725-739.


\bibitem{BERNSTEIN1917} \textsc{S.N. Bernstein.} Sur une theoreme degeometrie et ses applications aux equations derivees partielles du type elliptique, Commun. Soc. Math. Kharkov 15 (1915-1917) 38-45.


\bibitem{BDG_cones} \textsc{E. Bombieri, E. De Giorgi, E. Giusti.} Minimal cones and the Bernstein problem. \emph{Invent. Math.} \textbf{7} (1969), 243--268. 


\bibitem{CABRETERRA2009} \textsc{X. Cabr\'{e}, J. Terra.} Saddle-shaped solutions of bistable diffusion equations in all of R2m. J. Eur. Math. Soc. (JEMS) 11 (2009), no. 4, 819-843.

\bibitem{CABRETERRA2010}\textsc{X. Cabr\'{e}, J. Terra.} Qualitative properties of saddle-shaped solutions to bistable diffusion equations. Comm. Partial Differential Equations 35 (2010), no. 11, 1923-1957.

\bibitem{CP} \textsc{X. Cabre, G. Poggesi.} Stable solutions to some elliptic problems: minimal cones, the Allen-Cahn equation, and blow-up solutions. {Geometry of PDEs and related problems, 1--45, Lecture Notes in Math., 2220, Fond. CIME/CIME Found. Subser., Springer, Cham} 2018.


\bibitem{2018arXiv180302716C} \textsc{O. {Chodosh}, C. {Mantoulidis}.} {{Minimal surfaces and the
  Allen-Cahn equation on 3-manifolds: index, multiplicity, and curvature
  estimates}}, arXiv e-prints, to appear Annals of Mathematics (2018),
  arXiv:1803.02716.


\bibitem{CF} \textsc{M. Cozzi,A. Figalli.} Regularity theory for local and nonlocal minimal surfaces: an overview. {Nonlocal and Nonlinear Diffusions Interaction: New Methods and Directions, Lecture Notes in Mathematics, Vol. 2186, C.I.M.E. Foundation Subseries, Springer, Cham} (2017), 117--158.


\bibitem{DANCER2005} \textsc{E. N. Dancer.} Stable and finite Morse index solutions on Rn or on bounded domains with small diffusion. Trans. Amer. Math. Soc. 357 (2005), no. 3, 1225-1243.


\bibitem{D} \textsc{A. Davini.} On calibrations for Lawson's cones. \emph{Rend. Sem. Mat. Univ. Padova} {111} (2004), 55--70.


\bibitem{DEGIORGI1965} \textsc{E. De Giorgi.} Una estensione del teorema di Bernstein, Ann. Scuola Norm. Sup. Pisa (3) 19 (1965) 79-85.


\bibitem{DEGIORGI1978} \textsc{E. DeGiorgi.} Convergence problems for functionals and operators, in:Proc. Int. Meeting on Recent Methods in Nonlinear Analysis, Rome,1978, Pitagora,Bologna,1979,pp.131-188.


\bibitem{DELPINOKOWALCZYKWEIYANG2010} \textsc{M. del Pino, M. Kowalczyk, J.C. Wei, J. Yang.} Interface foliations of a positively curved manifold near a closed geodesic. Geom. Funct. Anal. 20 (2010), no. 4, 918-957.



\bibitem{DELPINOKOWALCKZYKWEI2011} \textsc{M. del Pino, M. Kowalczyk, J. Wei,} On De Giorgi conjecture in dimensions $N\geq 9$, Ann. Math.174(3) (2011) 1485-1569.



\bibitem{DELPINOKOWALCZYKWEI2013I} \textsc{M. del Pino, M. Kowalczyk, J. Wei.} Entire solutions of the Allen-Cahn equation and complete embedded minimal surfaces of finite total curvature. J. Diff. Geometry vol. 93 (2013), 67-131.

\bibitem{DELPINOKOWALCZYKWEI2013II} \textsc{M. del Pino, M. Kowalczyk, J. Wei}, Travelling waves with multiple and non-convex fronts for a bistable semilinear parabolic equation.
 Comm. Pure Appl. Math Vol. 66 no 4 (2013), 481-547.


\bibitem{DELPINOKOWALCZYKWEIPACARD2010} \textsc{M. del Pino, M. Kowalczyk, F. Pacard, J. Wei}, Multiple end solutions to the Allen-Cahn equation in $\mathbb{R}^2$. J. Funct. Analysis 258(2010), no.2, 458-503.


\bibitem{FLEMING1962} \textsc{W.H. Fleming.} On the oriented Plateau problem, Rend. Circ. Mat. Palermo. Serie II 11 (1962) 69-90.

\bibitem{GHOUSSOUBGUI1998} \textsc{N. Ghoussoub, C. Gui.} On a conjecture of De Giorgi and some related problems, Math. Ann. 311(1998) 481-491.


\bibitem{HARTMANBOOK} \textsc{P. Hartman.} Ordinary Differential Equations: An introduction. John Wiley \& Sons Inc.New York, 1964.


\bibitem{HS} \textsc{R. Hardt, L. Simon.} Area minimizing hypersurfaces with isolated singularities. \emph{J. Reine Angew. Math.} {362} (1985), 102--129.


\bibitem{HMY} \textsc{H.M. Huang, S.A.M  Marcantognini, N.Y. Young.} Chain rules for higher derivatives. \emph{Math. Intelligencer} {28} (2006), no. 2, 61--69.

\bibitem{Jerison2004} \textsc{D. Jerison and R. Monneau.} {Towards a counter-example to a
  conjecture of De Giorgi in high dimensions}, Annali di Matematica Pura ed
  Applicata {183} (2004), no.~4, 439--467.

\bibitem{kapouleas2010doubling}
\textsc{N. Kapouleas.} Doubling and desingularization constructions for minimal surfaces. Surveys in geometric analysis and relativity, 281?325,
Adv. Lect. Math. (ALM), 20, Int. Press, Somerville, MA, 2011. 

\bibitem{kapouleas2014minimal}
\bysame, Minimal surfaces in the round three-sphere by doubling the equatorial two-sphere, I. J. Differential Geom. 106 (2017), no. 3, 393-449. 


\bibitem{KAPOULEASYANG2010} \textsc{N. Kapouleas, S.D. Yang.} {Minimal surfaces in the
  three-sphere by doubling the {C}lifford torus}, Amer. J. Math. {132}
  (2010), no.~2, 257--295. \MR{2654775}

\bibitem{La} \textsc{Lawson, H.B.} The equivariant Plateau problem and interior regularity. \emph{Trans. Amer. Math. Soc.} {173} (1972), 231--249. 


\bibitem{La} \textsc{H.B. Lawson.} The equivariant Plateau problem and interior regularity. \emph{Trans. Amer. Math. Soc.} {173} (1972), 231--249. 

\bibitem{LIU2017818} \textsc{Y. Liu, K. Wang, J. Wei.} {Global minimizers of the
  {A}llen--{C}ahn equation in dimension $n\geq 8$}, Journal de Math\'ematiques
  Pures et Appliqu\'ees {108} (2017), no.~6, 818 -- 840.


\bibitem{M} 
\textsc{L. Mazet.} Minimal hypersurfaces asymptotic to Simons cones.


\bibitem{maz_pac_pol} \textsc{R. Mazzeo, F. Pacard, D. Pollack,} {Connected sums of constant
  mean curvature surfaces in {E}uclidean 3 space}, J. Reine Angew. Math.
  {536} (2001), 115--165.



\bibitem{MODICA1979} \textsc{L. Modica.} $\Gamma-$Convergence to Minimal Surfaces Problem and Global Solutions of $\Delta u=2(u3 - u)$. In Proceedings of the International Meeting on Recent Methods in Nonlinear
Analysis (Rome, 1978), 223-244. Bologna: Pitagora, 1979.

\bibitem{PW} \text{F. Pacard, J. Wei.} Stable solutions of the Allen-Cahn equation in dimension 8 and minimal cones. J. Funct. Anal. 264 (2013), no. 5, 1131-1167.

\bibitem{SAVIN2009} \textsc{O. Savin.} Regularity of at level sets in phase transitions, Ann. Math. 169(2009) 41-78.

\bibitem{Sim} \textsc{P. Simoes.} a class od minimal cones in $\R^n$, $n\ge 8$, that minimise area. Thesis (Ph.D.)--University of California, Berkeley. 1973.


\bibitem{Si} \textsc{J. Simons.} Minimal varieties in riemannian manifolds. {Ann. of Math. 2} {88} (1968), 62--105.

\bibitem{SS} \textsc{L. Simon, B. Solomon.} Minimal hypersurfaces asymptotic to quadratic cones in $\R^{n+1}$.


\bibitem{doi:10.1002/cpa.21812}
\textsc{K. Wang, J. Wei.} {Finite morse index implies finite ends}, Communications on Pure and Applied Mathematics {72} (2019), no.~5,
  1044--1119.

\end{thebibliography}
\end{document}